\documentclass[11pt,reqno]{amsart}

\usepackage{amssymb}
\usepackage{mathtools}
\usepackage{bm}
\usepackage{mathrsfs}

\usepackage[
    a4paper,
    left=30mm,
    right=30mm,
    top=30mm,
    bottom=40mm
]{geometry}
\usepackage{enumitem}
\usepackage{indentfirst}
\usepackage{array}
\usepackage{booktabs}

\usepackage{aliascnt}

\usepackage[hidelinks]{hyperref}
\usepackage[nameinlink,noabbrev]{cleveref}
\makeatletter
\renewcommand{\@defaultbiblabelstyle}[1]{[#1]}
\makeatother

\numberwithin{equation}{section}
\newlist{taggedconditions}{enumerate}{1}
\setlist[taggedconditions]{
    style=multiline,
    align=left,
    label={},
    leftmargin=4.75em,
    labelwidth=4em,
    labelsep=.75em,
    font=\normalfont,
    topsep=.4\baselineskip,
    itemsep=.25\baselineskip,
    parsep=0pt,
    partopsep=0pt
}

\theoremstyle{plain}

\newtheorem{theorem}{Theorem}[section]

\newaliascnt{proposition}{theorem}
\newtheorem{proposition}[proposition]{Proposition}
\aliascntresetthe{proposition}

\newaliascnt{lemma}{theorem}
\newtheorem{lemma}[lemma]{Lemma}
\aliascntresetthe{lemma}

\newaliascnt{corollary}{theorem}
\newtheorem{corollary}[corollary]{Corollary}
\aliascntresetthe{corollary}

\newaliascnt{conjecture}{theorem}

\aliascntresetthe{conjecture}

\theoremstyle{definition}

\newaliascnt{definition}{theorem}
\newtheorem{definition}[definition]{Definition}
\aliascntresetthe{definition}

\newaliascnt{assumption}{theorem}
\newtheorem{assumption}[assumption]{Assumption}
\aliascntresetthe{assumption}

\theoremstyle{remark}

\newaliascnt{remark}{theorem}
\newtheorem{remark}[remark]{Remark}
\aliascntresetthe{remark}

\crefname{theorem}{theorem}{theorems}
\Crefname{theorem}{Theorem}{Theorems}

\crefname{proposition}{proposition}{propositions}
\Crefname{proposition}{Proposition}{Propositions}

\crefname{lemma}{lemma}{lemmas}
\Crefname{lemma}{Lemma}{Lemmas}

\crefname{corollary}{corollary}{corollaries}
\Crefname{corollary}{Corollary}{Corollaries}

\crefname{conjecture}{conjecture}{conjectures}
\Crefname{conjecture}{Conjecture}{Conjectures}

\crefname{definition}{definition}{definitions}
\Crefname{definition}{Definition}{Definitions}

\crefname{assumption}{assumption}{assumptions}
\Crefname{assumption}{Assumption}{Assumptions}

\crefname{remark}{remark}{remarks}
\Crefname{remark}{Remark}{Remarks}

\crefname{subappendix}{appendix subsection}{appendix subsections}
\Crefname{subappendix}{Appendix subsection}{Appendix subsections}
\crefname{subsubappendix}{appendix subsection}{appendix subsections}
\Crefname{subsubappendix}{Appendix subsection}{Appendix subsections}

\newcommand{\bb}{\mathbb}
\newcommand{\mc}{\mathcal}
\newcommand{\mf}{\mathfrak}
\newcommand{\ms}{\mathscr}
\newcommand{\msf}{\mathsf}
\newcommand{\bs}{\bm}

\newcommand{\eps}{\varepsilon}
\newcommand{\lam}{\lambda}
\newcommand{\sig}{\sigma}
\newcommand{\vph}{\varphi}
\newcommand{\dd}{\,\mathrm d}

\DeclareMathOperator{\supp}{supp}
\DeclareMathOperator{\rank}{rank}
\DeclareMathOperator{\tr}{tr}
\DeclareMathOperator{\Ree}{Re}
\DeclareMathOperator{\spec}{spec}
\DeclareMathOperator{\dom}{dom}
\DeclareMathOperator{\ran}{ran}
\DeclareMathOperator{\diver}{div}
\DeclareMathOperator{\Diver}{Div}
\DeclareMathOperator*{\lowlim}{\underline{lim}}
\DeclareMathOperator*{\uplim}{\overline{lim}}

\title[Gamma Expansions for Metastable Diffusions]
{Gamma Expansions of Level Two Large Deviation Rate Functionals for Metastable Diffusions}

\hypersetup{
    pdftitle={Gamma Expansions of Level Two Large Deviation Rate Functionals for Metastable Diffusions},
    pdfauthor={Donghyun Kim and Insuk Seo}
}

\author{Donghyun Kim}
\address{
Department of Mathematical Sciences,
Seoul National University
}
\email[Donghyun Kim]{h3nry040515@snu.ac.kr}

\author{Insuk Seo}
\address{
Department of Mathematical Sciences and Research Institute of Mathematics,
Seoul National University
}
\email[Insuk Seo]{insuk.seo@snu.ac.kr}

\begin{document}

\begin{abstract}
Consider the diffusions
\[
    \dd X_t^\eps=b(X_t^\eps)\dd t+\sqrt{2\eps}\,\sig(X_t^\eps)\dd B_t
\]
on $\bb R^d$, where $b$ and $\sig$ are smooth coefficients and $B_t$ is an $r$-dimensional Brownian motion, with $r\le d$. We allow nonreversible dynamics and a degenerate diffusion matrix. Let $\mc I_\eps$ be the level-two large deviation rate functional of the empirical measure. Under our assumptions, including a complete metastable hierarchy described by resolvent convergence, we establish the full $\Gamma$-expansion
\[
    \mc I_\eps
    =\frac1\eps\mc J^{(-1)}+\mc J^{(0)}
    +\sum_{p=1}^{\mf q}\frac1{\theta_\eps^{(p)}}\mc J^{(p)}
\]
in the sense of successive $\Gamma$-limits in the weak topology on probability measures. Here $\theta_\eps^{(p)}$ are the metastable time scales, and each $\mc J^{(p)}$, $p\ge1$, is the lift of the level-two rate functional of the corresponding finite-state Markov chain. We establish the pre-metastable $\Gamma$-limits for generalized overdamped, underdamped, and third-order Langevin dynamics. For the generalized overdamped model, existing metastability results yield the full expansion under additional conditions. For the underdamped and third-order models, the full expansion remains conditional on verification of the complete metastable hierarchy.
\end{abstract}

\maketitle

\section{Introduction}

A diffusion with small noise can spend a long time near one stable equilibrium before moving to another. Different transitions may become visible on widely separated time scales. The empirical measure provides a common description of these regimes: it records the fraction of time spent in each region, while its large deviation functional measures the cost of an atypical occupation profile. Our aim is to identify the successive orders of this cost as the noise vanishes, using a resolvent description of the metastable dynamics.

We consider the diffusions
\begin{equation} \label{eq:intro-sde}
    \dd X_t^\eps
    =
    b(X_t^\eps)\dd t
    +
    \sqrt{2\eps}\,\sig(X_t^\eps)\dd B_t
\end{equation}
on $\mc X=\bb R^d$, where $B_t$ is an $r$-dimensional Brownian motion, $r\le d$, and $b$ and $\sig$ are smooth. The diffusion matrix $a=\sig\sig^{\msf T}$ may be degenerate, and we do not assume reversibility. Under our hypotheses, for each sufficiently small $\eps>0$, the empirical measures
\[
    \eta_t^\eps
    =\frac1t\int_0^t\delta_{X_s^\eps}\dd s
\]
satisfy a level-two large deviation principle as $t\to\infty$, with speed $t$ and a convex good rate functional $\mc I_\eps$ on $\mc P(\mc X)$. Heuristically, for fixed sufficiently small $\eps$ and large $t$,
\[
    \bb P_x^\eps\bigl(\eta_t^\eps\approx\nu\bigr)
    \approx\exp\bigl\{-t\,\mc I_\eps(\nu)\bigr\}.
\]
Thus $\mc I_\eps(\nu)$ measures the cost of spending time according to the occupation profile $\nu$. We study the small-noise limit of these functionals in the weak topology on probability measures.

Under the hypotheses stated in \cref{sec:model}, including a complete metastable hierarchy, our main result, \cref{thm:main}, gives the full $\Gamma$-expansion
\begin{equation} \label{eq:intro-full-expansion}
    \mc I_\eps
    =
    \frac1\eps\mc J^{(-1)}
    +
    \mc J^{(0)}
    +
    \sum_{p=1}^{\mf q}
    \frac1{\theta_\eps^{(p)}}\mc J^{(p)},
\end{equation}
with successive scales
\[
    \eps\prec1\prec\theta_\eps^{(1)}\prec\cdots\prec\theta_\eps^{(\mf q)},
\]
where $\alpha_\eps\prec\beta_\eps$ means $\alpha_\eps/\beta_\eps\to0$. The expansion denotes successive $\Gamma$-limits: $\eps\mc I_\eps$ converges to $\mc J^{(-1)}$, $\mc I_\eps$ to $\mc J^{(0)}$, and $\theta_\eps^{(p)}\mc I_\eps$ to $\mc J^{(p)}$. The precise definition is given in \cref{def:gamma-expansion}.

The first two scales are pre-metastable. The functional $\mc J^{(-1)}$ penalizes occupation away from equilibria. At the unit scale, $\mc J^{(0)}$ is finite only on measures supported on equilibria and records the cost of concentration near unstable equilibria. Its zero set consists of measures supported on stable equilibria. At each later scale, $\mc J^{(p)}$ is the lift of the Donsker--Varadhan functional of the finite-state chain describing transitions between the corresponding metastable groups. Each new limit therefore distinguishes occupation profiles that have zero cost at the preceding scale.

Full $\Gamma$-expansions have been established by Landim \cite{Landim2023FiniteState} for nonreversible finite-state Markov chains, by Landim, Lee and Mariani \cite{LandimLeeMariani2025} for reversible overdamped Langevin diffusions with Morse potentials, and by Landim and Maura \cite{LandimMaura2026} for nonreversible one-dimensional diffusions with periodic boundary conditions. In the reversible diffusion setting, the recovery measures in \cite{LandimLeeMariani2025} are built from approximate equilibrium potentials and the reversible Dirichlet form. We develop a general transfer argument and the analytic framework needed to apply it to nonreversible diffusions with possibly degenerate noise.

The paper has four main contributions.

\emph{First, we give a general argument transferring a complete metastable hierarchy into a full $\Gamma$-expansion, including when the effective chains are reducible.} The argument uses resolvent convergence, the Donsker--Varadhan variational formula, and contact measures for positive resolvent tests. It applies to general Markov processes, including Markov jump processes. In the resolvent approach of \cite{LandimMarcondesSeo2025}, the rescaled resolvent, applied to observables that are constant on the wells, converges there to the resolvent of a finite-state chain. A complete hierarchy supplies this description at successive scales, until only one closed communication class remains. \Cref{thm:resolvent-transfer} states the transfer theorem in our diffusion setting. Starting from the pre-metastable limits, we apply it successively along the hierarchy to obtain the full expansion.

Choi \cite[Theorem~3.2 and Section~4.1]{Choi2024ZeroRange} established a resolvent-based transfer at a single metastable scale for discrete-state systems. That recovery argument uses irreducibility of the limiting chain to identify the limiting measure through the uniqueness of the invariant law of a tilted chain. Along a hierarchy, the intermediate chains may be reducible. Stationarity then does not determine the weights of distinct closed classes, and measures charging transient classes must also be treated. We construct recovery measures separately on each communication class, retaining the cost of leaving a transient class. Combining these measures with the prescribed weights and using convexity gives the recovery bound for the full chain.

\emph{Second, we develop the analytic framework needed to apply this transfer argument to diffusions on an unbounded state space with possibly degenerate noise.} We establish the level-two large deviation principle, identify its rate functional through the Donsker--Varadhan formula, and justify the use of positive smooth test functions in this formula. Schematically,
\[
    \mc I_\eps(\nu)
    =\sup_{u>0}\int_{\mc X}-\frac{\mc L_\eps u}{u}\dd\nu,
    \qquad
    \mc L_\eps=b\cdot\nabla+\eps a:\nabla^2,
\]
where the admissible class of positive smooth tests is specified in \cref{prop:smooth-core-eps}. For a suitable positive resolvent test $F$, we construct the tilted Markov semigroup and prove that it admits an invariant probability measure $\nu_{\eps,F}$ satisfying
\[
    \mc I_\eps(\nu_{\eps,F})
    =\int_{\mc X}-\frac{\ms L_\eps F}{F}\dd\nu_{\eps,F}.
\]
Here $\ms L_\eps$ is the extended generator, defined in \cref{ssec:extended-generator-eps}. We call $\nu_{\eps,F}$ a \emph{contact measure} for $F$: the test $F$ attains the Donsker--Varadhan supremum at this measure. These results supply the analytic ingredients required by the general transfer argument, including the measures used to construct recovery sequences. The use of stationary measures of tilted dynamics follows \cite{Landim2023FiniteState,LandimMaura2026}; the construction and its relation to these earlier recovery arguments are explained in \cref{ssec:contact-measures-eps}.

\emph{Third, we identify the two pre-metastable limits without assuming a metastable hierarchy.} The assumptions in \cref{ass:algebraic-structure,ass:cocentration} concern the drift and noise and provide a Lyapunov function controlling the process at infinity and near equilibria. The requirements in \cref{ass:algebraic-structure} are automatic when the diffusion matrix is positive definite everywhere. Under these assumptions, \cref{prop:pre-metastable-gamma,prop:unit-gamma} establish the limits at scales $\eps$ and $1$.

\emph{Fourth, we apply the framework to generalized overdamped, underdamped, and third-order Langevin dynamics.} Under the respective model assumptions, we establish the pre-metastable limits for all three. The generalized overdamped model is
\[
    \dd X_t^\eps
    =-\bigl(\nabla U+\ell\bigr)(X_t^\eps)\dd t
    +\sqrt{2\eps}\,\dd B_t,
    \qquad
    \ell\cdot\nabla U=0.
\]
For this model, the condition $\diver\ell=0$, together with additional assumptions on the landscape of $U$, allows us to apply the resolvent results of \cite{LandimLeeSeo2024First,LandimLeeSeo2025General}, giving the full expansion. The underdamped and third-order models illustrate the treatment of degenerate noise; their full expansions remain conditional on the existence of a complete metastable hierarchy. The examples are presented in \cref{sec:applications}.

The transfer argument does not use an explicit formula for the invariant measure. Identifying the metastable scales and effective chains remains a separate, model-dependent problem. The available results used here to verify nontrivial hierarchies concern diffusions with explicitly known stationary measures. Establishing the hierarchy beyond this setting remains a central open problem; we discuss this limitation at the end of \cref{sec:model}.

The paper is organized as follows. \Cref{sec:model,sec:main-result} introduce the assumptions and the hierarchy, define the limiting functionals, and state the main results; \cref{sec:applications} presents the Langevin examples. \Cref{sec:DV-framework,sec:contact-resolvent} develop the large deviation framework and the contact measures used in \cref{sec:metastable-scales} to prove the transfer theorem. \Cref{sec:scale-eps,sec:unit-scale} establish the pre-metastable limits. The appendices collect the finite-state and control-theoretic facts and verify the assumptions for the examples.

\section{Model and metastable hierarchy} \label{sec:model}

\subsection{Low-temperature diffusions}

Let the state space be $\mc X=\bb R^d$. Let $B_t$ be a standard Brownian motion on $\bb R^r$, where $r\le d$. Denote by $B_b(\mc X)$ the space of bounded measurable functions on $\mc X$. For each temperature $\eps>0$, consider the diffusion
\begin{equation} \label{eq:sde-eps}
    \dd X_t^\eps
    =
    b(X_t^\eps)\dd t
    +
    \sqrt{2\eps}\,\sig(X_t^\eps)\dd B_t,
\end{equation}
where
\[
    b\in C^\infty(\mc X;\bb R^d),
    \qquad
    \sig\in C^\infty(\mc X;\bb R^{d\times r}).
\]
Define the diffusion matrix $a\coloneqq\sig\sig^{\msf T}$ and the formal generator
\[
    \mc L_\eps
    \coloneqq
    b\cdot\nabla+\eps a:\nabla^2,
\]
where $\msf A : \msf B = \tr(\msf A^{\msf T}\msf B)$.
Define the equilibrium set of the deterministic flow by
\[
    \mc C_0
    \coloneqq
    \{x\in\mc X:b(x)=0\}.
\]
\begin{assumption}[Algebraic structure]\label{ass:algebraic-structure}
    Assume the following.
    \begin{enumerate}
        \item \emph{Fixed noise range.}
            There exists a linear subspace $\mf H\le\bb R^d$ such that $\dim\mf H=r$ and $\ran\sig(x)=\mf H$ for every $x\in\mc X$.
        \item \emph{Affine nondiffusive drift.}
            There exists a constant matrix $\msf\Omega\in\bb R^{d\times d}$ such that
            \[
                b(x)-b(0)-\msf\Omega x\in\mf H,
                \qquad
                x\in\mc X.
            \]
        \item \emph{Controllability.}
            For some $x_*\in\mc X$, the pair $(\msf\Omega,\sig(x_*))$ is controllable:
            \[
                \rank
                \begin{pmatrix}
                    \sig(x_*) & \msf\Omega\sig(x_*) & \cdots & \msf\Omega^{d-1}\sig(x_*)
                \end{pmatrix}
                =d.
            \]
    \end{enumerate}
\end{assumption}

\Cref{ass:algebraic-structure} imposes no additional restriction when $a(x)$ is positive definite for every $x\in\mc X$: in that case $r=d$, $\mf H=\bb R^d$, and $\sig(x)$ is invertible. These conditions become genuine restrictions when $a$ is degenerate.

In the degenerate case, items \textup{(1)} and \textup{(2)} have the following concrete interpretation. After an orthogonal change of coordinates associated with $\bb R^d=\mf H^\perp\oplus\mf H$, write $x=(y,z)\in\bb R^{d-r}\times\bb R^r$. Then \eqref{eq:sde-eps} takes the form $X_t^\eps=(Y_t^\eps,Z_t^\eps)$ with
\begin{equation}\label{eq:qp-coordinates}
\begin{aligned}
    \dd Y_t^\eps
    &=
    b_y(Y_t^\eps,Z_t^\eps)\dd t,\\
    \dd Z_t^\eps
    &=
    b_z(Y_t^\eps,Z_t^\eps)\dd t
    +
    \sqrt{2\eps}\,
    \sig_z(Y_t^\eps,Z_t^\eps)\dd B_t,
\end{aligned}
\end{equation}
where $b_y(y,z)$ is affine and $\sig_z(y,z)\in\bb R^{r\times r}$ is invertible for every $(y,z)\in\mc X$. Thus the noise acts only in the $z$-variables, while the drift in the nondiffusive variables is affine. In \cref{ssec:third-order}, we demonstrate how a nonlinear change of variables can reveal this structure even when $b_y(y,z)$ is non-affine.

\Cref{ass:algebraic-structure} implies that $(Db(x),\sig(x))$ is controllable at every $x$. Indeed, differentiating item \textup{(2)} gives
\[
    \ran\bigl(Db(x)-\msf\Omega\bigr)
    \subseteq
    \mf H
    =
    \ran\sig(x).
\]
Thus $Db(x)$ differs from $\msf\Omega$ by state feedback, which preserves controllability by \cref{lem:feedback-invariance}. Since the full-rank matrices $\sig(x)$ have the same range, controllability of $(\msf\Omega,\sig(x))$ is independent of $x$.

Under item \textup{(1)}, $\sig(x)$ has full column rank. Throughout the paper, we use its smooth Moore--Penrose pseudoinverse
\[
    \msf R(x)
    \coloneqq
    \bigl(\sig(x)^{\msf T}\sig(x)\bigr)^{-1}\sig(x)^{\msf T}.
\]
Then $\msf R(x)\sig(x)=\msf{Id}_r$, while $\sig(x)\msf R(x)$ is the orthogonal projection onto $\mf H$.

For $x\in\mc X$ and $p\in\bb R^d$, define the Freidlin--Wentzell Hamiltonian
\[
    \ms H(x,p)
    \coloneqq
    b(x)\cdot p+|\sig(x)^{\msf T}p|^2.
\]

Throughout the paper, $C$ denotes a finite positive constant whose value may change from line to line. Dependence on parameters is indicated whenever relevant. The next assumption combines Lyapunov control at infinity with conditions on the equilibria of the deterministic flow.
\begin{assumption}[Concentration] \label{ass:cocentration}
    There exists $V\in C^\infty(\mc X;[0,\infty))$ such that, with
    \[
        W(x)
        \coloneqq
        -\ms H(x,\nabla V(x)),
    \]
    the following conditions hold.
    \begin{enumerate}
        \item \emph{Confinement.}
            The functions $V$ and $W$ are inf-compact, meaning that all their sublevel sets are compact.
        \item \emph{Lyapunov bound.}
            There exists $C<\infty$ such that
            \[
                a:\nabla^2V
                \le
                C(1+W)
                \qquad\text{on }\mc X.
            \]
        \item \emph{Strict dissipation.}
            We have
            \[
                W\ge0,
                \qquad
                W^{-1}(\{0\})=\mc C_0.
            \]
        \item \emph{Hyperbolic equilibria.}
            Every $c\in\mc C_0$ is hyperbolic:
            \[
                \spec(Db(c))\cap i\bb R=\emptyset.
            \]
    \end{enumerate}
\end{assumption}

\begin{remark}
    Consider the overdamped Langevin dynamics
    \[
        \dd X_t^\eps
        =-\nabla U(X_t^\eps)\dd t
        +\sqrt{2\eps}\,\dd B_t,
    \]
    where $U\in C^\infty(\bb R^d)$. Here $a=\sig=\msf{Id}_d$, so \cref{ass:algebraic-structure} is automatic. After normalizing $\min U=0$, the choice $V=U/2$ gives
    \[
        W=-\ms H(\,\cdot\,,\nabla V)=\frac14|\nabla U|^2,
        \qquad
        a:\nabla^2V=\frac12\Delta U.
    \]
    Thus \cref{ass:cocentration} is readily verified under growth and nondegeneracy conditions on $U$. We verify \cref{ass:algebraic-structure,ass:cocentration} for the three Langevin models in \cref{sec:applications}.
\end{remark}

\begin{remark}
    The main results extend, with the usual modifications, to the periodic state space $\mc X=\bb T^d$ when $a$ is uniformly elliptic. In this case $r=d$, and \cref{ass:algebraic-structure} may be replaced by uniform ellipticity. Compactness makes \cref{ass:cocentration}\textup{(1)} automatic, and item \textup{(2)} follows from boundedness of $a:\nabla^2V$ when $W\ge0$. Thus the geometric assumptions reduce to items \textup{(3)}--\textup{(4)}.

    The results also extend to the partially periodic state space $\mc X=\bb T^{d-r}\times\bb R^r$. In the adapted coordinates $(y,z)$, retain item \textup{(1)} of \cref{ass:algebraic-structure} with $\mf H=\{0\}^{d-r}\times\bb R^r$ and replace items \textup{(2)}--\textup{(3)} by the existence of a matrix $\msf\Omega'\in\bb R^{d\times r}$ such that
    \[
        b(y,z)-b(0,0)-\msf\Omega'z\in\mf H.
    \]
    This includes the underdamped Langevin dynamics with periodic position variable when $d=2n$ and $r=n$. All functions and test functions are then understood to be periodic in the $y$-variables.
\end{remark}

\subsection{The level-2 large deviation principle}
Suppose that \cref{ass:algebraic-structure,ass:cocentration} hold. By \cref{prop:diffusion-properties-eps}, for every sufficiently small $\eps>0$, the diffusion defines a non-explosive strong Markov family $(X_t^\eps,\bb P_x^\eps)$ with continuous paths.

Moreover, the empirical measures
\[
    \eta_t^\eps
    \coloneqq
    \frac1t\int_0^t\delta_{X_s^\eps}\dd s
\]
satisfy, with speed $t$, the level-2 large deviation principle (LDP) on $\mc P(\mc X)$ equipped with the weak topology, locally uniformly with respect to the initial point. More precisely, there exists a convex good rate functional $\mc I_\eps$ such that, for every nonempty compact set $K\subset\mc X$ and every Borel set $A\subset\mc P(\mc X)$,
\[
    \begin{aligned}
        -\inf_{\nu\in A^\circ}\mc I_\eps(\nu)
        &\le
        \lowlim_{t\to\infty}\,
        \inf_{x\in K}
        \frac1t\log\bb P_x^\eps(\eta_t^\eps\in A)\\
        &\le
        \uplim_{t\to\infty}\,
        \sup_{x\in K}
        \frac1t\log\bb P_x^\eps(\eta_t^\eps\in A)
        \le
        -\inf_{\nu\in\overline A}\mc I_\eps(\nu),
    \end{aligned}
\]
where $A^\circ$ and $\overline A$ denote the interior and closure of $A$ in the weak topology.
The LDP is proved in \cref{thm:LDP-eps}.

\subsection{Metastable hierarchy} \label{ssec:metastable-hierarchy}

Define the set of stable equilibria of the system $\dot z=b(z)$ by
\[
    \mc M_0
    \coloneqq
    \big\{
    m\in\mc C_0:
    \spec(Db(m))
    \subset
    \{z\in\bb C:\Ree z<0\}
    \big\}.
\]
\begin{lemma} \label{lem:C0-finite}
    Under \cref{ass:cocentration}, the set $\mc C_0$ is nonempty and finite, and $\mc M_0$ is nonempty.
\end{lemma}
\begin{proof}
    Since $V$ is inf-compact, it attains its minimum at some $c\in\mc X$. Then $\nabla V(c)=0$, so the definition of $W$ gives $W(c)=0$, and hence $c\in\mc C_0$. Moreover, $\mc C_0=\{W\le0\}$ is compact. Hyperbolicity makes $Db(c)$ invertible at each equilibrium, so the inverse function theorem shows that every equilibrium is isolated. Thus $\mc C_0$ is finite.

    Every global minimizer of $V$ belongs to $\mc C_0$, so $c$ is an isolated global minimizer. In a sufficiently small punctured neighborhood of $c$,
    \[
        V>V(c),
        \qquad
        b\cdot\nabla V
        =
        -W-|\sig^{\msf T}\nabla V|^2
        <0,
    \]
    so $c$ is asymptotically stable. Since $c$ is hyperbolic, $Db(c)$ is Hurwitz, meaning that all its eigenvalues have negative real parts. Hence $c\in\mc M_0$.
\end{proof}

\Cref{ass:algebraic-structure,ass:cocentration} provide the probabilistic framework. To describe the successive metastable scales, we use the resolvent approach of \cite{LandimMarcondesSeo2025}, in which the motion between neighborhoods of stable equilibria is described by finite-state Markov chains.

Write $\bb E_x^\eps$ for expectation under $\bb P_x^\eps$ and, for $\lam>0$ and $G\in B_b(\mc X)$, define the resolvent
\[
    R_\lam^\eps G(x)
    \coloneqq
    \bb E_x^\eps
    \left[
    \int_0^\infty e^{-\lam t}G(X_t^\eps)\dd t
    \right].
\]
Then $\theta^{-1}R_{\lam/\theta}^\eps G$ is the resolvent of the time-rescaled process $(X_{\theta t}^\eps)_{t\ge0}$.

For each $m\in\mc M_0$, fix a model-dependent relatively compact Borel neighborhood $\mc E(m)$ whose compact closure lies in the basin of attraction of $m$, such that the closures of these neighborhoods are pairwise disjoint and each boundary $\partial\mc E(m)$ has zero Lebesgue measure. For $\mc M\subset\mc M_0$, define
\[
    \mc E(\mc M)
    \coloneqq
    \bigcup_{m\in\mc M}\mc E(m).
\]
Then $\mc E(\mc M)\cap\mc M_0=\mc M$, and its boundary is disjoint from $\mc M_0$ and has zero Lebesgue measure. Hence $\mc E(\mc M)$ is a continuity set both for the Lebesgue measure and for every measure supported on $\mc M_0$, in particular for all lifted invariant measures introduced below.

For two positive scales $\alpha_\eps$ and $\beta_\eps$, write
\[
    \alpha_\eps\prec\beta_\eps
    \quad\Longleftrightarrow\quad
    \lim_{\eps\to0}
    \frac{\alpha_\eps}{\beta_\eps}
    =
    0.
\]

If $\mc M_0=\{m\}$, the metastable hierarchy is trivial. In this case, set $\mf q=0$,
\[
    \ms V^{(1)}
    \coloneqq
    \bigl\{\{m\}\bigr\},
    \qquad
    \pi_{\{m\}}
    \coloneqq
    \delta_m,
\]
and introduce no metastable time scale or reduced generator. By convention, this is a complete metastable hierarchy. Henceforth, suppose that $|\mc M_0|\ge2$.

We recursively construct the metastable tree structure from triples
\[
    \left(
    \theta_\eps^{(p)},
    \ms V^{(p)},
    \mf L^{(p)}
    \right)_{1\le p \le \mf q}
\]
at successive levels. Here $\theta_\eps^{(p)}$ is a metastable time scale, $\ms V^{(p)}$ is a finite collection of pairwise disjoint subsets of $\mc M_0$, and $\mf L^{(p)}$ is a nonzero Markov generator on $\ms V^{(p)}$. We denote by $\bs y^{(p)}(\cdot)$ the continuous-time Markov chain generated by $\mf L^{(p)}$. The terminal level $\mf q$ is specified in Step 3.

Along with these triples, we recursively define a probability measure $\pi_{\mc M}\in\mc P(\mc M_0)$ for every metastable state $\mc M\in\ms V^{(p)}$ appearing in the construction. Each $\pi_{\mc M}$ depends only on the set $\mc M\subset\mc M_0$ and is independent of the levels at which $\mc M$ appears.

For a finite-state Markov chain, two states communicate if each is reachable from the other. A communication class is a maximal communicating set, and it is closed if the chain has no transition from the class to its complement. At each level below, the closed communication classes of the limiting chain become the states at the next level, while transient states disappear. The following three steps formalize this recursive coarse-graining.

\medskip
\noindent{\textbf{Step 1.}} Initially, define
\[
    \theta_\eps^{(0)}
    \equiv
    1,
    \qquad
    n_0
    \coloneqq
    |\mc M_0|,
    \qquad
    \ms V^{(1)}
    \coloneqq
    \bigl\{\{m\}:m\in\mc M_0\bigr\}.
\]
For every $m\in\mc M_0$, define $\pi_{\{m\}}\coloneqq\delta_m$.

\medskip
\noindent{\textbf{Step 2.}} Assume that $\theta_\eps^{(p-1)}$ and $\ms V^{(p)}$ have been defined, with $|\ms V^{(p)}|=n_{p-1}\ge2$. Assume that there exist a scale $\theta_\eps^{(p)}$ and a nonzero Markov generator $\mf L^{(p)}$ on $\ms V^{(p)}$ such that $\theta_\eps^{(p-1)}\prec\theta_\eps^{(p)}$. Assume moreover that, for every $g:\ms V^{(p)}\to\bb R$ and $\lam>0$, if
\[
    G
    \coloneqq
    \sum_{\mc N\in\ms V^{(p)}}
    g(\mc N)\bs1_{\mc E(\mc N)},
    \qquad
    F_\eps
    \coloneqq
    \frac1{\theta_\eps^{(p)}}
    R_{\lam/\theta_\eps^{(p)}}^\eps G,
\]
then, writing $f\coloneqq(\lam-\mf L^{(p)})^{-1}g$,
\begin{equation} \label{eq:metastable-resolvent-convergence}
    \lim_{\eps\to0}
    \max_{\mc M\in\ms V^{(p)}}
    \sup_{x\in\mc E(\mc M)}
    |F_\eps(x)-f(\mc M)|
    =0.
\end{equation}

Denote the closed communication classes of $\bs y^{(p)}(\cdot)$, in the sense of \cref{def:finite-communication}, by
\[
    \mf C_1^{(p)},
    \ldots,
    \mf C_{n_p}^{(p)}.
\]
Since $\mf L^{(p)}$ is nonzero, $n_p<n_{p-1}$. The transient states of $\bs y^{(p)}(\cdot)$ are discarded and will not be considered at subsequent time scales. For every $i\in\{1,\ldots,n_p\}$, set
\[
    \mc M_i^{(p+1)}
    \coloneqq
    \bigcup_{\mc M\in\mf C_i^{(p)}}\mc M,
\]
and define
\[
    \ms V^{(p+1)}
    \coloneqq
    \left\{
    \mc M_1^{(p+1)},
    \ldots,
    \mc M_{n_p}^{(p+1)}
    \right\}.
\]

For each $i\in\{1,\ldots,n_p\}$, let $\omega_i^{(p)}\in\mc P(\ms V^{(p)})$ be the invariant probability measure of $\bs y^{(p)}(\cdot)$ supported on $\mf C_i^{(p)}$. By \cref{lem:finite-invariant-measures},
\[
    \operatorname{co}
    \left\{
    \omega_1^{(p)},
    \ldots,
    \omega_{n_p}^{(p)}
    \right\}
\]
is the complete set of invariant probability measures of $\bs y^{(p)}(\cdot)$.

For every $i\in\{1,\ldots,n_p\}$, define
\begin{equation}\label{eq:piM-def}
    \pi_{\mc M_i^{(p+1)}}
    \coloneqq
    \sum_{\mc M\in\mf C_i^{(p)}}
    \omega_i^{(p)}(\mc M)\pi_{\mc M}.
\end{equation}
Since \cref{lem:finite-invariant-measures} gives $\omega_i^{(p)}>0$ on $\mf C_i^{(p)}$, induction yields
\[
    \supp\pi_{\mc M_i^{(p+1)}}
    =
    \mc M_i^{(p+1)}.
\]
The definition \eqref{eq:piM-def} is consistent when a state persists across levels. Indeed, if $\mf C_i^{(p)}=\{\mc M\}$, then $\mc M_i^{(p+1)}=\mc M$, $\omega_i^{(p)}=\delta_{\mc M}$, and the formula leaves $\pi_{\mc M}$ unchanged. Otherwise, every $\mc M\in\mf C_i^{(p)}$ is a proper subset of $\mc M_i^{(p+1)}$, so \eqref{eq:piM-def} defines a new measure.

\medskip
\noindent{\textbf{Step 3.}} The recursion is complete if there exists a finite $\mf q\ge1$ such that Step 2 can be carried out for every $p\in\{1,\ldots,\mf q\}$, with
\[
    n_p\ge2
    \quad\text{for }p<\mf q,
    \qquad
    n_{\mf q}=1.
\]
Since $n_p<n_{p-1}$ at every step, such a recursion is necessarily finite. No generator is introduced beyond the terminal level, whose active state space is
\[
    \ms V^{(\mf q+1)}
    =
    \left\{
    \mc M_1^{(\mf q+1)}
    \right\}.
\]

\begin{assumption}[Complete metastable hierarchy]
    \label{ass:complete-metastable-hierarchy}
    The model \eqref{eq:sde-eps} admits a \emph{complete metastable hierarchy}. This means that either $|\mc M_0|=1$, in which case $\mf q=0$, or there exist an integer $\mf q\ge1$ and a recursive construction satisfying Steps 1--3 above.
\end{assumption}

\Cref{sec:applications} presents examples illustrating these hypotheses. In particular, the generalized overdamped Langevin diffusions satisfying \textup{\textbf{(OD1)}}--\textup{\textbf{(OD3)}} and \textup{\textbf{(OD-M)}} satisfy all the hypotheses. For the underdamped and third-order Langevin models, we establish the pre-metastable conclusions; the full expansion remains conditional on the existence of a complete metastable hierarchy.

Our argument does not directly require an explicit formula for the invariant measure. However, the results invoked here to verify \cref{ass:complete-metastable-hierarchy} for nontrivial metastable hierarchies concern diffusions whose invariant measures are explicitly known. Establishing such hierarchies when the invariant measure is not explicitly known remains a central open problem in metastability; see also the discussion in \cite[Section~1]{LandimMaura2026}.

\section{Full \texorpdfstring{$\Gamma$}{Gamma}-expansion} \label{sec:main-result}

\subsection{\texorpdfstring{$\Gamma$-convergence and full expansions}{Gamma-convergence and full expansions}}

We first recall the definition of $\Gamma$-convergence and then formulate the notion of a full $\Gamma$-expansion used in this paper; see \cite{Mariani2018} for discussion about $\Gamma$-convergence of LDP rate functionals.
\begin{definition}[$\Gamma$-convergence]
    Let $X$ be a Polish space. A family of functionals $I_\eps:X\to[0,\infty]$ $\Gamma$-converges to $I:X\to[0,\infty]$ as $\eps\to0$ if the following conditions hold.
    \begin{enumerate}
        \item
            For every $x\in X$ and every family $x_\eps\to x$,
            \[
                \lowlim_{\eps\to0}I_\eps(x_\eps)\ge I(x).
            \]
        \item
            For every $x\in X$, there exists a family $x_\eps\to x$ such that
            \[
                \uplim_{\eps\to0}I_\eps(x_\eps)\le I(x).
            \]
    \end{enumerate}
\end{definition}

The $\Gamma$-convergences below are taken on the set of Borel probability measures $\mc P(\mc X)$ equipped with the weak topology, which is a Polish space.
\begin{definition}[Full $\Gamma$-expansion] \label{def:gamma-expansion}
    Let $\mf r,\mf q\ge0$ be integers. A \emph{full $\Gamma$-expansion} of $(\mc I_\eps)_{\eps>0}$ is given by positive scales $\theta_\eps^{(p)}$ and functionals $\mc J^{(p)}:\mc P(\mc X)\to[0,\infty]$, indexed by the integers $p$ satisfying $-\mf r\le p\le\mf q$, if the following conditions hold.
    \begin{enumerate}
        \item
            Consecutive scales satisfy
            \[
                \theta_\eps^{(p)}\prec\theta_\eps^{(p+1)},
                \qquad -\mf r\le p<\mf q.
            \]
        \item
            If $\varrho_\eps\prec\theta_\eps^{(-\mf r)}$, then $\varrho_\eps\mc I_\eps$ $\Gamma$-converges to the zero functional.
        \item
            For every integer $p$ with $-\mf r\le p\le\mf q$, $\theta_\eps^{(p)}\mc I_\eps$ $\Gamma$-converges to $\mc J^{(p)}$.
        \item
            For every integer $p$ with $-\mf r\le p<\mf q$ and every $\nu\in\mc P(\mc X)$,
            \[
                \mc J^{(p)}(\nu)=0
                \quad\Longleftrightarrow\quad
                \mc J^{(p+1)}(\nu)<\infty.
            \]
        \item
            The $0$-level set of $\mc J^{(\mf q)}$ is a singleton.
    \end{enumerate}
    We call the scales with $p\ge1$ \emph{metastable} and those with $p\le0$ \emph{pre-metastable}. We write this expansion as
    \[
        \mc I_\eps
        =
        \sum_{p=-\mf r}^{\mf q}
        \frac1{\theta_\eps^{(p)}}\mc J^{(p)}.
    \]
\end{definition}

\subsection{Scales and limiting functionals} \label{ssec:limit-define}

Set $\theta_\eps^{(-1)}\coloneqq\eps$ and $\theta_\eps^{(0)}\coloneqq1$. For $p\in\{1,\ldots,\mf q\}$, let $\theta_\eps^{(p)}$ be the metastable scales in \cref{ass:complete-metastable-hierarchy}.

At the scale $\eps$, define
\begin{equation}\label{eq:def-J-minus}
\begin{split}
    \mc J^{(-1)}(\nu)
    &\coloneqq
    -\inf_{\vph\in C_c^\infty(\mc X)}
    \int_{\mc X}
    \ms H(x,\nabla\vph(x))
    \dd\nu\\
    &=
    \sup_{\vph\in C_c^\infty(\mc X)}
    \int_{\mc X}
    \left[
        -b\cdot\nabla\vph
        -
        |\sig^{\msf T}\nabla\vph|^2
    \right]
    \dd\nu.
\end{split}
\end{equation}
By \cref{lem:J-Riesz}, its effective domain is
\[
    \dom\mc J^{(-1)}
    =
    \left\{
    \nu\in\mc P(\mc X):
    \begin{array}{l}
    \text{there exists }j\in L^2(\nu;\bb R^r)\text{ such that}\\
    \diver\bigl((b-\sig j)\nu\bigr)=0
    \text{ in }\mc D'(\mc X)
    \end{array}
    \right\}.
\]
All divergence equalities are understood in the sense of distributions.
Under \cref{ass:algebraic-structure,ass:cocentration}, \cref{cor:J-minus-zeroset} gives
\begin{equation} \label{eq:J-minus-zero}
    \{\nu\in\mc P(\mc X):\mc J^{(-1)}(\nu)=0\}
    =
    \mc P(\mc C_0).
\end{equation}

At the unit scale, define
\[
    \kappa_c
    \coloneqq
    \sum_{\substack{\lam\in\spec(Db(c))\\\Ree\lam>0}}
    \Ree\lam,
    \qquad
    c\in\mc C_0,
\]
where the eigenvalues are counted with algebraic multiplicity, and set
\begin{equation}\label{eq:def-J-zero}
\mc J^{(0)}(\nu)
\coloneqq
\begin{cases}
    \displaystyle
    \sum_{c\in\mc C_0}\omega_c\kappa_c,
    & \displaystyle
    \nu=\sum_{c\in\mc C_0}\omega_c\delta_c
    \in\mc P(\mc C_0),\\
    +\infty,
    & \nu\notin\mc P(\mc C_0).
\end{cases}
\end{equation}
Thus
\[
    \dom\mc J^{(0)}=\mc P(\mc C_0).
\]
Under \cref{ass:algebraic-structure,ass:cocentration}, hyperbolicity implies that an equilibrium is stable exactly when it has no eigenvalue with positive real part. Hence
\begin{equation}\label{eq:J-zero-zero}
    \{\nu\in\mc P(\mc X):\mc J^{(0)}(\nu)=0\}
    =
    \mc P(\mc M_0).
\end{equation}

For $p\in\{1,\ldots,\mf q\}$, let $\mf J^{(p)}$ be the Donsker--Varadhan functional of the finite-state Markov chain $\bs y^{(p)}(\cdot)$:
\begin{equation} \label{eq:finite-Jp}
    \mf J^{(p)}(\omega)
    \coloneqq
    \sup_{h>0}
    \sum_{\mc M\in\ms V^{(p)}}
    -\frac{\mf L^{(p)}h(\mc M)}{h(\mc M)}
    \omega(\mc M),
    \qquad
    \omega\in\mc P(\ms V^{(p)}).
\end{equation}
Thus $\mf J^{(p)}=\mf J_{\mf L^{(p)}}$ in the notation of \eqref{eq:finite-DV}.
Lift this functional to $\mc P(\mc X)$ by
\begin{equation} \label{eq:J-p}
    \mc J^{(p)}(\nu)
    \coloneqq
    \begin{cases}
        \mf J^{(p)}(\omega),
        & \displaystyle
        \nu
        =
        \sum_{\mc M\in\ms V^{(p)}}
        \omega(\mc M)\pi_{\mc M},
        \quad
        \omega\in\mc P(\ms V^{(p)}),\\
        +\infty,
        & \text{otherwise}.
    \end{cases}
\end{equation}
This representation is unique because the sets in $\ms V^{(p)}$ are pairwise disjoint and $\supp\pi_{\mc M}=\mc M$ for every $\mc M\in\ms V^{(p)}$. By \cref{lem:finite-DV-finite}, $\mf J^{(p)}$ is finite on $\mc P(\ms V^{(p)})$. Therefore,
\[
    \dom\mc J^{(p)}
    =
    \left\{
        \sum_{\mc M\in\ms V^{(p)}}
        \omega(\mc M)\pi_{\mc M}:
        \omega\in\mc P(\ms V^{(p)})
    \right\}.
\]
By \cref{lem:finite-DV-zero,lem:finite-invariant-measures}, the zero set of a finite-state Donsker--Varadhan functional is the convex hull of the invariant laws on its closed communication classes. Lifting these laws by \eqref{eq:piM-def} gives
\begin{equation} \label{eq:zero-set-Jp}
    \{\nu\in\mc P(\mc X):\mc J^{(p)}(\nu)=0\}
    =
    \bigg\{
    \sum_{\mc M\in\ms V^{(p+1)}}
    \omega(\mc M)\pi_{\mc M}:\omega\in\mc P(\ms V^{(p+1)})
    \bigg\},
\end{equation}
for $1\le p\le\mf q$. The same identity holds for $p=0$ by \eqref{eq:J-zero-zero} and $\pi_{\{m\}}=\delta_m$. For $p<\mf q$, its right-hand side is $\dom\mc J^{(p+1)}$; for $p=\mf q$, it is a singleton because $\ms V^{(\mf q+1)}$ has one element. Together with \eqref{eq:J-minus-zero}, this gives all the zero-set and domain relations required in \cref{def:gamma-expansion}.

\subsection{Main result}
The following theorem is the main result of this paper.
\begin{theorem}[$\Gamma$-expansion for metastable diffusion] \label{thm:main}
    Under \cref{ass:algebraic-structure,ass:cocentration,ass:complete-metastable-hierarchy}, the scales and limiting functionals defined in \cref{ssec:limit-define}, indexed by $p\in\{-1,\ldots,\mf q\}$, form a full $\Gamma$-expansion of the level-2 rate functionals $(\mc I_\eps)_{\eps>0}$. In the notation of \cref{def:gamma-expansion},
    \[
        \mc I_\eps
        =
        \frac1{\eps}\mc J^{(-1)}
        +
        \mc J^{(0)}
        +
        \sum_{p=1}^{\mf q}\frac1{\theta_\eps^{(p)}}\mc J^{(p)}.
    \]
\end{theorem}

\begin{samepage}
The two pre-metastable limits are established separately.
\begin{proposition} \label{prop:pre-metastable-gamma}
    Under \cref{ass:algebraic-structure} and items \textup{(1)}--\textup{(2)} of \cref{ass:cocentration}, the following statements hold.
    \begin{enumerate}
        \item
            The functionals $\eps\mc I_\eps$ $\Gamma$-converge to $\mc J^{(-1)}$.
        \item
            If $\varrho_\eps\prec\eps$, then $\varrho_\eps\mc I_\eps$ $\Gamma$-converges to the zero functional.
    \end{enumerate}
\end{proposition}

\end{samepage}

\begin{proposition} \label{prop:unit-gamma}
    Under \cref{ass:algebraic-structure,ass:cocentration}, the functionals $\mc I_\eps$ $\Gamma$-converge to $\mc J^{(0)}$.
\end{proposition}

\Cref{prop:pre-metastable-gamma,prop:unit-gamma} are proved in \cref{sec:scale-eps,sec:unit-scale}, respectively.

The next theorem transfers metastable resolvent convergence at a single scale to $\Gamma$-convergence. It does not require a complete hierarchy. For this theorem, the superscript $p$ labels a single scale. Let $\ms V^{(p)}$ be a finite state space with Markov generator $\mf L^{(p)}$. For each $\mc M\in\ms V^{(p)}$, choose the wells $\mc E(\mc M)$ as in \cref{ssec:metastable-hierarchy} and a probability measure $\pi_{\mc M}$ compactly supported in that interior. Define $\mf J^{(p)}$ and its lift $\mc J^{(p)}$ by \eqref{eq:finite-Jp} and \eqref{eq:J-p}; the disjoint supports of the measures $\pi_{\mc M}$ make the lift well defined.

\begin{theorem}[Resolvent transfer at a single scale] \label{thm:resolvent-transfer}
    Suppose that \cref{ass:algebraic-structure} and items \textup{(1)}--\textup{(2)} of \cref{ass:cocentration} hold. Let $\theta_\eps^{(p)}$ be a positive scale, and assume the following.
    \begin{enumerate}
        \item \emph{$\Gamma$-liminf from previous scale.}
            There exist a positive scale $\beta_\eps\prec\theta_\eps^{(p)}$ and a functional $\mc J:\mc P(\mc X)\to[0,+\infty]$ such that
            \begin{equation} \label{eq:transfer-previous-lower}
                \lowlim_{\eps\to0}\beta_\eps\mc I_\eps(\nu_\eps)\ge\mc J(\nu)
                \qquad\text{whenever }\nu_\eps\rightharpoonup\nu,
            \end{equation}
            \begin{equation} \label{eq:transfer-zero-set}
                \{\nu\in\mc P(\mc X):\mc J(\nu)=0\}
                =
                \left\{
                    \sum_{\mc M\in\ms V^{(p)}}\omega(\mc M)\pi_{\mc M}
                    :\omega\in\mc P(\ms V^{(p)})
                \right\}.
            \end{equation}
        \item \emph{Resolvent convergence at target scale.}
            For every $\lam>0$ and every $g:\ms V^{(p)}\to\bb R$, the convergence in \eqref{eq:metastable-resolvent-convergence} holds at scale $\theta_\eps^{(p)}$ for the wells and generator specified above.
    \end{enumerate}
    Then $\theta_\eps^{(p)}\mc I_\eps$ $\Gamma$-converges to $\mc J^{(p)}$.
\end{theorem}

The proof is given in \cref{sec:metastable-scales}. Now we combine the results to prove the main theorem.

\begin{proof}[Proof of \cref{thm:main}]
    We prove the five items in \cref{def:gamma-expansion}. Item \textup{(1)} follows from $\eps\prec1$ and \cref{ass:complete-metastable-hierarchy}. \Cref{prop:pre-metastable-gamma,prop:unit-gamma} give item \textup{(2)} and the pre-metastable part of item \textup{(3)}. If $\mf q\ge1$, apply \cref{thm:resolvent-transfer} first with $\beta_\eps=1$ and $\mc J=\mc J^{(0)}$, using \eqref{eq:J-zero-zero}. At level $p\ge2$, take $\beta_\eps=\theta_\eps^{(p-1)}$ and $\mc J=\mc J^{(p-1)}$. The lower bound already obtained at level $p-1$ and \eqref{eq:zero-set-Jp} verify \eqref{eq:transfer-previous-lower} and \eqref{eq:transfer-zero-set}, respectively; the hierarchy supplies resolvent convergence. Induction gives every metastable $\Gamma$-limit, completing item \textup{(3)}. Items \textup{(4)} and \textup{(5)} follow from \eqref{eq:J-minus-zero} and \eqref{eq:zero-set-Jp}, as observed above. All conditions of \cref{def:gamma-expansion} are therefore satisfied, including when $\mf q=0$.
\end{proof}
\section{Applications} \label{sec:applications}

We apply \cref{thm:main} to three Langevin diffusions. In each case, the model assumptions imply the pre-metastable conclusions of \cref{prop:pre-metastable-gamma,prop:unit-gamma}: the limits at scales $\eps$ and $1$, and the zero limit at every smaller scale $\varrho_\eps\prec\eps$. A complete metastable hierarchy then gives the full $\Gamma$-expansion, with scales and limiting functionals as in \cref{ssec:limit-define}. For the generalized overdamped model, existing resolvent results verify this hierarchy under an additional condition. For the underdamped and third-order models, it remains an assumption. The proofs and structural verifications are given in \cref{app:langevin-verification}.

\subsection{Generalized overdamped Langevin dynamics} \label{ssec:overdamped}

Let $n\ge2$. On $\bb R^n$, consider
\[
    \dd X_t^\eps=-\bigl(\nabla U+\ell\bigr)(X_t^\eps)\dd t+\sqrt{2\eps}\,\dd B_t.
\]
\begin{samepage}
In the notation of \cref{sec:model}, $d=r=n$, $a=\sig=\msf{Id}_n$, and $b=-\nabla U-\ell$.
We impose the following conditions.
\begin{taggedconditions}
    \item[\textup{\textbf{(OD1)}}]
        We have $U\in C^\infty(\bb R^n)$ and $\ell\in C^\infty(\bb R^n;\bb R^n)$, and $U$ is Morse.

    \item[\textup{\textbf{(OD2)}}]
        As $|x|\to\infty$,
        \[
            \frac{x}{|x|}\cdot\nabla U(x)\longrightarrow\infty,
            \qquad
            |\nabla U(x)|-2\Delta U(x)\longrightarrow\infty.
        \]
        We normalize $\min U=0$.

    \item[\textup{\textbf{(OD3)}}]
        The vector field $\ell$ is orthogonal to $\nabla U$:
        \[
            \ell\cdot\nabla U=0.
        \]
\end{taggedconditions}
\end{samepage}

Under \textup{\textbf{(OD1)}}--\textup{\textbf{(OD3)}}, \cref{lem:overdamped-structural} verifies the assumptions for the pre-metastable limits. To obtain the complete hierarchy from the available resolvent results, we impose the additional condition below.
\begin{taggedconditions}
    \item[\textup{\textbf{(OD-M)}}]
        We have $\diver\ell=0$, and, for every index-one saddle of $U$, both forward branches of its one-dimensional unstable manifold under $\dot x=b(x)$ converge to local minima of $U$; the two limiting minima need not be distinct.
\end{taggedconditions}

\begin{theorem} \label{thm:overdamped-gamma-expansion}
    Under \textup{\textbf{(OD1)}}--\textup{\textbf{(OD3)}} and \textup{\textbf{(OD-M)}}, the level-2 rate functionals of the generalized overdamped Langevin diffusion admit the full $\Gamma$-expansion of \cref{thm:main}.
\end{theorem}

When $\ell=0$, the model reduces to the classical overdamped Langevin diffusion, a canonical reversible model of metastability; see, for instance, \cite{BovierEckhoffGayrardKlein2004}. Under an additional integrability hypothesis, the full $\Gamma$-expansion of its level-2 rate functional was established in \cite[Theorem~2.6]{LandimLeeMariani2025}. Thus \cref{thm:overdamped-gamma-expansion} extends this conclusion to the nonreversible Gibbsian setting.

\subsection{Underdamped Langevin dynamics} \label{ssec:underdamped}

Let $n\ge1$ be the dimension of the position space. On $\bb R^{2n}$, consider the family of diffusions $X_t^\eps=(Q_t^\eps,P_t^\eps)\in\bb R^n\times\bb R^n$ where
\[
    \begin{aligned}
        \dd Q_t^\eps
        &=
        P_t^\eps\dd t,\\
        \dd P_t^\eps
        &=
        -\bigl(\nabla U(Q_t^\eps)+P_t^\eps\bigr)\dd t
        +
        \sqrt{2\eps}\,\dd B_t.
    \end{aligned}
\]
In the notation of \cref{sec:model}, $d=2n$, $r=n$, and
\[
    \sig
    =
    \begin{pmatrix}
        0\\
        \msf{Id}_n
    \end{pmatrix},
    \qquad
    a
    =
    \begin{pmatrix}
        0&0\\
        0&\msf{Id}_n
    \end{pmatrix},
    \qquad
    b(q,p)
    =
    \begin{pmatrix}
        p\\
        -\nabla U(q)-p
    \end{pmatrix}.
\]
\begin{samepage}
We impose the following condition.
\begin{taggedconditions}
    \item[\textup{\textbf{(UD)}}]
        The function $U\in C^\infty(\bb R^n)$ is Morse and
        \[
            \frac{q}{|q|}\cdot\nabla U(q)
            \longrightarrow
            \infty
            \qquad
            \text{as }|q|\to\infty.
        \]
        We normalize $\min U=0$.
\end{taggedconditions}
\end{samepage}
Under \textup{\textbf{(UD)}}, \cref{lem:underdamped-structural} verifies the assumptions for the pre-metastable limits. The full expansion follows once the metastable hierarchy is available.

\begin{theorem} \label{thm:underdamped-gamma-expansion}
    Assume \textup{\textbf{(UD)}}. If the underdamped Langevin diffusion admits a complete metastable hierarchy, its level-2 rate functionals admit the full $\Gamma$-expansion of \cref{thm:main}.
\end{theorem}

\subsection{Third-order Langevin dynamics} \label{ssec:third-order}

Let $n\ge1$. On $\bb R^{3n}$, consider $X_t^\eps=(Q_t^\eps,P_t^\eps,R_t^\eps)$ satisfying
\begin{equation} \label{eq:third-order-langevin}
\begin{aligned}
    \dd Q_t^\eps&=P_t^\eps\dd t,\\
    \dd P_t^\eps&=\bigl(-\nabla U(Q_t^\eps)+R_t^\eps\bigr)\dd t,\\
    \dd R_t^\eps&=-\bigl(P_t^\eps+R_t^\eps\bigr)\dd t+\sqrt{2\eps}\,\dd B_t.
\end{aligned}
\end{equation}
Here $B_t$ is an $n$-dimensional Brownian motion. We impose the following condition.
\begin{taggedconditions}
    \item[\textup{\textbf{(TD)}}]
        The function $U\in C^\infty(\bb R^n)$ is Morse and, after the normalization $\min U=0$, there exist $c_0,R_0>0$ such that
        \begin{equation}\label{eq:third-order-radial-coercivity}
            q\cdot\nabla U(q)\ge c_0\bigl(|q|^2+U(q)\bigr),
            \qquad |q|\ge R_0.
        \end{equation}
        Moreover,
        \begin{equation}\label{eq:third-order-force-growth}
            \frac{|\nabla U(q)|^2}{(1+U(q))(1+q\cdot\nabla U(q))}
            \longrightarrow0
            \qquad\text{as }|q|\to\infty.
        \end{equation}
\end{taggedconditions}

The nonlinear force acts in the $p$-equation, outside the noise range, so item \textup{(2)} of \cref{ass:algebraic-structure} need not hold in these coordinates. The diffeomorphism
\[
    T(q,p,r)=\bigl(q,p,r-\nabla U(q)\bigr)=(q,p,\widetilde r)
\]
leaves the noise coefficient unchanged and gives the drift
\[
    \widetilde b(q,p,\widetilde r)
    =\bigl(p,\widetilde r,-p-\widetilde r-\nabla U(q)-\nabla^2U(q)p\bigr).
\]
The nonlinearity now appears only in the noisy equation. Under \textup{\textbf{(TD)}}, \cref{lem:third-order-structural} verifies \cref{ass:algebraic-structure,ass:cocentration} for the transformed process. The pre-metastable limits transfer to the original coordinates through the homeomorphism $T_\#$; the same argument transports a complete metastable hierarchy and the full expansion, as shown in \cref{ssec:third-order-verification}.

\begin{theorem} \label{thm:third-order-gamma-expansion}
    Assume \textup{\textbf{(TD)}}. If the third-order Langevin diffusion admits a complete metastable hierarchy, its level-2 rate functionals admit the full $\Gamma$-expansion of \cref{thm:main}.
\end{theorem}

\begin{remark}[A unique local minimum]
    For each model, assume \textup{\textbf{(OD1)}}--\textup{\textbf{(OD3)}}, \textup{\textbf{(UD)}}, or \textup{\textbf{(TD)}}, respectively. If $U$ has a unique local minimum, there is exactly one stable equilibrium. By the convention in \cref{ssec:metastable-hierarchy}, the hierarchy is then complete with $\mf q=0$, and the full $\Gamma$-expansion consists of the scales $\eps$ and $1$. No additional metastable assumption is needed; in particular, \textup{\textbf{(OD-M)}} is unnecessary in this case.
\end{remark}

\section{The Donsker--Varadhan framework} \label{sec:DV-framework}

Throughout this section, $(X_t^\eps)$ denotes the diffusion in \eqref{eq:sde-eps}. We assume \cref{ass:algebraic-structure} and items \textup{(1)}--\textup{(2)} of \cref{ass:cocentration}.

\subsection{Properties of the diffusion}

For $\eps>0$, set $\Psi_\eps\coloneqq\exp(V/\eps)$ and $\Phi_\eps\coloneqq-\mc L_\eps\Psi_\eps/\Psi_\eps$. Since $V$ is nonnegative and inf-compact, $\Psi_\eps\in C^\infty(\mc X;[1,\infty))$ is inf-compact. The chain rule gives
\begin{equation} \label{eq:DV-Lyapunov-eps}
    \Phi_\eps
    =
    \frac1\eps W-a:\nabla^2V.
\end{equation}
By the Lyapunov bound,
\[
    \Phi_\eps
    \ge
    \left(\frac1\eps-C\right)W-C.
\]
For every sufficiently small $\eps>0$, the coefficient $\eps^{-1}-C$ is positive; hence the inf-compactness of $W$ implies that of $\Phi_\eps$.

Let $\sig_1,\ldots,\sig_r$ be the columns of $\sig$. The Stratonovich form of \eqref{eq:sde-eps} is
\[
    \dd X_t^\eps
    =
    b_\eps^\circ(X_t^\eps)\dd t
    +
    \sqrt{2\eps}\,\sig(X_t^\eps)\circ\dd B_t,
\]
where
$b_\eps^\circ\coloneqq b-\eps\sum_{k=1}^rD\sig_k\sig_k$.
Accordingly,
\[
    \mc L_\eps
    =
    b_\eps^\circ\cdot\nabla
    +
    \eps\sum_{k=1}^r(\sig_k\cdot\nabla)^2.
\]

\begin{lemma} \label{lem:hypoellipticity-eps}
    The vector fields $b_\eps^\circ\cdot\nabla$ and $\sig_1\cdot\nabla,\ldots,\sig_r\cdot\nabla$ satisfy the parabolic H\"ormander bracket condition for every $\eps>0$. Consequently, $\mc L_\eps$, $\alpha-\mc L_\eps$ for every $\alpha\in\bb R$, the Lebesgue formal adjoint $\mc L_\eps^*$, and $\partial_t-\mc L_\eps$ are hypoelliptic.
\end{lemma}
\begin{proof}
    Use the convention $[X,Y]=DY\,X-DX\,Y$. Since every $\sig_k$ takes values in the fixed space $\mf H$, its derivatives do so as well. The Stratonovich drift therefore has the form
    \[
        b_\eps^\circ(x)
        =
        b(0)+\msf\Omega x+h_\eps(x),
        \qquad
        h_\eps\in C^\infty(\mc X;\mf H).
    \]
    Let $\mf M$ be the smooth module generated by the fields $\sig_k$ and their iterated brackets with $b_\eps^\circ$. It contains every constant field $e\in\mf H$, since $e=\sig\msf R e$. The identity
    \[
        [b_\eps^\circ,fY]
        =
        (b_\eps^\circ f)Y+f[b_\eps^\circ,Y]
    \]
    shows that $[b_\eps^\circ,\mf M]\subset\mf M$. If a constant field $e$ belongs to $\mf M$, then
    \[
        [b_\eps^\circ,e]
        =
        -\msf\Omega e-Dh_\eps e,
    \]
    where $Dh_\eps e$ is $\mf H$-valued and hence belongs to $\mf M$. Thus $\msf\Omega e$ also belongs to $\mf M$. By induction, $\mf M$ contains every constant field in
    \[
        \mf H+\msf\Omega\mf H+\cdots+\msf\Omega^{d-1}\mf H
        =\bb R^d.
    \]
    This proves the parabolic H\"ormander bracket condition.

    Write $X_0\coloneqq b_\eps^\circ\cdot\nabla$ and
    $X_k\coloneqq\sig_k\cdot\nabla$. Since
    \[
        \mc L_\eps
        =
        X_0+\eps\sum_{k=1}^rX_k^2,
    \]
    H\"ormander's theorem \cite{Hormander1967} gives the hypoellipticity
    of $\mc L_\eps$. The same theorem applies after adding the constant
    zeroth-order term $-\alpha$, so $\alpha-\mc L_\eps$ is hypoelliptic
    for every $\alpha\in\bb R$.

    On $\bb R\times\mc X$, the fields $X_k$ together with
    $X_0-\partial_t$ satisfy the H\"ormander condition: their iterated
    commutators generate all spatial directions, while
    $X_0-\partial_t$ supplies the time direction. Hence
    $\partial_t-\mc L_\eps$ is hypoelliptic.

    Finally, since $X_k^*=-X_k-\diver\sig_k$,
    \[
        \mc L_\eps^*
        =
        \widetilde X_0
        +
        \eps\sum_{k=1}^rX_k^2
        +
        c_\eps,
        \qquad
        \widetilde X_0
        \coloneqq
        -X_0
        +
        2\eps\sum_{k=1}^r(\diver\sig_k)X_k,
    \]
    for some $c_\eps\in C^\infty(\mc X)$. The coefficient vector field of
    $\widetilde X_0$ has the form
    \[
        -b(0)-\msf\Omega x+\widetilde h_\eps(x),
        \qquad
        \widetilde h_\eps\in C^\infty(\mc X;\mf H).
    \]
    The preceding module argument therefore applies with $-\msf\Omega$
    in place of $\msf\Omega$. Since
    $\mf H+(-\msf\Omega)\mf H+\cdots+(-\msf\Omega)^{d-1}\mf H
    =\bb R^d$, H\"ormander's theorem also gives the hypoellipticity of
    $\mc L_\eps^*$.
\end{proof}
\begin{remark}
    The assumption that $b$ and $\sig$ are smooth is made for convenience. It may be replaced by sufficiently high finite regularity, provided that the corresponding finite-order hypoelliptic estimates are available. We retain smooth coefficients so that H\"ormander's hypoellipticity theorem applies directly, without tracking the differentiability required in each regularity argument.
\end{remark}

\begin{proposition} \label{prop:diffusion-properties-eps}
    For every sufficiently small $\eps>0$, the following statements hold.
    \begin{enumerate}
        \item \emph{Well-posedness.}
            The SDE \eqref{eq:sde-eps} is globally strongly well-posed. It defines a non-explosive strong Markov process $(X_t^\eps,\bb P_x^\eps)$ with continuous paths. Its semigroup is denoted by
            \[
                P_t^\eps f(x)
                \coloneqq
                \bb E_x^\eps[f(X_t^\eps)].
            \]
            The operator $R_\alpha^\eps$ defined in \cref{ssec:metastable-hierarchy} is its resolvent, so $R_\alpha^\eps g=\int_0^\infty e^{-\alpha t}P_t^\eps g\dd t$.
        \item \emph{Strong Feller property.}
            The operator $P_t^\eps$ is strong Feller for every $t>0$.
        \item \emph{Resolvent irreducibility.}
            For every $\alpha>0$ and every nonempty open set $O\subset\mc X$, $R_\alpha^\eps \bs1_O>0$ pointwise.
        \item \emph{Smooth invariant measure.}
            There exists a unique invariant probability measure $\pi_\eps$, and $\pi_\eps$ has a smooth Lebesgue density and full support in $\mc X$.
    \end{enumerate}
\end{proposition}
\begin{proof}
    Fix a sufficiently small $\eps>0$. By \eqref{eq:DV-Lyapunov-eps}, $\Phi_\eps$ is inf-compact, and hence
    \[
        \mc L_\eps\Psi_\eps
        \le
        C_\eps-\Psi_\eps
    \]
    for some $C_\eps<\infty$. Since $b$ and $\sig$ are smooth, local strong existence and pathwise uniqueness hold. The preceding estimate and the inf-compactness of $\Psi_\eps$ imply non-explosion. Thus, the SDE is globally strongly well-posed and defines a non-explosive strong Markov process with continuous paths. Its resolvent is given by the Laplace transform of its semigroup, which proves~(1).

    For $h\in B_b(\mc X)$, set $u(t,x)=P_t^\eps h(x)$. Localization of the diffusion coefficients, followed by removal of the localization using non-explosion, shows that $u$ is a locally bounded distributional solution of
    \[
        (\partial_t-\mc L_\eps)u=0
        \qquad
        \text{on }(0,\infty)\times\mc X.
    \]
    By \cref{lem:hypoellipticity-eps}, the operator $\partial_t-\mc L_\eps$ is hypoelliptic. Hence $u\in C^\infty((0,\infty)\times\mc X)$, and the equation holds pointwise. In particular $P_t^\eps$ is strong Feller. This proves~(2).

    Let $x\in\mc X$ and let $O\subset\mc X$ be nonempty and open. Choose $y\in O$. By \cref{lem:controlled-connectors}, there are a smooth path $\gamma$ from $x$ to $y$ in time one and a smooth control $u$ such that $\dot\gamma=b(\gamma)-\sig(\gamma)u$. The Stratonovich correction takes values in $\mf H$, so the same path solves
    \[
        \dot\gamma
        =
        b_\eps^\circ(\gamma)
        +\sqrt{2\eps}\,\sig(\gamma)v,
        \qquad
        v(t)
        \coloneqq
        \frac{\msf R(\gamma(t))(b-b_\eps^\circ)(\gamma(t))-u(t)}{\sqrt{2\eps}}.
    \]
    Applying the Stroock--Varadhan support theorem \cite{StroockVaradhan1972} after localizing the coefficients around this trajectory gives $\bb P_x^\eps(X_1^\eps\in O)>0$. By continuity of the paths, the process then spends a positive amount of time in $O$ on an event of positive probability. Consequently, $R_\alpha^\eps\bs1_O(x)>0$ for every $\alpha>0$, proving~(3).

    The Lyapunov function $\Psi_\eps$ yields an invariant probability measure $\pi_\eps$. Strong Fellerness and resolvent irreducibility imply its uniqueness; see \cite[Section~2.1]{Wu2001}. Moreover, for every nonempty open set $O$ and every $\alpha>0$,
    \[
        \frac1\alpha\pi_\eps(O)
        =
        \pi_\eps(R_\alpha^\eps\bs1_O)>0,
    \]
    so $\pi_\eps$ has full support. Finally, invariance implies $\mc L_\eps^*\pi_\eps=0$ in distributions. Since $\mc L_\eps^*$ is hypoelliptic by \cref{lem:hypoellipticity-eps}, the measure $\pi_\eps$ has a smooth Lebesgue density. This proves~(4).
\end{proof}

\subsection{The extended generator} \label{ssec:extended-generator-eps}

Fix a sufficiently small $\eps>0$. For Borel functions $f,g:\mc X\to\bb R$, write $\overline{\ms L}_\eps f=g$ if
\[
    \int_0^t|g(X_s^\eps)|\dd s<\infty
    \quad
    \bb P_x^\eps\text{-a.s. for every }x\in\mc X,\ t>0,
\]
and
\[
    f(X_t^\eps)-f(X_0^\eps)-\int_0^tg(X_s^\eps)\dd s
\]
is a local $\bb P_x^\eps$-martingale for every $x\in\mc X$. The extended generator may be undefined and its representative $g$ need not be unique. If $f,g\in B_b(\mc X)$, we write $f\in D(\ms L_\eps)$ and $\ms L_\eps f=g$. If $f\in C^\infty(\mc X)$, It\^o's formula and non-explosion give the representative $\overline{\ms L}_\eps f=\mc L_\eps f$.

\begin{lemma} \label{lem:generator-equivalence}
    For $f,g\in B_b(\mc X)$, the following are equivalent:
    \begin{enumerate}
        \item $\ms L_\eps f=g$;
        \item $P_t^\eps f-f=\int_0^tP_s^\eps g\dd s$ pointwise for every $t>0$;
        \item $f=R_\alpha^\eps (\alpha f-g)$ for some, and hence for every, $\alpha>0$.
    \end{enumerate}
    Moreover, $D(\ms L_\eps)\subset C(\mc X)$.
\end{lemma}
\begin{proof}
    If $\ms L_\eps f=g$, the defining local $\bb P_x^\eps$-martingale is bounded on every bounded time interval and is therefore a true martingale. Taking $\bb E_x^\eps$ gives (2). Laplace transforming (2) gives (3) for every $\alpha>0$. Finally, for every $h\in B_b(\mc X)$,
    \[
        R_\alpha^\eps h(X_t^\eps)
        -
        R_\alpha^\eps h(X_0^\eps)
        -
        \int_0^t
        \bigl(
            \alpha R_\alpha^\eps h-h
        \bigr)(X_s^\eps)\dd s
    \]
    is a $\bb P_x^\eps$-martingale. If (3) holds for some $\alpha>0$, applying this identity with $h=\alpha f-g$ proves (1). This proves the equivalence.

    Item \textup{(2)} gives $\|P_t^\eps f-f\|_\infty\le t\|g\|_\infty$. Since $P_t^\eps f$ is continuous for $t>0$ by the strong Feller property, $f$ is a uniform limit of continuous functions.
\end{proof}

The next lemma makes precise how the choice of a generator representative affects calculations.
\begin{lemma} \label{lem:generator-representatives-eps}
    If a Borel set $A$ satisfies $\pi_\eps(A)=0$, then $P_t^\eps(x,A)=0$ for every $x\in\mc X$ and $t>0$, and $A$ has zero occupation time under every $\bb P_x^\eps$. For each $f\in D(\ms L_\eps)$, the representatives of $\ms L_\eps f$ are precisely all bounded Borel functions in one $L^\infty(\pi_\eps)$ equivalence class.
\end{lemma}
\begin{proof}
    If $\pi_\eps(A)=0$, invariance gives $\int P_t^\eps(x,A)\dd\pi_\eps(x)=0$. The function $x\mapsto P_t^\eps(x,A)$ is continuous by the strong Feller property, and the full support of $\pi_\eps$ implies $P_t^\eps(x,A)=0$ for every $x$ and $t>0$. Tonelli's theorem then gives zero occupation time under every $\bb P_x^\eps$.

    If $g$ and $\widetilde g$ are two bounded representatives of $\ms L_\eps f$, then
    \[
        \int_0^t(g-\widetilde g)(X_s^\eps)\dd s
    \]
    is a finite-variation local $\bb P_x^\eps$-martingale and hence vanishes identically. Absolute continuity in $t$ gives $(g-\widetilde g)(X_t^\eps)=0$ for almost every $t$, almost surely. Starting $X^\eps$ from $\pi_\eps$ and using stationarity yields $\pi_\eps(|g-\widetilde g|)=0$. Conversely, the zero occupation-time assertion shows that modifying a representative on a $\pi_\eps$-null set does not change its defining martingale.
\end{proof}

\subsection{The level-2 large deviation principle} \label{ssec:LDP-eps}

Recall
\[
    \eta_t^\eps
    \coloneqq
    \frac1t\int_0^t\delta_{X_s^\eps}\dd s,
\]
and equip $\mc P(\mc X)$ with the $\tau$-topology, namely, convergence against bounded measurable functions. Define
\[
    D^+(\ms L_\eps)
    \coloneqq
    D(\ms L_\eps)\cap\{u:u\ge1\}.
\]

\begin{theorem}[Donsker--Varadhan LDP] \label{thm:LDP-eps}
    For every sufficiently small $\eps>0$, the laws of $\eta_t^\eps$ under $\bb P_x^\eps$ satisfy, with speed $t$, the level-2 LDP on the $\tau$-topology, locally uniformly with respect to $x$, with convex good rate functional
    \begin{equation} \label{eq:DV-rate-eps}
        \mc I_\eps(\nu)
        =
        \begin{cases}
            \displaystyle
            \sup_{u\in D^+(\ms L_\eps)}
            \int_{\mc X}
            -\frac{\ms L_\eps u}{u}\dd\nu,
            &\nu\ll\pi_\eps,\\[1ex]
            +\infty,
            &\nu\not\ll\pi_\eps.
        \end{cases}
    \end{equation}
    In particular, the LDP also holds on the weak topology. In the first branch of \eqref{eq:DV-rate-eps}, the integral is independent of the representative chosen for $\ms L_\eps u$.
\end{theorem}
\begin{proof}
    Fix a sufficiently small $\eps>0$. By It\^o's formula, $\overline{\ms L}_\eps\Psi_\eps=\mc L_\eps\Psi_\eps$; hence \eqref{eq:DV-Lyapunov-eps} shows that the inf-compact functions $\Psi_\eps$ and $\Phi_\eps=-\overline{\ms L}_\eps\Psi_\eps/\Psi_\eps$ verify the Lyapunov condition in \cite[Corollary~2.2]{Wu2001}. This gives the hyper-exponential recurrence required in \cite[Theorem~2.1]{Wu2001}. Together with \cref{prop:diffusion-properties-eps}, that theorem yields the locally uniform level-2 LDP in the $\tau$-topology, with $\tau$-compact level sets. The identification \eqref{eq:DV-rate-eps} follows from \cite[Theorem~A.1]{Wu2001}, while convexity is immediate from \eqref{eq:DV-rate-eps}. Finally, \cref{lem:generator-representatives-eps} gives independence of the representative because $\nu\ll\pi_\eps$ in the first branch of \eqref{eq:DV-rate-eps}.
\end{proof}

Define
\[
    D_0^+(\ms L_\eps)
    \coloneqq
    \left\{
    0<u\in D(\ms L_\eps):
    -\frac{\ms L_\eps u}{u}
    \text{ admits a representative in }B_b(\mc X)
    \right\}.
\]
\begin{corollary} \label{cor:D0-DV-eps}
    In the first branch of \eqref{eq:DV-rate-eps}, the class $D^+(\ms L_\eps)$ may be replaced by $D_0^+(\ms L_\eps)$.
\end{corollary}
\begin{proof}
    We have $D^+(\ms L_\eps)\subset D_0^+(\ms L_\eps)$. Conversely, if $u\in D_0^+(\ms L_\eps)$ and $q_u=-\ms L_\eps u/u\in B_b(\mc X)$ is a pointwise representative, then $(u+\delta)/\delta\in D^+(\ms L_\eps)$ and
    \[
        -\frac{\ms L_\eps(u+\delta)}{u+\delta}
        =
        q_u\frac{u}{u+\delta}
        \longrightarrow q_u
    \]
    pointwise as $\delta\downarrow0$, with absolute value bounded by $\|q_u\|_\infty$. Bounded convergence proves the claim.
\end{proof}

\subsection{Smooth test functions} \label{ssec:smooth-core-eps}

Define
\[
    D_s^+(\ms L_\eps)
    \coloneqq
    \left\{
    u\in C^\infty(\mc X):
    \mc L_\eps u\in B_b(\mc X),\
    \exists\,0<m<M\text{ such that }m<u<M
    \right\}.
\]
By It\^o's formula, $D_s^+(\ms L_\eps)\subset D(\ms L_\eps)$ and $\ms L_\eps u=\mc L_\eps u$ on this domain. The following proposition gives a smooth test-function representation of \eqref{eq:DV-rate-eps}.

\begin{proposition}\label{prop:smooth-core-eps}
    For every $\nu\in\mc P(\mc X)$,
    \[
        \mc I_\eps(\nu)
        =
        \sup_{u\in D_s^+(\ms L_\eps)}
        \int_{\mc X}
        -\frac{\mc L_\eps u}{u}\dd\nu.
    \]
    In particular, the supremum is $+\infty$ when $\nu\not\ll\pi_\eps$.
\end{proposition}
\begin{proof}
    We first verify the regularity needed for approximation. For $\alpha>0$ and $g\in C_c^\infty(\mc X)$, set $f=R_\alpha^\eps g$. For $\delta>0$, the function $P_\delta^\eps f$ is smooth by the proof of \cref{prop:diffusion-properties-eps}. Applying $P_\delta^\eps$ to the semigroup identity in \cref{lem:generator-equivalence} and differentiating in time gives
    \[
        (\alpha-\mc L_\eps)P_\delta^\eps f=P_\delta^\eps g.
    \]
    Since $f$ and $g$ are bounded and continuous, path continuity and bounded convergence allow $\delta\downarrow0$ in distributions. Thus $(\alpha-\mc L_\eps)f=g$ in $\mc D'(\mc X)$. Hypoellipticity from \cref{lem:hypoellipticity-eps} makes $f$ smooth and the identity pointwise.

    For an arbitrary $\nu\in\mc P(\mc X)$, fix $\alpha>0$ and $h\in B_b(\mc X)$ with $\inf h>0$. Multiplication of $h$ by a positive constant does not change $h/R_\alpha^\eps h$, so assume $h\ge1$. Define
    \[
        \mu_\alpha
        \coloneqq
        \nu+\nu R_\alpha^\eps ,
        \qquad
        (\nu R_\alpha^\eps )(A)
        \coloneqq
        \int_{\mc X}R_\alpha^\eps \bs1_A\dd\nu.
    \]
    Choose nonnegative $\vph_n\in C_c^\infty(\mc X)$ such that $\vph_n\to h-1/2$ in $L^1(\mu_\alpha)$, and set
    \[
        h_n=\frac12+\vph_n,
        \qquad
        f=R_\alpha^\eps h,
        \qquad
        f_n=R_\alpha^\eps h_n.
    \]
    Since $R_\alpha^\eps \bs1=\alpha^{-1}\bs1$, we have $f_n=(2\alpha)^{-1}+R_\alpha^\eps\vph_n$. The regularity just proved gives $f_n\in C^\infty(\mc X)$ and $(\alpha-\mc L_\eps)f_n=h_n$. Moreover,
    \[
        f_n\ge\frac1{2\alpha},
        \qquad
        f\ge\frac1\alpha,
        \qquad
        \|f_n-f\|_{L^1(\nu)}
        \le
        \|h_n-h\|_{L^1(\nu R_\alpha^\eps )}.
    \]
    Since $h$ is bounded,
    \[
        \left\|
            \frac{h_n}{f_n}-\frac h{f}
        \right\|_{L^1(\nu)}
        \le
        2\alpha\|h_n-h\|_{L^1(\nu)}
        +
        2\alpha^2\|h\|_\infty
        \|f_n-f\|_{L^1(\nu)}
        \longrightarrow0.
    \]
    Since $h_n$ is bounded, $f_n\le\alpha^{-1}\|h_n\|_\infty$ and $\mc L_\eps f_n=\alpha f_n-h_n$ is bounded. Thus $f_n\in D_s^+(\ms L_\eps)$ and
    \begin{equation} \label{eq:smooth-approx-eps}
        \int_{\mc X}
        -\frac{\mc L_\eps f_n}{f_n}\dd\nu
        =
        \int_{\mc X}
        \left(
            \frac{h_n}{f_n}-\alpha
        \right)\dd\nu
        \longrightarrow
        \int_{\mc X}
        \left(
            \frac h{R_\alpha^\eps h}-\alpha
        \right)\dd\nu.
    \end{equation}

    Suppose first that $\nu\ll\pi_\eps$ and take $u\in D^+(\ms L_\eps)$. Choose a bounded representative $g=\ms L_\eps u$ and $\alpha>\|g\|_\infty$. Then $h=\alpha u-g$ is bounded and strictly positive, $u=R_\alpha^\eps h$ by \cref{lem:generator-equivalence}, and
    \[
        -\frac {\ms L_\eps u}u
        =
        \frac h{R_\alpha^\eps h}-\alpha.
    \]
    By \eqref{eq:smooth-approx-eps}, every test in \eqref{eq:DV-rate-eps} can be approximated by tests in $D_s^+(\ms L_\eps)$. The reverse inequality follows from $D_s^+(\ms L_\eps)\subset D_0^+(\ms L_\eps)$ and \cref{cor:D0-DV-eps}.

    Now suppose that $\nu\not\ll\pi_\eps$. Choose a Borel set $A$ such
    that $\pi_\eps(A)=0$ and $\nu(A)>0$.
    \Cref{lem:generator-representatives-eps} gives
    $P_t^\eps(x,A)=0$ for every $x$ and $t>0$, and hence
    $R_\alpha^\eps \bs1_A=0$. For $N>0$, take $h=1+N\bs1_A$. Then
    $R_\alpha^\eps h=\alpha^{-1}$, so \eqref{eq:smooth-approx-eps} gives
    smooth tests whose integrals converge to
    \[
        \int_{\mc X}
        \left(
            \frac h{R_\alpha^\eps h}-\alpha
        \right)\dd\nu
        =
        \alpha N\nu(A).
    \]
    Letting $N\uparrow\infty$ completes the proof.
\end{proof}

\begin{corollary} \label{cor:smooth-eq-eps} \label{cor:smooth-lower-eps}
    For every $\nu\in\mc P(\mc X)$,
    \[
        \mc I_\eps(\nu)
        \ge
        \sup_{\vph\in C_c^\infty(\mc X)}
        \int_{\mc X}
        -\mc L_\eps\vph
        -
        \eps|\sig^{\msf T}\nabla\vph|^2
        \dd\nu.
    \]
    Equality holds when $\nu$ has compact support.
\end{corollary}
\begin{proof}
    For any $\vph\in C_c^\infty(\mc X)$ we have $e^\vph\in D_s^+(\ms L_\eps)$ and
    \[
        \frac{\mc L_\eps e^\vph}{e^\vph}
        =
        \mc L_\eps\vph
        +\eps|\sig^{\msf T}\nabla\vph|^2.
    \]
    Thus \cref{prop:smooth-core-eps} gives the inequality. If $\nu$ has compact support, take $u\in D_s^+(\ms L_\eps)$, choose $\chi\in C_c^\infty(\mc X)$ such that $\chi=1$ on a neighborhood of $\supp\nu$, and set $\vph=\chi\log u$. Then
    \[
        \int_{\mc X}
        -\frac{\mc L_\eps e^\vph}{e^\vph}\dd\nu
        =
        \int_{\mc X}
        -\frac{\mc L_\eps u}{u}\dd\nu.
    \]
    Taking the supremum over $u$ and using \cref{prop:smooth-core-eps} proves equality.
\end{proof}

\section{Resolvents and contact measures}
\label{sec:contact-resolvent}

Throughout this section, we retain the assumptions of \cref{sec:DV-framework}. Each test $F\in D_0^+(\ms L_\eps)$ gives a lower bound in the Donsker--Varadhan formula of \cref{cor:D0-DV-eps}. We will show that, for every such $F$, there is a probability measure $\nu_{\eps,F}\ll\pi_\eps$ for which this bound is exact:
\[
    \mc I_\eps(\nu_{\eps,F})
    =
    \int_{\mc X}-\frac{\ms L_\eps F}{F}\dd\nu_{\eps,F}.
\]
We call $\nu_{\eps,F}$ a \emph{contact measure} for $F$: the test $F$ attains the Donsker--Varadhan supremum at this measure. We construct it as an invariant probability measure of the dynamics tilted by $F$, defined in \cref{ssec:contact-measures-eps}.

We first show that the resolvents of suitable nonnegative, compactly supported Borel functions belong to $D_0^+(\ms L_\eps)$. In \cref{sec:metastable-scales}, we will apply these results to resolvents with sources supported on the metastable wells to construct recovery sequences.

\subsection{Resolvents and the extended generator}

\begin{lemma} \label{lem:positive-resolvent-eps}
    Let $\lam,\theta>0$, and let $G\in B_b(\mc X)$ be nonnegative and compactly supported. Suppose that $G\ge c\bs1_O$ for some $c>0$ and some nonempty open set $O\subset\mc X$. Then
    \[
        F(x)
        \coloneqq
        \bb E_x^\eps
        \left[
            \int_0^\infty e^{-\lam t}G(X_{\theta t}^\eps)\dd t
        \right]
        =
        \frac1\theta R_{\lam/\theta}^\eps G
    \]
    is the unique bounded solution, in the extended-generator sense, of
    \begin{equation}\label{eq:meta-resolvent-temp}
        (\lam-\theta\ms L_\eps)F=G.
    \end{equation}
    Moreover, $F\in D_0^+(\ms L_\eps)$.
\end{lemma}
\begin{proof}
    Applying the equivalence from \cref{lem:generator-equivalence} shows $F$ solves \eqref{eq:meta-resolvent-temp}. Conversely, any bounded solution of \eqref{eq:meta-resolvent-temp} satisfies
    \[
        F
        =
        R_{\lam/\theta}^\eps
        \left(
            \frac\lam\theta F-\ms L_\eps F
        \right)
        =
        \frac1\theta R_{\lam/\theta}^\eps G,
    \]
    proving uniqueness. By the same lemma we obtain the continuity of $F$. By resolvent irreducibility in \cref{prop:diffusion-properties-eps},
    \[
        F
        \ge
        \frac c\theta R_{\lam/\theta}^\eps\bs1_O
        >0
    \]
    pointwise. Hence $F$ has a positive minimum on the compact set $\supp G$, and therefore $G/F\in B_b(\mc X)$. Since $\ms L_\eps F=(\lam F-G)/\theta$,
    \[
        -\frac{\ms L_\eps F}{F}
        =
        \frac1\theta\left(\frac GF-\lam\right)
        \in B_b(\mc X),
    \]
    which proves that $F\in D_0^+(\ms L_\eps)$.
\end{proof}

\subsection{Contact measures} \label{ssec:contact-measures-eps}

The construction below uses stationary measures of tilted dynamics to obtain contact measures. This method was used previously to construct recovery sequences by Landim \cite[Sections~4--5]{Landim2023FiniteState} for finite-state Markov chains and by Landim and Maura \cite[Section~9]{LandimMaura2026} for nonreversible one-dimensional diffusions. We establish the construction here for the admissible extended-generator tests of the present diffusion setting, including degenerate noise on an unbounded state space.

Fix $F\in D_0^+(\ms L_\eps)$. Recall from \cref{lem:generator-equivalence} that $F$ is continuous. By definition, we may choose a pointwise bounded Borel representative
\[
    q_F
    \coloneqq
    -\frac{\ms L_\eps F}{F}
    \in B_b(\mc X)
\]
and the corresponding representative $\ms L_\eps F=-q_FF$. Define the Feynman--Kac semigroup
\[
    P_t^{\eps,q_F}\vph(x)
    \coloneqq
    \bb E_x^\eps\left[
        \vph(X_t^\eps)
        \exp\left(
            \int_0^tq_F(X_s^\eps)\dd s
        \right)
    \right],
    \qquad
    Q_t^{\eps,F}\vph
    \coloneqq
    \frac1F P_t^{\eps,q_F}(F\vph).
\]
By \cref{lem:generator-representatives-eps}, these operators do not depend on the selected representative $q_F$ because every $\pi_\eps$-null set has zero occupation time under every $\bb P_x^\eps$. If
\[
    A_F^\eps(t)
    \coloneqq
    \int_0^tq_F(X_s^\eps)\dd s,
\]
the extended generator martingale for $F$ and integration by parts show that $F(X_t^\eps)e^{A_F^\eps(t)}$ is a local $\bb P_x^\eps$-martingale. It is bounded on bounded time intervals and is therefore a true martingale. Thus $P_t^{\eps,q_F}F=F$ and $Q_t^{\eps,F}\bs1=\bs1$. The Markov property and additivity of $A_F^\eps$ give the semigroup property, so $(Q_t^{\eps,F})$ is a Markov semigroup.

This semigroup describes the dynamics tilted by $F$. When $F$ is smooth, its generator acts on $\vph\in C_c^\infty(\mc X)$ as
\[
    \mc L_\eps^F\vph
    =\frac{\mc L_\eps(F\vph)-\vph\mc L_\eps F}{F}
    =\mc L_\eps\vph+2\eps\,a\nabla\log F\cdot\nabla\vph.
\]
Thus the tilted diffusion has drift $b+2\eps a\nabla\log F$ and the same diffusion matrix as the original process. For general $F\in D_0^+(\ms L_\eps)$, the semigroup construction above defines the tilt without requiring smoothness.

\begin{proposition}[Contact measures] \label{prop:contact-measure-eps}
    For every $F\in D_0^+(\ms L_\eps)$, the tilted semigroup $Q^{\eps,F}$ admits an invariant probability measure $\nu_{\eps,F}\ll\pi_\eps$ such that
    \[
        \mc I_\eps(\nu_{\eps,F})
        =
        \int_{\mc X}
        -\frac{\ms L_\eps F}{F}\dd\nu_{\eps,F}.
    \]
\end{proposition}
\begin{proof}
    We construct $\nu_{\eps,F}$ as an invariant measure of $Q^{\eps,F}$. First, $Q^{\eps,F}$ is strong Feller. For $\vph\in B_b(\mc X)$ and $0<\delta<t$, set $g_\delta=P_{t-\delta}^{\eps,q_F}\vph$. Applying the Markov property gives
    \[
        P_t^{\eps,q_F}\vph(x)-P_\delta^\eps g_\delta(x)
        =\bb E_x^\eps\left[
            \bigl(e^{A_F^\eps(\delta)}-1\bigr)g_\delta(X_\delta^\eps)
        \right].
    \]
    Since $|A_F^\eps(\delta)|\le\delta\|q_F\|_\infty$, the inequality $|e^z-1|\le |z|e^{|z|}$ and the bound $\|g_\delta\|_\infty\le e^{(t-\delta)\|q_F\|_\infty}\|\vph\|_\infty$ yield
    \[
        \left\|
            P_t^{\eps,q_F}\vph
            -P_\delta^\eps g_\delta
        \right\|_\infty
        \le
        \delta\|q_F\|_\infty e^{t\|q_F\|_\infty}\|\vph\|_\infty.
    \]
    The function $g_\delta$ is bounded Borel, so $P_\delta^\eps g_\delta$ is continuous by the strong Feller property of $P_\delta^\eps$. The estimate shows that these continuous functions converge uniformly to $P_t^{\eps,q_F}\vph$ as $\delta\downarrow0$; hence the limit is continuous. Finally, $F\vph$ is bounded Borel and $F$ is positive and continuous, so $Q_t^{\eps,F}\vph=F^{-1}P_t^{\eps,q_F}(F\vph)$ is continuous as well, proving the strong Feller property.

    Recall from \cref{sec:DV-framework} that
    \[
        \Psi_\eps=e^{V/\eps},
        \qquad
        \Phi_\eps=-\frac{\mc L_\eps\Psi_\eps}{\Psi_\eps}
        =\frac1\eps W-a:\nabla^2V,
    \]
    and that both functions are inf-compact for sufficiently small $\eps>0$. Write $\widetilde\Psi_{\eps,F}\coloneqq\Psi_\eps/F$. This function is continuous and inf-compact, since $F$ is positive, continuous, and bounded above. The inf-compactness of $\Phi_\eps$ and boundedness of $q_F$ with the continuity of $\Phi_\eps$, $\Psi_\eps$, and $F$ give a constant $C<\infty$ such that
    \[
        (1+q_F-\Phi_\eps)\Psi_\eps\le CF.
    \]
    Choose compact exhaustion stopping times $\tau_n\uparrow\infty$, so that the stopped stochastic integrals below are true martingales. It\^o's formula gives
    \[
        \begin{aligned}
            &e^{t\wedge\tau_n+A_F^\eps(t\wedge\tau_n)}
            \Psi_\eps(X_{t\wedge\tau_n}^\eps)\\
            &\quad=
            \Psi_\eps(x)
            +\int_0^{t\wedge\tau_n}
            e^{s+A_F^\eps(s)}
            (1+q_F-\Phi_\eps)(X_s^\eps)
            \Psi_\eps(X_s^\eps)\dd s
            +M_{t\wedge\tau_n}\\
            &\quad\le
            \Psi_\eps(x)
            +C\int_0^{t\wedge\tau_n}
            e^{s+A_F^\eps(s)}F(X_s^\eps)\dd s
            +M_{t\wedge\tau_n},
        \end{aligned}
    \]
    for some local martingale $M$. Since $e^{A_F^\eps(s)}F(X_s^\eps)$ is a nonnegative martingale with expectation $F(x)$, taking expectations and then applying Fatou's lemma as $n\to\infty$ yields
    \[
        e^tP_t^{\eps,q_F}\Psi_\eps(x)
        \le
        \Psi_\eps(x)+CF(x)(e^t-1).
    \]
    Equivalently,
    \begin{equation} \label{eq:Q-Lyapunov-eps}
        Q_t^{\eps,F}\widetilde\Psi_{\eps,F}(x)
        \le
        e^{-t}\widetilde\Psi_{\eps,F}(x)+C(1-e^{-t}).
    \end{equation}

    Fix $x\in\mc X$ and set
    \[
        \mu_{\eps,F,T}
        \coloneqq
        \frac1T\int_0^TQ_t^{\eps,F}(x,\cdot)\dd t.
    \]
    By \eqref{eq:Q-Lyapunov-eps},
    \[
        \mu_{\eps,F,T}(\widetilde\Psi_{\eps,F})
        <
        \frac{1-e^{-T}}{T}\widetilde\Psi_{\eps,F}(x)+C.
    \]
    The inf-compactness of $\widetilde\Psi_{\eps,F}$ gives that the family $(\mu_{\eps,F,T})_{T\ge1}$ is tight. Choose $T_n\to\infty$ such that $\mu_{\eps,F,T_n}\rightharpoonup\nu_{\eps,F}$. For $s>0$ and $\vph\in C_b(\mc X)$,
    \[
        \left|
            \mu_{\eps,F,T}(Q_s^{\eps,F}\vph)
            -\mu_{\eps,F,T}(\vph)
        \right|
        \le
        \frac{2s}{T}\|\vph\|_\infty.
    \]
    The strong Feller property allows passage to the limit, so $\nu_{\eps,F}$ is $Q^{\eps,F}$-invariant. Portmanteau's theorem also gives
    \[
        \nu_{\eps,F}(\widetilde\Psi_{\eps,F})\le C<\infty.
    \]
    If $\pi_\eps(A)=0$, then \cref{lem:generator-representatives-eps} gives $P_t^\eps(x,A)=0$ for every $x$ and $t>0$, hence $Q_t^{\eps,F}(x,A)=0$. Invariance implies $\nu_{\eps,F}(A)=0$. This proves $\nu_{\eps,F}\ll\pi_\eps$.

    It remains to show that $F$ attains the Donsker--Varadhan supremum at $\nu_{\eps,F}$. Because $\Psi_\eps\ge1$, we have $\nu_{\eps,F}(F^{-1})<\infty$. Fix a test $u\in D_0^+(\ms L_\eps)$ and choose a representative of $\ms L_\eps u$ such that $\ms L_\eps u/u$ is bounded. The ratio $h=u/F$ is positive and belongs to $L^1(\nu_{\eps,F})$, since $u$ is bounded.

    Apply integration by parts to $e^{A_F^\eps(t)}u(X_t^\eps)$. The extended generator martingale for $u$ contributes the drift $\ms L_\eps u$, while differentiating the exponential contributes $q_Fu$. The resulting local martingale is bounded on bounded time intervals, since $u$, $\ms L_\eps u$, and $q_F$ are bounded. Taking expectations therefore gives
    \[
        P_t^{\eps,q_F}u
        =u+\int_0^tP_s^{\eps,q_F}(\ms L_\eps u+q_Fu)\dd s.
    \]
    By definition, this yields
    \begin{equation} \label{eq:semi-jensen-eps}
        Q_t^{\eps,F}h
        =h+\int_0^tQ_s^{\eps,F}
        \left(\frac{\ms L_\eps u+q_Fu}{F}\right)\dd s.
    \end{equation}
    The quotient $(\ms L_\eps u+q_Fu)/F$ is bounded in absolute value by a constant times $h$, so it also belongs to $L^1(\nu_{\eps,F})$. On this space $Q^{\eps,F}$ is a strongly continuous contraction semigroup: path continuity gives $Q_t^{\eps,F}\vph\to\vph$ pointwise for $\vph\in C_b(\mc X)$, and bounded convergence, density, and invariance extend the convergence to $L^1(\nu_{\eps,F})$. Thus \eqref{eq:semi-jensen-eps} gives
    \begin{equation} \label{eq:semi-derivative-eps}
        \frac{Q_t^{\eps,F}h-h}{t}
        \longrightarrow\frac{\ms L_\eps u+q_Fu}{F}
        \quad\text{in }L^1(\nu_{\eps,F})
        \quad\text{as }t\downarrow0.
    \end{equation}

    For $\delta>0$, Jensen's inequality with invariance and $\log z\le z-1$ give
    \[
        0\le
        \int_{\mc X}\log\frac{Q_t^{\eps,F}h+\delta}{h+\delta}\dd\nu_{\eps,F}
        \le
        \int_{\mc X}\frac{Q_t^{\eps,F}h-h}{h+\delta}\dd\nu_{\eps,F}.
    \]
    Dividing by $t$ and letting $t\downarrow0$, by \eqref{eq:semi-derivative-eps} we obtain
    \[
        \int_{\mc X}
        \frac{\ms L_\eps u+q_Fu}{u+\delta F}\dd\nu_{\eps,F}
        \ge0.
    \]
    As $\delta\downarrow0$, the integrand converges to $\ms L_\eps u/u+q_F$ and is bounded in absolute value by $\|\ms L_\eps u/u+q_F\|_\infty$. Dominated convergence therefore yields the desired comparison:
    \[
        \int_{\mc X}-\frac{\ms L_\eps u}{u}\dd\nu_{\eps,F}
        \le
        \int_{\mc X}q_F\dd\nu_{\eps,F}.
    \]
    Taking the supremum over $u$ in \cref{cor:D0-DV-eps} proves that $\mc I_\eps(\nu_{\eps,F})\le\nu_{\eps,F}(q_F)$. Since $F$ itself belongs to $D_0^+(\ms L_\eps)$, the reverse inequality follows from the same formula.
\end{proof}

\section{Resolvent convergence and metastable scales}
\label{sec:metastable-scales}

We prove \cref{thm:resolvent-transfer}, whose hypotheses are in force throughout this section. Write $\bs y^{(p)}$ for the finite-state chain generated by $\mf L^{(p)}$. We first obtain the $\Gamma$-liminf by lifting the finite-state test functions through the resolvent. The $\Gamma$-limsup uses the contact measures from \cref{sec:contact-resolvent} and is constructed separately on each communication class.

\subsection{The \texorpdfstring{$\Gamma$}{Gamma}-liminf}
\label{ssec:metastable-lower-bound}
\begin{proof}[Proof of \cref{thm:resolvent-transfer}: $\Gamma$-liminf]
    Let $\nu_\eps\rightharpoonup\nu$. If $\nu\notin\dom\mc J^{(p)}$, then $\mc J(\nu)>0$ by \eqref{eq:transfer-zero-set}. Since $\beta_\eps\prec\theta_\eps^{(p)}$ and $\mc I_\eps\ge0$, the lower bound \eqref{eq:transfer-previous-lower} gives $\lowlim_{\eps\to0}\theta_\eps^{(p)}\mc I_\eps(\nu_\eps)=+\infty$. We may therefore assume that $\nu=\sum_{\mc M\in\ms V^{(p)}}\omega(\mc M)\pi_{\mc M}$ for some $\omega\in\mc P(\ms V^{(p)})$.
    Fix a positive function $h$ on $\ms V^{(p)}$, and choose $\lam>0$ large enough so that $k\coloneqq(\lam-\mf L^{(p)})h$ is positive.
    Define
    \[
        K
        \coloneqq
        \sum_{\mc M\in\ms V^{(p)}}
        k(\mc M)\bs1_{\mc E(\mc M)},
        \qquad
        H_\eps
        \coloneqq
        \frac1{\theta_\eps^{(p)}}
        R_{\lam/\theta_\eps^{(p)}}^\eps K.
    \]
    Since $k>0$ and the wells are neighborhoods, the function $K$ is nonnegative, compactly supported, and dominates a positive constant on a nonempty open subset of $\bigcup_{\mc M\in\ms V^{(p)}}\mc E(\mc M)$. Hence \cref{lem:positive-resolvent-eps} gives $H_\eps\in D_0^+(\ms L_\eps)$ and $(\lam-\theta_\eps^{(p)}\ms L_\eps)H_\eps=K$.

    Since $(\lam-\mf L^{(p)})^{-1}k=h$, the convergence in \eqref{eq:metastable-resolvent-convergence} also gives
    \[
        \lim_{\eps\to0}
        \max_{\mc M\in\ms V^{(p)}}
        \sup_{x\in\mc E(\mc M)}
        |H_\eps(x)-h(\mc M)|
        =
        0.
    \]

    By \cref{cor:D0-DV-eps},
    \[
        \begin{aligned}
            \theta_\eps^{(p)}\mc I_\eps(\nu_\eps)
            &\ge
            \int_{\mc X}
            -\frac{\theta_\eps^{(p)}\ms L_\eps H_\eps}{H_\eps}
            \dd\nu_\eps \\
            &=
            -\lam
            +
            \int_{\mc X}\frac{K}{H_\eps}\dd\nu_\eps\\
            &=
            -\lam
            +
            \sum_{\mc M\in\ms V^{(p)}}
            \left(
                \frac{k(\mc M)}{h(\mc M)}
                +
                o_\eps(1)
            \right)
            \nu_\eps(\mc E(\mc M)).
        \end{aligned}
    \]
    By construction, each $\mc E(\mc M)$ is a $\nu$-continuity set, and therefore $\nu_\eps(\mc E(\mc M))\to\nu(\mc E(\mc M))=\omega(\mc M)$ for every $\mc M\in\ms V^{(p)}$.
    Consequently,
    \[
        \begin{aligned}
            \lowlim_{\eps\to0}
            \theta_\eps^{(p)}\mc I_\eps(\nu_\eps)
            &\ge
            -\lam
            +
            \sum_{\mc M\in\ms V^{(p)}}
            \frac{k(\mc M)}{h(\mc M)}
            \omega(\mc M) \\
            &=
            \sum_{\mc M\in\ms V^{(p)}}
            -\frac{\mf L^{(p)}h(\mc M)}{h(\mc M)}
            \omega(\mc M).
        \end{aligned}
    \]
    Taking the supremum over $h>0$ and using \eqref{eq:finite-Jp} and \eqref{eq:J-p} proves the claim.
\end{proof}

\subsection{The \texorpdfstring{$\Gamma$}{Gamma}-limsup}
\label{ssec:metastable-upper-bound}

We first construct the $\Gamma$-limsup recovery measures on a single communication class $\mf D$ of the chain generated by $\mf L^{(p)}$. Denote the jump rates of this chain by $\msf q(\mc M,\mc N)$. For a function $f$ on $\mf D$ or a probability measure $\omega\in\mc P(\mf D)$, write $\widehat f$ and $\widehat\omega$ for their respective zero extensions to $\ms V^{(p)}$.

The reflected generator suppresses jumps out of $\mf D$, while the escape rate records their total rate:
\[
    \mf L_{\mf D}f(\mc M)
    \coloneqq
    \sum_{\mc N\in\mf D}
    \msf q(\mc M,\mc N)\bigl[f(\mc N)-f(\mc M)\bigr],
    \qquad
    c_{\mf D}(\mc M)
    \coloneqq
    \sum_{\mc N\notin\mf D}\msf q(\mc M,\mc N),
    \quad
    \mc M\in\mf D.
\]
Then we have the identity
\begin{equation} \label{eq:class-zero-extension-p}
    \mf L^{(p)}\widehat f
    =
    \mf L_{\mf D}f-c_{\mf D}f
    \quad\text{on }\mf D.
\end{equation}
Write $\mf J_{\mf D}$ for the Donsker--Varadhan rate functional of the reflected chain, given by
\[
    \mf J_{\mf D}(\omega)
    =
    \sup_{g>0}
    \sum_{\mc M\in\mf D}
    -\frac{\mf L_{\mf D}g(\mc M)}{g(\mc M)}\omega(\mc M),
    \qquad \omega\in\mc P(\mf D).
\]
The finite-state definitions and results used below are collected in \cref{app:finite-state}.

Fix $\omega\in\mc P(\mf D)$ with $\omega(\mc M)>0$ for every $\mc M\in\mf D$. By \cref{lem:finite-class-optimizer}, the supremum above is attained. Choose a maximizing function $h:\mf D\to(0,\infty)$, so that
\[
    \mf J_{\mf D}(\omega)
    =
    \sum_{\mc M\in\mf D}
    -\frac{\mf L_{\mf D}h(\mc M)}{h(\mc M)}\omega(\mc M).
\]
We lift $h$ to a test function for the diffusion and take a corresponding contact measure from \cref{prop:contact-measure-eps}. General weights $\omega\in\mc P(\mf D)$ are then treated by approximation.

\begin{lemma} \label{lem:class-recovery-p}
    Let $\omega\in\mc P(\mf D)$ and set
    \[
        \nu
        \coloneqq
        \sum_{\mc M\in\mf D}
        \omega(\mc M)\pi_{\mc M}.
    \]
    There exists a family $\nu_\eps\rightharpoonup\nu$ such that
    \[
        \lim_{\eps\to0}
        \theta_\eps^{(p)}\mc I_\eps(\nu_\eps)
        =
        \mf J^{(p)}(\widehat\omega)
        =
        \mc J^{(p)}(\nu).
    \]
\end{lemma}

\begin{proof}
    First suppose that $\omega>0$ on $\mf D$, and choose an optimizer $h:\mf D\to(0,\infty)$ as above. For sufficiently large $\alpha>0$, define
    \[
        H_\eps^\alpha
        \coloneqq
        \frac1{\theta_\eps^{(p)}}R_{\alpha/\theta_\eps^{(p)}}^\eps
        \left(
            \sum_{\mc M\in\mf D}
            [\alpha h(\mc M)-\mf L^{(p)}\widehat h(\mc M)]
            \bs1_{\mc E(\mc M)}
        \right).
    \]
    Since $h>0$ and $\mf D$ is finite, the function inside the resolvent is positive on the chosen wells and zero elsewhere. It has compact support, so \cref{lem:positive-resolvent-eps} gives $H_\eps^\alpha\in D_0^+(\ms L_\eps)$.

    Let $g$ equal $\alpha h-\mf L^{(p)}\widehat h$ on $\mf D$ and zero elsewhere, and set $f=(\alpha-\mf L^{(p)})^{-1}g$. The chain cannot return to $\mf D$ after leaving it. Since $g$ is zero outside $\mf D$, $f$ vanishes at every state reachable after exit. In particular, $f(\mc N)=0$ whenever $\mc M\in\mf D$, $\mc N\notin\mf D$, and $\msf q(\mc M,\mc N)>0$. Splitting the jumps from $\mc M\in\mf D$ according to whether they stay in or leave $\mf D$, we obtain
    \[
        \begin{aligned}
            \mf L^{(p)}f(\mc M)
            &=\mf L_{\mf D}(f|_{\mf D})(\mc M)
            +\sum_{\mc N\notin\mf D}
            \msf q(\mc M,\mc N)[f(\mc N)-f(\mc M)]\\
            &=\mf L_{\mf D}(f|_{\mf D})(\mc M)
            -c_{\mf D}(\mc M)f(\mc M).
        \end{aligned}
    \]
    Substituting this into $(\alpha-\mf L^{(p)})f=g$ gives
    \[
        (\alpha-\mf L_{\mf D}+c_{\mf D})(f|_{\mf D})
        =g|_{\mf D}
        =(\alpha-\mf L_{\mf D}+c_{\mf D})h,
    \]
    where the last equality follows from \eqref{eq:class-zero-extension-p}. Since $\alpha>0$, the operator on the left is invertible, so $f=h$ on $\mf D$. Therefore \eqref{eq:metastable-resolvent-convergence} yields
    \[
        \max_{\mc M\in\mf D}\sup_{x\in\mc E(\mc M)}
        |H_\eps^\alpha(x)-h(\mc M)|\longrightarrow0
        \qquad\text{for each fixed sufficiently large }\alpha.
    \]
    Choose $\lam_\eps\to\infty$ slowly enough that this error, with $\alpha=\lam_\eps$, is $o(\lam_\eps^{-1})$, and set $H_\eps=H_\eps^{\lam_\eps}$. This diagonal choice is possible by first making the error small for each fixed parameter. The resolvent equation then gives
    \begin{equation} \label{eq:class-resolvent-convergence-p}
        \begin{aligned}
            \max_{\mc M\in\mf D}\sup_{x\in\mc E(\mc M)}
            |H_\eps(x)-h(\mc M)|&\longrightarrow0,\\
            \max_{\mc M\in\mf D}\sup_{x\in\mc E(\mc M)}
            |\theta_\eps^{(p)}\ms L_\eps H_\eps(x)
            -\mf L^{(p)}\widehat h(\mc M)|&\longrightarrow0,
        \end{aligned}
    \end{equation}
    because on $\mc E(\mc M)$ the expression in the second line equals $\lam_\eps|H_\eps-h(\mc M)|$.

    Let $\nu_\eps=\nu_{\eps,H_\eps}$ be a contact measure from \cref{prop:contact-measure-eps}, and put
    \[
        \mc E_{\mf D}\coloneqq\bigcup_{\mc M\in\mf D}\mc E(\mc M),
        \qquad
        \Delta_{\mf D}\coloneqq\mc X\setminus\mc E_{\mf D}.
    \]
    On $\Delta_{\mf D}$, the function inside the resolvent is zero, so $\theta_\eps^{(p)}\ms L_\eps H_\eps=\lam_\eps H_\eps$. By construction of the contact measure,
    \[
        \begin{aligned}
            \theta_\eps^{(p)}\mc I_\eps(\nu_\eps)
            &=\int_{\mc X}
            -\frac{\theta_\eps^{(p)}\ms L_\eps H_\eps}{H_\eps}
            \dd\nu_\eps\\
            &=\int_{\mc E_{\mf D}}
            -\frac{\theta_\eps^{(p)}\ms L_\eps H_\eps}{H_\eps}
            \dd\nu_\eps
            -\lam_\eps\nu_\eps(\Delta_{\mf D}).
        \end{aligned}
    \]
    By \eqref{eq:class-resolvent-convergence-p}, we have
    \begin{equation} \label{eq:recovery-cost-bound-p}
        \begin{split}
            \theta_\eps^{(p)}\mc I_\eps(\nu_\eps)
            +\lam_\eps\nu_\eps(\Delta_{\mf D})
            &=\int_{\mc E_{\mf D}}
            -\frac{\theta_\eps^{(p)}\ms L_\eps H_\eps}{H_\eps}
            \dd\nu_\eps\\
            &=\sum_{\mc M\in\mf D}
            -\frac{\mf L^{(p)}\widehat h(\mc M)}{h(\mc M)}
            \nu_\eps(\mc E(\mc M))
            +o_\eps(1)
            \le C
        \end{split}
    \end{equation}
    for all sufficiently small $\eps$. Consequently,
    \begin{equation} \label{eq:recovery-concentration-p}
        \nu_\eps(\Delta_{\mf D})\le\frac C{\lam_\eps}\longrightarrow0,
        \qquad
        \beta_\eps\mc I_\eps(\nu_\eps)\longrightarrow0.
    \end{equation}

    Compactness of the well closures and \eqref{eq:recovery-concentration-p} give tightness as $\eps\to0$. Let $\nu_{\eps_j}\rightharpoonup\nu_*$ along a subsequence. The lower bound \eqref{eq:transfer-previous-lower} gives $\mc J(\nu_*)=0$, and hence \eqref{eq:transfer-zero-set} yields
    \[
        \nu_*=
        \sum_{\mc M\in\ms V^{(p)}}\omega_*(\mc M)\pi_{\mc M}
        \qquad\text{for some }\omega_*\in\mc P(\ms V^{(p)}).
    \]
    Every well is a $\nu_*$-continuity set. Therefore the well masses $\nu_{\eps_j}(\mc E(\mc M))$ converge to $\omega_*(\mc M)$, and \eqref{eq:recovery-concentration-p} shows that $\omega_*$ is supported on $\mf D$.

    From now on, $\omega_*$ denotes the restriction of the limiting weights to $\mf D$, and $\widehat\omega_*$ denotes its zero extension to $\ms V^{(p)}$. To show that $\omega_*=\omega$, we first apply the $\Gamma$-liminf proved in \cref{ssec:metastable-lower-bound}:
    \[
        \begin{aligned}
            \lowlim_{j\to\infty}
            \theta_{\eps_j}^{(p)}\mc I_{\eps_j}(\nu_{\eps_j})
            &\ge\mc J^{(p)}(\nu_*)\\
            &=\mf J^{(p)}(\widehat\omega_*)\\
            &=\mf J_{\mf D}(\omega_*)+\omega_*(c_{\mf D}),
        \end{aligned}
    \]
    where the last equality is \cref{lem:finite-class-restriction}.

    For the upper bound, drop the nonnegative term $\lam_\eps\nu_\eps(\Delta_{\mf D})$ from \eqref{eq:recovery-cost-bound-p} to obtain
    \[
        \theta_\eps^{(p)}\mc I_\eps(\nu_\eps)
        \le
        \sum_{\mc M\in\mf D}
        -\frac{\mf L^{(p)}\widehat h(\mc M)}{h(\mc M)}
        \nu_\eps(\mc E(\mc M))+o_\eps(1).
    \]
    Along the chosen subsequence, $\nu_{\eps_j}(\mc E(\mc M))\to\omega_*(\mc M)$ for every $\mc M\in\mf D$. Since \eqref{eq:class-zero-extension-p} gives
    \[
        -\frac{\mf L^{(p)}\widehat h(\mc M)}{h(\mc M)}
        =-\frac{\mf L_{\mf D}h(\mc M)}{h(\mc M)}+c_{\mf D}(\mc M),
        \qquad \mc M\in\mf D,
    \]
    it follows that
    \[
        \begin{aligned}
            \uplim_{j\to\infty}
            \theta_{\eps_j}^{(p)}\mc I_{\eps_j}(\nu_{\eps_j})
            &\le\sum_{\mc M\in\mf D}
            -\frac{\mf L_{\mf D}h(\mc M)}{h(\mc M)}\omega_*(\mc M)
            +\omega_*(c_{\mf D})\\
            &\le\mf J_{\mf D}(\omega_*)+\omega_*(c_{\mf D}).
        \end{aligned}
    \]
    The last inequality holds because $h>0$ is admissible in the supremum defining $\mf J_{\mf D}(\omega_*)$. Comparing the lower and upper bounds, and using $\lowlim\le\uplim$, forces equality throughout. In particular,
    \[
        \mf J_{\mf D}(\omega_*)
        =\sum_{\mc M\in\mf D}
        -\frac{\mf L_{\mf D}h(\mc M)}{h(\mc M)}\omega_*(\mc M),
    \]
    so $h$ is an optimizer at $\omega_*$. It was also chosen to be an optimizer at $\omega$. By \cref{lem:finite-class-optimizer}, $\omega_*=\omega$ and hence $\nu_*=\nu$. Thus the tight family $\nu_\eps$ has a unique limit point, so $\nu_\eps\rightharpoonup\nu$. Equality in the bounds also gives
    \[
        \lim_{\eps\to0}
        \theta_{\eps}^{(p)}\mc I_{\eps}(\nu_{\eps})
        =\mf J_{\mf D}(\omega)+\omega(c_{\mf D})
        =\mf J^{(p)}(\widehat\omega).
    \]
    
    For a general $\omega\in\mc P(\mf D)$, fix a strictly positive probability measure $\tau$ on $\mf D$ and put $\omega_t=(1-t)\omega+t\tau$, for $0<t<1$. Convexity, finiteness, and lower semicontinuity of the finite-state functional give $\mf J^{(p)}(\widehat\omega_t)\to\mf J^{(p)}(\widehat\omega)$ as $t\downarrow0$. Applying the preceding construction to $\omega_t$ and choosing a diagonal family proves the assertion.
\end{proof}

\begin{proof}[Proof of \cref{thm:resolvent-transfer}: $\Gamma$-limsup]
    Fix $\nu\in\mc P(\mc X)$. If
    $\nu\notin\dom\mc J^{(p)}$, the constant family is a recovery family.
    Otherwise, write
    \[
        \nu
        =
        \sum_{\mc M\in\ms V^{(p)}}
        \omega(\mc M)\pi_{\mc M},
        \qquad
        \omega\in\mc P(\ms V^{(p)}).
    \]
    Let $\mf D_1,\ldots,\mf D_N$ be all communication classes of $\bs y^{(p)}(\cdot)$. For every $i$ with $\omega(\mf D_i)>0$, let $\omega_i$ be the conditional probability measure of $\omega$ on $\mf D_i$, denote its zero extension to $\ms V^{(p)}$ by $\widehat\omega_i$, and set
    \[
        \nu_i
        \coloneqq
        \sum_{\mc M\in\mf D_i}
        \omega_i(\mc M)\pi_{\mc M}.
    \]
    By \cref{lem:class-recovery-p}, there exists $\nu_\eps^i\rightharpoonup\nu_i$ such that
    \[
        \uplim_{\eps\to0}
        \theta_\eps^{(p)}\mc I_\eps(\nu_\eps^i)
        \le
        \mf J^{(p)}(\widehat\omega_i).
    \]
    Define $\nu_\eps\coloneqq\sum_{\omega(\mf D_i)>0}\omega(\mf D_i)\nu_\eps^i$. Then $\nu_\eps\rightharpoonup\nu$, and convexity of $\mc I_\eps$ gives
    \[
        \begin{aligned}
            \uplim_{\eps\to0}
            \theta_\eps^{(p)}\mc I_\eps(\nu_\eps)
            &\le
            \sum_{\omega(\mf D_i)>0}
            \omega(\mf D_i)
            \uplim_{\eps\to0}
            \theta_\eps^{(p)}\mc I_\eps(\nu_\eps^i)\\
            &\le
            \sum_{\omega(\mf D_i)>0}
            \omega(\mf D_i)\mf J^{(p)}(\widehat\omega_i)
            =
            \mf J^{(p)}(\omega)
            =
            \mc J^{(p)}(\nu),
        \end{aligned}
    \]
    where the penultimate equality follows from \cref{lem:finite-class-splitting}.
\end{proof}

\section{The scale \texorpdfstring{$\eps$}{epsilon}}
\label{sec:scale-eps}

We prove \cref{prop:pre-metastable-gamma}, which supplies the first pre-metastable limit and the limits at smaller scales. Throughout this section, we assume \cref{ass:algebraic-structure,ass:cocentration}. The proofs of \cref{lem:J-Riesz,lem:level-2.5-bound} and the arguments in \cref{ssec:scale-eps,ssec:smaller-scales} also apply under the weaker hypotheses of \cref{prop:pre-metastable-gamma}; in particular, the proof of \cref{prop:compact-stationary-approximation} only requires $W$ to be bounded below.

\subsection{Variational formulas for stationary currents}

Here and below, $\Diver$ denotes the row-wise divergence of a matrix-valued measure or function. For a probability measure $\nu$ and $j\in L^2(\nu;\bb R^r)$, we call $(\nu,j\nu)$ a \emph{stationary pair} if $\diver((b-\sig j)\nu)=0$.

The first formula below identifies $\mc J^{(-1)}$ with the least quadratic cost of a control that makes $\nu$ stationary. The second allows us to use the same current at positive temperature, after correcting for the diffusion term.

\begin{lemma} \label{lem:J-Riesz}
    For every $\nu\in\mc P(\mc X)$,
    \[
        \mc J^{(-1)}(\nu)
        =\frac14\inf_{\substack{j\in L^2(\nu;\bb R^r)\\
        \diver((b-\sig j)\nu)=0}}
        \int_{\mc X}|j|^2\dd\nu,
    \]
    where the infimum of the empty set is $+\infty$. If $\mc J^{(-1)}(\nu)<\infty$, the infimum has a unique minimizer, denoted by $j_\nu$.
\end{lemma}
\begin{proof}
    Suppose first that $\mc J^{(-1)}(\nu)<\infty$. Optimizing over the scalar multiples of each $\vph\in C_c^\infty(\mc X)$ in \eqref{eq:def-J-minus} gives
    \[
        \left|\int_{\mc X}b\cdot\nabla\vph\dd\nu\right|^2
        \le4\mc J^{(-1)}(\nu)
        \int_{\mc X}|\sig^{\msf T}\nabla\vph|^2\dd\nu.
    \]
    Thus the map $\sig^{\msf T}\nabla\vph\mapsto\int b\cdot\nabla\vph\dd\nu$ is a well-defined bounded linear functional on a subspace of $L^2(\nu;\bb R^r)$. Its Riesz representative $j_\nu$ in the closure of that subspace satisfies
    \[
        \int_{\mc X}b\cdot\nabla\vph\dd\nu
        =\int_{\mc X}j_\nu\cdot\sig^{\msf T}\nabla\vph\dd\nu,
        \qquad \vph\in C_c^\infty(\mc X).
    \]
    This is precisely stationarity. Every other stationary control differs from $j_\nu$ by a vector orthogonal to the subspace. Completing the square in \eqref{eq:def-J-minus}, and using density in its closure, gives $\mc J^{(-1)}(\nu)=\frac14\|j_\nu\|_{L^2(\nu)}^2$ and proves minimality and uniqueness. Conversely, any stationary $j\in L^2(\nu;\bb R^r)$ bounds \eqref{eq:def-J-minus} by $\frac14\|j\|_{L^2(\nu)}^2$. This also proves the assertion when $\mc J^{(-1)}(\nu)=+\infty$.
\end{proof}

\begin{lemma} \label{lem:level-2.5-bound}
    Let $\nu\in\mc P(\mc X)$ have compact support. If $j,q\in L^2(\nu;\bb R^r)$ satisfy
    \[
        \diver((b-\sig j)\nu)=0,
        \qquad
        \diver\Diver(a\nu)=-\diver(\sig q\nu),
    \]
    then
    \[
        \mc I_\eps(\nu)
        \le\frac1{4\eps}\int_{\mc X}|j+\eps q|^2\dd\nu.
    \]
\end{lemma}
\begin{proof}
    For $\vph\in C_c^\infty(\mc X)$, testing the two divergence identities against $\vph$ gives, respectively,
    \[
        \int_{\mc X}b\cdot\nabla\vph\dd\nu
        =\int_{\mc X}j\cdot\sig^{\msf T}\nabla\vph\dd\nu,
        \qquad
        \int_{\mc X}a:\nabla^2\vph\dd\nu
        =\int_{\mc X}q\cdot\sig^{\msf T}\nabla\vph\dd\nu.
    \]
    Since $\mc L_\eps\vph=b\cdot\nabla\vph+\eps a:\nabla^2\vph$, it follows that
    \[
        \int_{\mc X}\mc L_\eps\vph\dd\nu
        =\int_{\mc X}(j+\eps q)\cdot\sig^{\msf T}\nabla\vph\dd\nu.
    \]
    By \cref{cor:smooth-eq-eps} and completion of the square,
    \begin{align*}
        \mc I_\eps(\nu)
        &=\sup_{\vph\in C_c^\infty(\mc X)}
        \int_{\mc X}\left[-(j+\eps q)\cdot\sig^{\msf T}\nabla\vph
        -\eps|\sig^{\msf T}\nabla\vph|^2\right]\dd\nu\\
        &\le\frac1{4\eps}\int_{\mc X}|j+\eps q|^2\dd\nu.
    \end{align*}
\end{proof}

Fix a nonincreasing function $\psi\in C^\infty([0,\infty);[0,1])$ equal to $1$ on $[0,1]$ and to $0$ on $[2,\infty)$. For $R>0$, set
\begin{equation}\label{eq:cutoff}
    \Xi_R(s)\coloneqq-\int_s^\infty\psi(t/R)\dd t.
\end{equation}
Since $V$ is inf-compact, $\Xi_R\circ V\in C_c^\infty(\mc X)$. We will use this cutoff here and later.

\begin{corollary} \label{cor:J-minus-zeroset}
    The zero set of $\mc J^{(-1)}$ is $\mc P(\mc C_0)$.
\end{corollary}
\begin{proof}
    If $\mc J^{(-1)}(\nu)=0$, \cref{lem:J-Riesz} gives $\diver(b\nu)=0$. Testing this identity with $\Xi_R\circ V$ yields
    \[
        \int_{\mc X}\psi(V/R)
        \left(W+|\sig^{\msf T}\nabla V|^2\right)\dd\nu=0.
    \]
    Since $W\ge0$, the integrand is nonnegative and increases to $W+|\sig^{\msf T}\nabla V|^2$ as $R\to\infty$. Monotone convergence gives $W=0$ $\nu$-almost everywhere, hence $\nu(\mc C_0)=1$. Conversely, if $\nu\in\mc P(\mc C_0)$, then $b\nu=0$, and \cref{lem:J-Riesz} applies with $j=0$.
\end{proof}

\subsection{\texorpdfstring{$\Gamma$}{Gamma}-convergence at scale \texorpdfstring{$\eps$}{epsilon}}
\label{ssec:scale-eps}

The $\Gamma$-liminf inequality is a direct consequence of \cref{cor:smooth-lower-eps}. The $\Gamma$-limsup inequality requires the construction of a recovery sequence, which is more involved. We first prove the $\Gamma$-liminf inequality.

\begin{proof}[Proof of \cref{prop:pre-metastable-gamma}\textup{(1)}: $\Gamma$-liminf]
    Let $\nu_\eps\rightharpoonup\nu$. For any fixed $\vph\in C_c^\infty(\mc X)$, \cref{cor:smooth-lower-eps}, tested with $\vph/\eps$, gives
    \[
        \eps\mc I_\eps(\nu_\eps)
        \ge\int_{\mc X}\left[-b\cdot\nabla\vph
        -|\sig^{\msf T}\nabla\vph|^2-\eps a:\nabla^2\vph\right]\dd\nu_\eps.
    \]
    The functions in the integrand are bounded and continuous, so passage to the limit followed by the supremum over $\vph$ proves the $\Gamma$-liminf inequality.
\end{proof}

To construct a $\Gamma$-limsup recovery sequence for a general measure $\nu$, we aim to apply \cref{lem:level-2.5-bound}, let $\eps\to0$, and then minimize over stationary controls $j$ using \cref{lem:J-Riesz}. This requires two approximation steps: first, truncate $\nu$ to compactly supported probability measures while preserving stationarity; second, smooth the resulting stationary pairs to ensure the existence of a square-integrable $q$ satisfying the diffusion divergence identity in \cref{lem:level-2.5-bound}. Both steps must preserve the quadratic control cost in the limit.

\begin{proposition} \label{prop:compact-stationary-approximation}
    For every stationary pair $(\nu,j\nu)$ with $j\in L^2(\nu;\bb R^r)$, there are compactly supported stationary pairs $(\mu_n,j_n\mu_n)$ such that
    \[
        \|\mu_n-\nu\|_{\mathrm{TV}}\longrightarrow0,
        \qquad
        \uplim_{n\to\infty}\int_{\mc X}|j_n|^2\dd\mu_n
        \le\int_{\mc X}|j|^2\dd\nu.
    \]
    Here $\|\eta\|_{\mathrm{TV}}=\sup_{|f|\le1}|\int f\dd\eta|$.
\end{proposition}

\begin{proof}
    Let
    \[
        v=b-\sig j,
        \qquad E=\int_{\mc X}|j|^2\dd\nu<\infty.
    \]
    We first estimate $v\cdot\nabla V$. Recall the cutoff $\Xi_R$ from \eqref{eq:cutoff}. Since $b\cdot\nabla V=-W-|\sig^{\msf T}\nabla V|^2$, testing $\diver(v\nu)=0$ against $\Xi_R(V)$ and completing the square gives
    \[
        \begin{aligned}
            0&=\int_{\mc X}\psi(V/R)
            \left(W+|\sig^{\msf T}\nabla V|^2
            +j\cdot\sig^{\msf T}\nabla V\right)\dd\nu\\
            &\ge\int_{\mc X}\psi(V/R)
            \left(W+\tfrac12|\sig^{\msf T}\nabla V|^2\right)\dd\nu
            -\tfrac12E.
        \end{aligned}
    \]
    Since $W$ is bounded below, monotone convergence as $R\to\infty$ gives $W,|\sig^{\msf T}\nabla V|^2\in L^1(\nu)$. Moreover,
    \[
        v\cdot\nabla V
        =-W-|\sig^{\msf T}\nabla V|^2-j\cdot\sig^{\msf T}\nabla V,
    \]
    and the last term is integrable by Cauchy--Schwarz. Hence $v\cdot\nabla V\in L^1(\nu)$.

    Fix $T>0$. We first construct a complete Riemannian metric for which $v$ is integrable. Define
    \[
        g_x(\xi,\xi)
        =|\nabla V(x)\cdot\xi|^2
        +\frac{|\xi|^2}{1+|b(x)|^2+\|\sig(x)\|_{\mathrm F}^2},
    \]
    where $\|\cdot\|_{\mathrm F}$ is the Frobenius norm. From $v=b-\sig j$ we get
    \[
        \|v\|_g
        \le |v\cdot\nabla V|
        +\frac{|b|+\|\sig\|_{\mathrm F}|j|}
        {\sqrt{1+|b|^2+\|\sig\|_{\mathrm F}^2}}
        \le |v\cdot\nabla V|+1+|j|.
    \]
    Integrating gives
    \[
        \int_{\mc X}\|v\|_g\dd\nu
        \le\int_{\mc X}|v\cdot\nabla V|\dd\nu+1+E^{1/2}<\infty.
    \]
    To check completeness, let $d_g$ be the distance induced by $g$. By definition of $g$ we have
    \[
        |V(x)-V(y)|\le d_g(x,y).
    \]
    If $(x_n)$ is $d_g$-Cauchy, then $(V(x_n))$ is bounded. Thus $(x_n)$ lies in a compact sublevel set of $V$ and has a subsequence converging in the Euclidean metric. Since $g$ is smooth and positive definite, the Euclidean and $g$-distances are locally equivalent. The same subsequence therefore converges for $d_g$, and the Cauchy property gives convergence of the whole sequence.

    By the stationary case of \cite[Theorem~5.8]{Bernard2008}, $\diver(v\nu)=0$ and $\int\|v\|_g\dd\nu<\infty$ yield a probability measure $Q_T$ on $C([0,T];\mc X)$, concentrated on absolutely continuous solution paths whose distribution at every time is $\nu$. More precisely,
    \begin{equation}\label{eq:stationary-path-law}
        \begin{aligned}
            (e_t)_\#Q_T&=\nu\quad(0\le t\le T),\\
            \dot\gamma(t)&=b(\gamma(t))-\sig(\gamma(t))j(\gamma(t))
            \quad\text{for a.e. }t.
        \end{aligned}
    \end{equation}
    Here $e_t(\gamma)=\gamma(t)$, and the second identity holds for $Q_T$-almost every path $\gamma$. Each path has compact image, on which the metric norms are uniformly equivalent, so these paths are also absolutely continuous in the Euclidean metric. Tonelli's theorem and the time marginals give
    \[
        \int\int_0^T|j(\gamma(t))|^2\dd t\dd Q_T(\gamma)
        =\int_0^T\int_{\mc X}|j|^2\dd\nu\dd t=TE.
    \]

    We next turn these paths into closed paths. Fix a compact set $K\subset\mc X$. By \cref{lem:controlled-connectors}, any two points of $K$ can be joined in time $1$ by a controlled path with energy at most $C_K$. These joining paths lie in a common compact set and depend smoothly on their endpoints. To bound both the joining cost and the range of the original paths, retain only the paths
    \[
        A=\left\{\gamma:\gamma(0),\gamma(T)\in K,
        \ \sup_{0\le t\le T}|\gamma(t)|\le L\right\},
        \qquad p=Q_T(A).
    \]
    This is a Borel event. By the estimate
    \[
        Q_T\{\gamma(0),\gamma(T)\in K\}\ge1-2\nu(K^c),
    \]
    and the fact that every continuous path is bounded on $[0,T]$, $p\to 1$ as $K$ increases and $L$ increases.

    Suppose $p>0$. Append to each path in $A$ the chosen path from $\gamma(T)$ back to $\gamma(0)$. Denote the resulting closed path by $\overline\gamma:[0,T+1]\to\mc X$ and its control by $u_\gamma$. On $[0,T]$, this control is $j(\gamma(t))$; on $[T,T+1]$, it is the control of the joining path. For $Q_T$-almost every $\gamma\in A$, we have
    \[
        \dot{\overline\gamma}
        =b(\overline\gamma)-\sig(\overline\gamma)u_\gamma
        \quad\text{for a.e. }t,
        \qquad
        \overline\gamma(T+1)=\overline\gamma(0).
    \]
    The control satisfies
    \[
        \int_A\int_0^{T+1}|u_\gamma(t)|^2\dd t\dd Q_T(\gamma)
        \le TE+pC_K.
    \]
    All these paths lie in the union of the closed ball of radius $L$ and the compact set containing the joining paths. Their dependence on $\gamma$ is Borel. Define the averaged occupation measure and the vector measure of controls by
    \begin{align*}
        \mu&=\frac1{p(T+1)}\int_A\int_0^{T+1}
        \delta_{\overline\gamma(t)}\dd t\dd Q_T(\gamma),\\
        \bs k&=\frac1{p(T+1)}\int_A\int_0^{T+1}
        u_\gamma(t)\delta_{\overline\gamma(t)}\dd t\dd Q_T(\gamma).
    \end{align*}
    The normalizing factor $p(T+1)$ makes $\mu$ a probability measure, and its support is compact.

    Choose $\gamma$ with law $Q_T(\cdot\mid A)$ and, independently, choose $t$ uniformly in $[0,T+1]$. The position $\overline\gamma(t)$ then has law $\mu$. Let $h(x)$ be the conditional expectation of $u_\gamma(t)$ given $\overline\gamma(t)=x$. By this definition, $\bs k=h\mu$, and conditional Jensen's inequality gives
    \begin{equation}\label{eq:closed-path-cost}
        \int_{\mc X}|h|^2\dd\mu
        \le\frac1{p(T+1)}\int_A\int_0^{T+1}|u_\gamma(t)|^2\dd t\dd Q_T(\gamma)
        \le\frac{TE+pC_K}{p(T+1)}.
    \end{equation}
    In particular, $h\in L^2(\mu;\bb R^r)$. For every $\vph\in C_c^\infty(\mc X)$, the chain rule gives
    \[
        \begin{aligned}
            \int_{\mc X}(b-\sig h)\cdot\nabla\vph\dd\mu
            &=\frac1{p(T+1)}\int_A\int_0^{T+1}
            \frac{\dd}{\dd t}\vph(\overline\gamma(t))\dd t\dd Q_T(\gamma)\\
            &=\frac1{p(T+1)}\int_A
            [\vph(\overline\gamma(T+1))-\vph(\overline\gamma(0))]\dd Q_T(\gamma)
            =0.
        \end{aligned}
    \]
    Hence $\diver((b-\sig h)\mu)=0$, so $(\mu,h\mu)$ is a compactly supported stationary pair.

    It remains to estimate the distance from $\mu$ to $\nu$. Separate the occupation of the original paths from that of the joining paths by setting
    \[
        \mu_{\mathrm{orig}}=\frac1{pT}\int_A\int_0^T\delta_{\gamma(t)}\dd t\dd Q_T(\gamma),
        \qquad
        \mu_{\mathrm{join}}=\frac1p\int_A\int_T^{T+1}\delta_{\overline\gamma(t)}\dd t\dd Q_T(\gamma).
    \]
    Both are probability measures, and
    \[
        \mu=\frac{T}{T+1}\mu_{\mathrm{orig}}+\frac1{T+1}\mu_{\mathrm{join}}.
    \]
    By \eqref{eq:stationary-path-law}, averaging over all original paths gives $\nu$. Restricting this average to $A$ gives $p\mu_{\mathrm{orig}}\le\nu$. Thus $\nu-p\mu_{\mathrm{orig}}$ is a positive measure of mass $1-p$, and
    \[
        \|\mu_{\mathrm{orig}}-\nu\|_{\mathrm{TV}}
        \le\|(1-p)\mu_{\mathrm{orig}}\|_{\mathrm{TV}}
        +\|\nu-p\mu_{\mathrm{orig}}\|_{\mathrm{TV}}
        =2(1-p).
    \]
    Thus we have
    \begin{equation}\label{eq:closed-path-approximation}
        \|\mu-\nu\|_{\mathrm{TV}}
        \le\|\mu_{\mathrm{orig}}-\nu\|_{\mathrm{TV}}
        +\frac{\|\mu_{\mathrm{join}}-\mu_{\mathrm{orig}}\|_{\mathrm{TV}}}{T+1}
        \le2(1-p)+\frac2{T+1}.
    \end{equation}

    By choosing $K_n$, $L_n$, then $T_n$ appropriately, we can ensure that $p_n\to1$ and $T_n\to\infty$. For the corresponding stationary pairs $(\mu_n,j_n\mu_n)$, \eqref{eq:closed-path-cost} and \eqref{eq:closed-path-approximation} give
    \[
        \int_{\mc X}|j_n|^2\dd\mu_n\le E + o_n(1),
        \qquad
        \|\mu_n-\nu\|_{\mathrm{TV}}=o_n(1).
    \]
    Letting $n\to\infty$ proves both claims.
\end{proof}

\begin{samepage}
The next lemma smooths a compactly supported stationary pair, using a mollifier adapted to the fixed noise range and the affine part of the drift.

\begin{lemma} \label{lem:eps-div-construction}
    Let $(\mu,j\mu)$ be a compactly supported stationary pair with $j\in L^2(\mu;\bb R^r)$. There are smooth compactly supported probability densities $\rho_\delta$ and $j_\delta,q_\delta\in L^2(\rho_\delta\dd x;\bb R^r)$ such that
    \begin{gather*}
        \rho_\delta\dd x\rightharpoonup\mu,\\
        \diver((b-\sig j_\delta)\rho_\delta)=0,\\
        \diver\Diver(a\rho_\delta)=-\diver(\sig q_\delta\rho_\delta),\\
        \uplim_{\delta\to0}\int_{\mc X}|j_\delta|^2\rho_\delta\dd x
        \le\int_{\mc X}|j|^2\dd\mu.
    \end{gather*}
\end{lemma}
\end{samepage}

\begin{proof}
    We will choose an ellipsoidal kernel for which the field $z\mapsto\msf\Omega z$ is divergence free and tangent to the level sets of the kernel. The freedom to change $\msf\Omega$ by a matrix with range in $\mf H$ is useful here: it leaves the nondiffusive drift unchanged.

    Fix $x_*\in\mc X$ and set $\sig_*\coloneqq\sig(x_*)$. By controllability and \cref{lem:stabilizing-feedback}, choose $\msf K\in\bb R^{r\times d}$ such that $\msf A\coloneqq\msf\Omega-\sig_*\msf K$ is Hurwitz. By \cref{lem:feedback-invariance}, $(\msf A,\sig_*)$ remains controllable. \Cref{lem:Lyapunov} yields a positive definite matrix $\msf Q$ for which $\widetilde{\msf\Omega}\coloneqq\msf A+\sig_*\sig_*^{\msf T}\msf Q^{-1}$ satisfies $\widetilde{\msf\Omega}\msf Q+\msf Q\widetilde{\msf\Omega}^{\msf T}=0$.

    Since $\widetilde{\msf\Omega}-\msf\Omega=\sig_*(\sig_*^{\msf T}\msf Q^{-1}-\msf K)$ has range in $\mf H$, we still have $b(x)-b(0)-\widetilde{\msf\Omega}x\in\mf H$. Replacing $\msf\Omega$ by this matrix and relabeling it, we may therefore assume that
    \begin{equation}\label{eq:Omega-Q}
        \msf\Omega\msf Q+\msf Q\msf\Omega^{\msf T}=0.
    \end{equation}

    Let $\vartheta_0\in C_c^\infty(\mc X)$ be a nonnegative radial probability density and define the $\msf Q$-radial probability density $\vartheta(z)\coloneqq(\det\msf Q)^{-1/2}\vartheta_0(\msf Q^{-1/2}z)$. Set $\vartheta_\delta(z)\coloneqq\delta^{-d}\vartheta(z/\delta)$ and $\rho_\delta\coloneqq\vartheta_\delta*\mu$. Then $\rho_\delta\in C_c^\infty(\mc X)$ is a probability density and $\rho_\delta\dd x\rightharpoonup\mu$.

    The identity \eqref{eq:Omega-Q} gives $\tr\msf\Omega=0$ and $(\msf\Omega z)\cdot\msf Q^{-1}z=0$. Since $\nabla\vartheta_\delta(z)$ is a scalar multiple of $\msf Q^{-1}z$, we have
    \[
        \diver_z\big(\msf\Omega z\,\vartheta_\delta(z)\big)=0.
    \]

    We now choose the control so that its current equals the convolved stationary current, up to the affine term just made divergence free. On $\{\rho_\delta>0\}$, define
    \[
        j_\delta(x)
        \coloneqq
        \frac{1}{\rho_\delta(x)}
        \int_{\mc X}
        \msf R(x)
        \big[
        \sig(y)j(y)+b(x)-b(y)-\msf\Omega(x-y)
        \big]
        \vartheta_\delta(x-y)\dd\mu(y),
    \]
    and set $j_\delta=0$ on $\{\rho_\delta=0\}$. The vector inside the square brackets belongs to $\mf H$, because $\ran\sig(y)=\mf H$ and $b(x)-b(y)-\msf\Omega(x-y)\in\mf H$. Hence $\sig(x)\msf R(x)$ acts as the identity on this vector, and rearranging the equation gives
    \[
        (b(x)-\sig(x)j_\delta(x))\rho_\delta(x)
        =
        \int_{\mc X}
        \big[
            b(y)-\sig(y)j(y)+\msf\Omega(x-y)
        \big]
        \vartheta_\delta(x-y)\dd\mu(y).
    \]
    The convolution of $(b-\sig j)\mu$ with $\vartheta_\delta$ is divergence free by stationarity of $(\mu,j\mu)$, while the last term is divergence free because $\diver_z(\msf\Omega z\,\vartheta_\delta(z))=0$. Thus $\diver((b-\sig j_\delta)\rho_\delta)=0$.

    On $\{\rho_\delta>0\}$, set
    \[
        q_\delta
        \coloneqq
        -\sig^{\msf T}\frac{\nabla\rho_\delta}{\rho_\delta}
        -\msf R\Diver a,
    \]
    and set $q_\delta=0$ on $\{\rho_\delta=0\}$. Every column of $a(x)$ belongs to the fixed space $\mf H$, so $\Diver a(x)\in\mf H$. Therefore
    \[
        \sig q_\delta\rho_\delta
        =
        -a\nabla\rho_\delta
        -\rho_\delta\Diver a.
    \]
    This equality holds globally because $\nabla\rho_\delta=0$ on $\{\rho_\delta=0\}$. It follows that
    \[
        -\diver(\sig q_\delta\rho_\delta)
        =
        \diver\big(
        a\nabla\rho_\delta
        +\rho_\delta\Diver a
        \big)
        =
        \diver\Diver(a\rho_\delta).
    \]

    To verify the required integrability, note that $\sig$, $\msf R$, and $\Diver a$ are bounded on the compact support of $\rho_\delta$. Since every $0\le f\in C_c^\infty(\mc X)$ satisfies $|\nabla f|^2\le2d f\|\nabla^2f\|_\infty$, we have
    \[
        \int_{\{\rho_\delta>0\}}
        \frac{|\nabla\rho_\delta|^2}{\rho_\delta}\dd x
        <\infty.
    \]
    Therefore $q_\delta\in L^2(\rho_\delta\dd x;\bb R^r)$.

    For each $x$ with $\rho_\delta(x)>0$, the measure $\vartheta_\delta(x-y)\dd\mu(y)/\rho_\delta(x)$ is a probability measure, and the formula for $j_\delta(x)$ is an average with respect to this measure. Whenever $y\in\supp\mu$ and $\vartheta_\delta(x-y)\ne0$, we have $|x-y|\le C\delta$. Thus, for small $\delta$, both points lie in a fixed compact set. Smoothness and the identity $\msf R(y)\sig(y)=\msf{Id}_r$ then give
    \[
        \begin{aligned}
            \|\msf R(x)\sig(y)-\msf{Id}_r\|&\le C\delta,\\
            |\msf R(x)[b(x)-b(y)-\msf\Omega(x-y)]|&\le C\delta.
        \end{aligned}
    \]
    We apply Jensen's inequality to the formula for $j_\delta(x)$ and integrate with respect to $\rho_\delta(x)\dd x$. Minkowski's inequality then gives
    \[
        \|j_\delta\|_{L^2(\rho_\delta\dd x)}
        \le(1+C\delta)\|j\|_{L^2(\mu)}+C\delta.
    \]
    In particular, $j_\delta\in L^2(\rho_\delta\dd x;\bb R^r)$. Letting $\delta\downarrow0$ gives the required bound.
\end{proof}

We now combine the two approximation results with \cref{lem:level-2.5-bound} to prove the $\Gamma$-limsup inequality.

\begin{proof}[Proof of \cref{prop:pre-metastable-gamma}\textup{(1)}: $\Gamma$-limsup]
    For the upper bound, suppose $\mc J^{(-1)}(\nu)<\infty$ and take the minimizing control $j_\nu$ from \cref{lem:J-Riesz}. Applying \cref{prop:compact-stationary-approximation,lem:eps-div-construction} successively and taking a diagonal sequence gives smooth compactly supported probability densities $\rho_n$ and controls $j_n,q_n\in L^2(\rho_n\dd x;\bb R^r)$ satisfying the two divergence identities and
    \[
        \rho_n\dd x\rightharpoonup\nu,
        \qquad
        \uplim_{n\to\infty}\|j_n\|_{L^2(\rho_n\dd x)}^2
        \le4\mc J^{(-1)}(\nu).
    \]
    The norms of $q_n$ may grow with $n$, but each is finite. We can therefore let the approximation improve slowly enough that the diffusion correction still vanishes: choose $n=n(\eps)\to\infty$ with $\eps\|q_{n(\eps)}\|_{L^2(\rho_{n(\eps)}\dd x)}\to0$. By \cref{lem:level-2.5-bound},
    \[
        \eps\mc I_\eps(\rho_{n(\eps)}\dd x)
        \le\frac14\left(
        \|j_{n(\eps)}\|_{L^2(\rho_{n(\eps)}\dd x)}
        +\eps\|q_{n(\eps)}\|_{L^2(\rho_{n(\eps)}\dd x)}
        \right)^2.
    \]
    This is the required recovery family. If $\mc J^{(-1)}(\nu)=+\infty$, the upper bound is automatic.
\end{proof}

\subsection{Smaller scales}
\label{ssec:smaller-scales}

\begin{proof}[Proof of \cref{prop:pre-metastable-gamma}, part \textup{(2)}]
    Let $\varrho_\eps\prec\eps$. The $\Gamma$-liminf inequality follows from $\mc I_\eps\ge0$. For the upper bound, the key observation is that part \textup{(1)} provides a recovery family with bounded $\eps\mc I_\eps$ whenever $\mc J^{(-1)}$ is finite. Multiplication by $\varrho_\eps/\eps\to0$ then makes its cost vanish at the smaller scale. We will therefore first show that $\dom\mc J^{(-1)}$ is weakly dense, and then apply this observation through a diagonal argument.

    Fix $x\in\mc X$. We approximate $\delta_x$ by occupation measures of closed controlled paths that spend most of their time near $x$. Let $\phi_t(x)$ be the local flow of $b$, and choose $\tau_n\downarrow0$ within its interval of existence. Follow the flow from $x$ for time $\tau_n$, with zero control. By \cref{lem:controlled-connectors}, we can then return from $\phi_{\tau_n}(x)$ to $x$ in time $\tau_n^2$ using a control of finite energy. Denote the closed path by $\gamma_n$, its control by $u_n$, and its total duration by $T_n=\tau_n+\tau_n^2$. Thus $\dot\gamma_n=b(\gamma_n)-\sig(\gamma_n)u_n$ almost everywhere, and its occupation probability is
    \[
        \mu_n\coloneqq\frac1{T_n}\int_0^{T_n}\delta_{\gamma_n(t)}\dd t.
    \]
    The flow segment stays uniformly close to $x$, while the return segment occupies a fraction $\tau_n^2/T_n\to0$ of the total time. More precisely, for every $f\in C_b(\mc X)$,
    \[
        \left|\int_{\mc X}f\dd\mu_n-f(x)\right|
        \le\sup_{0\le t\le\tau_n}|f(\phi_t(x))-f(x)|
        +2\|f\|_\infty\frac{\tau_n^2}{T_n}
        \longrightarrow0.
    \]
    Hence $\mu_n\rightharpoonup\delta_x$, even if the return paths travel far from $x$.

    To see that $\mc J^{(-1)}(\mu_n)$ is finite, take $\vph\in C_c^\infty(\mc X)$. Since $\gamma_n(0)=\gamma_n(T_n)$, the chain rule gives
    \[
        \int_{\mc X}b\cdot\nabla\vph\dd\mu_n
        =\frac1{T_n}\int_0^{T_n}
        u_n(t)\cdot\sig(\gamma_n(t))^{\msf T}\nabla\vph(\gamma_n(t))\dd t.
    \]
    Substituting this identity into \eqref{eq:def-J-minus} and completing the square yields
    \[
        \mc J^{(-1)}(\mu_n)
        \le\frac1{4T_n}\int_0^{T_n}|u_n(t)|^2\dd t<\infty.
    \]
    Since $\mc J^{(-1)}$ is convex, finite convex combinations of these measures also have finite cost. Every finitely supported probability can therefore be approximated by measures in $\dom\mc J^{(-1)}$. The weak density of finitely supported probabilities proves the desired density of this domain.

    Fix $\nu\in\mc P(\mc X)$ and choose $\nu_n\rightharpoonup\nu$ with $\mc J^{(-1)}(\nu_n)<\infty$. Part \textup{(1)} gives recovery families $\nu_\eps^n\rightharpoonup\nu_n$ with bounded $\eps\mc I_\eps(\nu_\eps^n)$ for each fixed $n$. Since $\varrho_\eps/\eps\to0$, their costs $\varrho_\eps\mc I_\eps(\nu_\eps^n)$ tend to zero for fixed $n$. Choosing $n=n(\eps)\to\infty$ sufficiently slowly gives
    \[
        \nu_\eps\coloneqq\nu_\eps^{n(\eps)}\rightharpoonup\nu,
        \qquad
        \varrho_\eps\mc I_\eps(\nu_\eps)\longrightarrow0,
    \]
    which proves the $\Gamma$-limsup inequality.
\end{proof}

\section{The unit scale}
\label{sec:unit-scale}

We prove \cref{prop:unit-gamma}, which supplies the limit at the unit scale. Throughout this section, we assume \cref{ass:algebraic-structure,ass:cocentration}. The scale-$\eps$ $\Gamma$-liminf forces every weak limit of a family with bounded unit-scale cost to be supported on the equilibria. It remains to identify the cost of concentration near each $c\in\mc C_0$. We prove the $\Gamma$-liminf by a compactly supported test function and the $\Gamma$-limsup by a probability measure concentrated near $c$.

Fix $c\in\mc C_0$, write $x=c+y$, and set $\msf B_c=Db(c)$, $\sig_c=\sig(c)$, and $a_c=a(c)$. The pair $(\msf B_c,\sig_c)$ is controllable by \cref{ass:algebraic-structure}, and $\msf B_c$ is hyperbolic by \cref{ass:cocentration}. Hence \cref{lem:Riccati} provides a symmetric nonnegative matrix $\msf P_c$ satisfying the Riccati equation
\[
    \msf B_c^{\msf T}\msf P_c+\msf P_c\msf B_c=2\msf P_ca_c\msf P_c,
    \qquad
    \tr(a_c\msf P_c)=\kappa_c.
\]
This matrix is chosen so that $\msf B_c-2a_c\msf P_c$ is Hurwitz, meaning that all its eigenvalues have negative real parts. \Cref{tab:unit-matrices} collects the matrix notation used throughout the section. The remaining matrices are introduced in the proofs where they are used.

\begin{table}[htbp]
    \centering
    \caption{Matrix notation near an equilibrium $c$.}
    \label{tab:unit-matrices}
    \begin{minipage}{.60\textwidth}
        \small
        \renewcommand{\arraystretch}{1.3}
        \begin{tabular}{@{}>{\centering\arraybackslash}m{3.6em}
            @{\hspace{2.4em}}>{\raggedright\arraybackslash}m{\dimexpr\linewidth-6em\relax}@{}}
            \toprule
            \textbf{Matrix} & \textbf{Definition} \\
            \midrule
            $\msf B_c$ & $Db(c)$ \\
            $\sig_c$ & $\sig(c)$ \\
            $a_c$ & $a(c)=\sig_c\sig_c^{\msf T}$ \\
            \addlinespace[.4ex]
            $\msf P_c$
            & $\msf P_c=\msf P_c^{\msf T}\ge0$,\quad
              $\msf B_c^{\msf T}\msf P_c+\msf P_c\msf B_c=2\msf P_ca_c\msf P_c$,\newline
              chosen so that $\msf B_c-2a_c\msf P_c$ is Hurwitz \\
            $\msf S_c$
            & $\msf S_c=\msf S_c^{\msf T}$,\newline
              $\msf B_c^{\msf T}\msf S_c+\msf S_c\msf B_c+2\msf S_ca_c\msf S_c<0$ \\
            \addlinespace[.4ex]
            $\msf A$ & $\msf B_c-2a_c\msf P_c$ \\
            \addlinespace[.4ex]
            $\msf\Sigma$
            & $\displaystyle 2\int_0^\infty e^{t\msf A}a_c e^{t\msf A^{\msf T}}\dd t$ \\
            \addlinespace[.4ex]
            $\msf J$ & $\msf A+a_c\msf\Sigma^{-1}$ \\
            \addlinespace[.4ex]
            $\msf R(x)$
            & $\bigl(\sig(x)^{\msf T}\sig(x)\bigr)^{-1}\sig(x)^{\msf T}$ \\
            \bottomrule
        \end{tabular}

        \medskip
        \footnotesize\raggedright
        All matrices are $d\times d$, except $\sig_c$ ($d\times r$) and $\msf R(x)$ ($r\times d$).\par
        The dependence of $\msf A$, $\msf\Sigma$, and $\msf J$ on $c$ is suppressed.
    \end{minipage}
\end{table}

\subsection{The \texorpdfstring{$\Gamma$}{Gamma}-liminf}
\label{ssec:unit-lower-bound}

Testing \cref{cor:smooth-lower-eps} with $\vph/\eps$ gives
\begin{equation} \label{eq:unit-smooth-lower}
    \mc I_\eps(\nu)
    \ge \int_{\mc X}
    \left[-\frac1\eps\ms H(x,\nabla\vph(x))-a:\nabla^2\vph(x)\right]\dd\nu,
    \qquad \vph\in C_c^\infty(\mc X).
\end{equation}
The Riccati equation gives $(\msf B_cy)\cdot\msf P_cy=|\sig_c^{\msf T}\msf P_cy|^2$. Thus the quadratic test $-\tfrac12y\cdot\msf P_cy$ has zero Hamiltonian for the linearized coefficients, while its Hessian contributes exactly $\kappa_c$ in \eqref{eq:unit-smooth-lower}. For the original coefficients, even a small negative error in $-\ms H$ could spoil the limit after multiplication by $\eps^{-1}$. We localize this test while preserving the Hamiltonian inequality in the proof below.
\begin{proof}[Proof of \cref{prop:unit-gamma}: $\Gamma$-liminf]
    Let $\nu_\eps\rightharpoonup\nu$. If the lower limit of $\mc I_\eps(\nu_\eps)$ is finite, pass to a subsequence along which it is attained. Along this subsequence $\eps\mc I_\eps(\nu_\eps)\to0$. The scale-$\eps$ lower bound and \cref{cor:J-minus-zeroset} imply $\nu=\sum_{c\in\mc C_0}\omega_c\delta_c$.

    Fix $\eta>0$. We construct $\vph_\eta\in C_c^\infty(\mc X)$ such that
    \[
        -\ms H(x,\nabla\vph_\eta(x))\ge0
        \quad(x\in\mc X),
        \qquad
        -a(c):\nabla^2\vph_\eta(c)\ge\kappa_c-\eta
        \quad(c\in\mc C_0).
    \]

    We first cut off the quadratic functions $-\tfrac12y\cdot\msf P_cy$ near the equilibria. This preserves their Hessians at the equilibria but introduces a small error in the subsolution inequality. We then construct a second function to correct this error.

    Both constructions use the same local expansion. Recall that $\ms H(x,p)=b(x)\cdot p+p\cdot a(x)p$. Since $b(c)=0$, Taylor expansion gives $b(c+y)=\msf B_cy+O(|y|^2)$ and $a(c+y)=a_c+O(|y|)$. Thus, for any symmetric matrix $\msf M$,
    \begin{equation}\label{eq:unit-quadratic-expansion}
        \begin{aligned}
            -\ms H(c+y,\msf My)
            & =-(\msf B_cy)\cdot\msf My-(\msf My)\cdot a_c\msf My+O(|y|^3)\\
            & =-\tfrac12y\cdot
            \bigl(\msf B_c^{\msf T}\msf M+\msf M\msf B_c
            +2\msf Ma_c\msf M\bigr)y+O(|y|^3).
        \end{aligned}
    \end{equation}
    The error is uniform for $c\in\mc C_0$ and $\msf M$ in a bounded set.

    Choose $\chi\in C^\infty(\bb R;[0,1])$ equal to one on $(-\infty,0]$ and zero on $[1,\infty)$. For $0<s<e^{-1}$, define $\xi_s=1$ on $B(0,s^2)$, $\xi_s=0$ outside $B(0,s)$, and
    \[
        \xi_s(y)=\chi\left(\frac{\log|y|-2\log s}{|\log s|}\right)
        \quad(s^2<|y|<s).
    \]
    Differentiating shows that $|y||\nabla\xi_s(y)|\le C/|\log s|$. Take $s$ small enough that the balls $B(c,s)$ are disjoint, and set
    \[
        \vph_s(x)=-\frac12\sum_{c\in\mc C_0}
        \xi_s(x-c)(x-c)\cdot\msf P_c(x-c).
    \]
    For $x=c+y\in B(c,s)$, its gradient is
    \[
        \nabla\vph_s(c+y)
        =-\xi_s(y)\msf P_cy
        -\tfrac12(y\cdot\msf P_cy)\nabla\xi_s(y).
    \]
    Apply \eqref{eq:unit-quadratic-expansion} to the first term, with $\msf M=-\xi_s(y)\msf P_c$. The Riccati equation gives
    \[
        -\ms H(c+y,-\xi_s(y)\msf P_cy)
        =\xi_s(y)\bigl(1-\xi_s(y)\bigr)|\sig_c^{\msf T}\msf P_cy|^2
        +O(|y|^3).
    \]
    The leading term is nonnegative because $0\le\xi_s\le1$. The second term of $\nabla\vph_s(c+y)$ has norm at most $C|y|/|\log s|$. Since $b(c+y)=O(|y|)$ and $\ms H$ is quadratic in its second variable, adding this term changes the Hamiltonian by at most $C|y|^2/|\log s|$. Consequently,
    \[
        -\ms H(c+y,\nabla\vph_s(c+y))
        \ge-C\left(s+\frac1{|\log s|}\right)|y|^2
        \quad(|y|<s).
    \]
    The constant $C$ can be chosen uniformly because $\mc C_0$ is finite. Outside these balls, $\vph_s$ and its gradient vanish. Since $\xi_s=1$ near zero, we also have
    \[
        \nabla^2\vph_s(c)=-\msf P_c,
        \qquad
        -a_c:\nabla^2\vph_s(c)=\tr(a_c\msf P_c)=\kappa_c.
    \]

    We next construct a fixed compactly supported function $v$ for which $-\ms H(x,\nabla v)$ is nonnegative everywhere and bounded below by a positive multiple of $|x-c|^2$ near each equilibrium $c$.
    For each $c\in\mc C_0$, fix a symmetric matrix $\msf S_c$ supplied by \cref{lem:strict-riccati-inequality}, so that
    \begin{equation}\label{eq:unit-strict-riccati}
        \msf B_c^{\msf T}\msf S_c+\msf S_c\msf B_c
        +2\msf S_ca_c\msf S_c<0.
    \end{equation}
    Set $q_c(y)=\tfrac12y\cdot\msf S_cy$, so that $\nabla q_c(y)=\msf S_cy$. Apply \eqref{eq:unit-quadratic-expansion} again, now with $\msf M=\msf S_c$. By \eqref{eq:unit-strict-riccati}, its quadratic term is bounded below by a positive multiple of $|y|^2$. The $O(|y|^3)$ error is smaller than half this bound when $|y|$ is sufficiently small. We can therefore fix disjoint balls about the equilibria and constants $\alpha_c>0$ such that
    \[
        -\ms H(c+y,\nabla q_c(y))\ge\alpha_c|y|^2.
    \]

    Choose smooth cutoffs $\chi_c$ supported in these balls and equal to one on smaller balls about $c$. Define
    \[
        F(x)=\sum_{c\in\mc C_0}\chi_c(x)q_c(x-c).
    \]
    Choose $R$ so large that all the supports of $\chi_c$ lie in $\{V<R\}$, and set $v_0=\Xi_R\circ V$, using \eqref{eq:cutoff}. Since $\Xi_R'(V)=\psi(V/R)$,
    \[
        -\ms H(x,\nabla v_0)
        =\psi(V/R)W
        +\psi(V/R)\bigl(1-\psi(V/R)\bigr)|\sig^{\msf T}\nabla V|^2
        \ge0.
    \]
    On the supports of the cutoffs, $\psi(V/R)=1$, so this expression equals $W$.

    For $t\in(0,1)$, set $v=(1-t)v_0+tF$. Concavity in the gradient variable gives
    \[
        -\ms H(x,\nabla v)
        \ge(1-t)\bigl[-\ms H(x,\nabla v_0)\bigr]
        +t\bigl[-\ms H(x,\nabla F)\bigr].
    \]
    Near $c$, we have $F=q_c(\cdot-c)$, so the right-hand side is at least $t\alpha_c|x-c|^2$. On the compact regions where the cutoffs vary, $W$ has a positive minimum because these regions contain no equilibria. The quantity $-\ms H(x,\nabla F)$ is bounded below there, so choosing $t$ sufficiently small makes the right-hand side nonnegative. Outside the cutoff supports, $F$ and its gradient vanish, and the same inequality gives nonnegativity. Fix such a $t$. Since $\mc C_0$ is finite, we can choose $C_1,s_0>0$ such that
    \[
        -\ms H(x,\nabla v)\ge0
        \quad(x\in\mc X),
        \qquad
        -\ms H(c+y,\nabla v(c+y))\ge C_1|y|^2
        \quad(|y|<2s_0).
    \]
    Both $v_0$ and $F$ are smooth and compactly supported, so $v$ has these properties as well.

    Finally, we combine $\vph_s$ and $v$. First choose $\theta\in(0,1)$ so small that
    \[
        (1-\theta)\kappa_c-\theta\,a_c:\nabla^2v(c)\ge\kappa_c-\eta
        \quad(c\in\mc C_0).
    \]
    This is possible because there are only finitely many equilibria. With $\theta$ fixed, choose $s<s_0$ so small that
    \[
        (1-\theta)C\left(s+\frac1{|\log s|}\right)\le\theta C_1,
    \]
    and set $\vph_\eta=(1-\theta)\vph_s+\theta v$. On each $B(c,s)$, concavity and the preceding estimates give
    \[
        \begin{aligned}
            -\ms H(c+y,\nabla\vph_\eta(c+y))
            &\ge(1-\theta)\bigl[-\ms H(c+y,\nabla\vph_s(c+y))\bigr]
            +\theta\bigl[-\ms H(c+y,\nabla v(c+y))\bigr]\\
            &\ge\left[\theta C_1-(1-\theta)C
            \left(s+\frac1{|\log s|}\right)\right]|y|^2\ge0.
        \end{aligned}
    \]
    Outside these balls, $\nabla\vph_s=0$, so concavity gives $-\ms H(x,\nabla\vph_\eta)\ge\theta[-\ms H(x,\nabla v)]\ge0$. At each equilibrium,
    \[
        -a_c:\nabla^2\vph_\eta(c)
        =(1-\theta)\kappa_c-\theta\,a_c:\nabla^2v(c)
        \ge\kappa_c-\eta.
    \]

    The function $\vph_\eta$ is smooth and compactly supported. Dropping the nonnegative Hamiltonian term in \eqref{eq:unit-smooth-lower} and using Portmanteau's theorem gives
    \[
        \lowlim_{\eps\to0}\mc I_\eps(\nu_\eps)
        \ge\int_{\mc X}-a:\nabla^2\vph_\eta\dd\nu
        \ge\sum_{c\in\mc C_0}\omega_c\kappa_c-\eta.
    \]
    Letting $\eta\downarrow0$ proves the lower bound. In particular, it is infinite when $\nu\notin\mc P(\mc C_0)$.
\end{proof}

\subsection{The \texorpdfstring{$\Gamma$}{Gamma}-limsup}
\label{ssec:unit-upper-bound}

For the linearized diffusion near $c$, a suitable Gaussian measure is invariant under a linear control. We localize this Gaussian and adjust the control to preserve stationarity for the original diffusion, then estimate the resulting cost using \cref{lem:level-2.5-bound}.

\begin{proof}[Proof of \cref{prop:unit-gamma}: $\Gamma$-limsup]
    For recovery, it is enough to consider $\nu=\sum_{c\in\mc C_0}\omega_c\delta_c$.

    For each $c\in\mc C_0$, we construct $\mu_\eps^c\in\mc P(\mc X)$ such that
    \[
        \mu_\eps^c\rightharpoonup\delta_c,
        \qquad
        \uplim_{\eps\downarrow0}\mc I_\eps(\mu_\eps^c)\le\kappa_c.
    \]

    Fix $c\in\mc C_0$ and write $x=c+y$. We first choose a Gaussian density and a stationary current that will remain compatible after truncation. Set $\msf A=\msf B_c-2a_c\msf P_c$, which is Hurwitz by the choice of $\msf P_c$. By \cref{lem:feedback-invariance}, $(\msf A,\sig_c)$ remains controllable, so \cref{lem:Lyapunov} gives a positive definite matrix
    \[
        \msf\Sigma
        \coloneqq2\int_0^\infty e^{t\msf A}a_c e^{t\msf A^{\msf T}}\dd t,
        \qquad
        \msf A\msf\Sigma+\msf\Sigma\msf A^{\msf T}=-2a_c.
    \]
    The centered Gaussian density $g_\eps$ with covariance $\eps\msf\Sigma$ is invariant for the linear diffusion with drift $\msf Ay$ and noise coefficient $\sqrt{2\eps}\,\sig_c$. Since $\nabla g_\eps=-\eps^{-1}\msf\Sigma^{-1}y\,g_\eps$, its probability current is
    \[
        \msf Ay\,g_\eps-\eps a_c\nabla g_\eps
        =(\msf A+a_c\msf\Sigma^{-1})y\,g_\eps.
    \]
    We therefore set $\msf J=\msf A+a_c\msf\Sigma^{-1}$, so that the current is $\msf Jy\,g_\eps$. The Lyapunov equation implies
    \begin{equation} \label{eq:unit-matrix-identities}
        \tr\msf J=0,
        \qquad
        (\msf Jy)\cdot\msf\Sigma^{-1}y=0
        \quad(y\in\bb R^d).
    \end{equation}
    The first identity says that $y\mapsto\msf Jy$ is divergence free, and the second says that it is tangent to the level sets of $s(y)=y\cdot\msf\Sigma^{-1}y$. Multiplying this vector field by any smooth function of $s$ therefore preserves zero divergence. In particular, we can cut off the Gaussian along these ellipsoids while keeping its current stationary.

    Choose a smooth cutoff $\chi$ on $[0,\infty)$ equal to one on $[0,1]$, positive on $[0,2)$, and zero on $[2,\infty)$, such that $|\chi'|^2/\chi$ is bounded on $[0,2)$. Fix $\delta>0$ small enough that the ellipsoid $\{s\le2\delta^2\}$ lies in a Taylor neighborhood of $c$, and define
    \[
        \rho_\eps(c+y)
        =Z_\eps^{-1}e^{-s(y)/(2\eps)}\chi(s(y)/\delta^2),
        \qquad
        Z_\eps=\int_{\bb R^d}e^{-s(y)/(2\eps)}\chi(s(y)/\delta^2)\dd y.
    \]
    We keep $\delta$ fixed as $\eps\to0$: the Gaussian itself concentrates on the scale $\sqrt\eps$, well inside the cutoff region. The change of variables $y=\sqrt\eps\,\msf\Sigma^{1/2}z$ and dominated convergence show that $Z_\eps\sim(2\pi\eps)^{d/2}(\det\msf\Sigma)^{1/2}$, that $\rho_\eps\dd x\rightharpoonup\delta_c$, and that
    \begin{equation} \label{eq:unit-gaussian-moments}
        \frac1\eps\int_{\bb R^d}yy^{\msf T}\rho_\eps(c+y)\dd y\longrightarrow\msf\Sigma,
        \qquad
        \int_{\bb R^d}|y|^k\rho_\eps(c+y)\dd y=O(\eps^{k/2})
        \quad(k\ge0).
    \end{equation}

    We next construct the controls required by \cref{lem:level-2.5-bound} for $\rho_\eps\dd x$. We choose $j_\eps$ so that the controlled drift is $\msf Jy$, and $q_\eps$ to account for the diffusion term. To solve for these controls, we must check that the required corrections lie in the noise range $\mf H$. The definitions of $\msf A$ and $\msf J$ give
    \begin{equation}\label{eq:unit-current-drift}
        \msf B_c-\msf J=a_c(2\msf P_c-\msf\Sigma^{-1}).
    \end{equation}
    Thus $\ran(\msf B_c-\msf J)\subseteq\mf H$. The affine-remainder condition in \cref{ass:algebraic-structure} gives $b(c+y)-\msf B_cy\in\mf H$, hence $b(c+y)-\msf Jy\in\mf H$. Every column of $a$ takes values in the fixed space $\mf H$, so $\Diver a\in\mf H$ as well. Recall that the pseudoinverse $\msf R$ in \cref{tab:unit-matrices} satisfies $\msf R\sig=\msf{Id}_r$ and $\sig\msf Rv=v$ for $v\in\mf H$. It therefore converts these drift corrections into controls. On $\{\rho_\eps>0\}$ define
    \[
        \begin{aligned}
            j_\eps(c+y)&=\msf R(c+y)\bigl[b(c+y)-\msf Jy\bigr],\\
            q_\eps(c+y)&=-\sig(c+y)^{\msf T}\nabla\log\rho_\eps(c+y)
            -\msf R(c+y)\Diver a(c+y),
        \end{aligned}
    \]
    and extend the fields by zero elsewhere. The pseudoinverse identities then give
    \[
        \sig j_\eps\rho_\eps=(b-\msf Jy)\rho_\eps,
        \qquad
        \sig q_\eps\rho_\eps=-\Diver(a\rho_\eps).
    \]
    These identities hold on the zero set as well, because the smooth nonnegative function $\rho_\eps$ and its gradient vanish there. Since $\rho_\eps(c+y)$ depends only on $s(y)$, \eqref{eq:unit-matrix-identities} implies
    \[
        \diver\bigl((b-\sig j_\eps)\rho_\eps\bigr)
        =\diver(\msf Jy\rho_\eps)=0,
        \qquad
        \diver\Diver(a\rho_\eps)=-\diver(\sig q_\eps\rho_\eps).
    \]
    In particular, the ellipsoidal cutoff preserves exact stationarity.

    It remains to estimate $j_\eps+\eps q_\eps$, the combined control in \cref{lem:level-2.5-bound}. On the positive set of the density,
    \[
        \nabla\log\rho_\eps(c+y)
        =-\frac1\eps\msf\Sigma^{-1}y
        +\frac2{\delta^2}\frac{\chi'(s(y)/\delta^2)}{\chi(s(y)/\delta^2)}\msf\Sigma^{-1}y.
    \]
    Define
    \[
        \begin{aligned}
            r_\eps(y)\coloneqq{}&
            \msf R(c+y)\bigl[b(c+y)-\msf Jy\bigr]
            +\sig(c+y)^{\msf T}\msf\Sigma^{-1}y\\
            &-2\sig_c^{\msf T}\msf P_cy
            -\eps\msf R(c+y)\Diver a(c+y).
        \end{aligned}
    \]
    We then have the decomposition and estimate
    \begin{equation} \label{eq:unit-current-expansion}
        \begin{gathered}
            j_\eps(c+y)+\eps q_\eps(c+y)
            =2\sig_c^{\msf T}\msf P_c y+r_\eps(y)+e_\eps(y),\\
            |r_\eps(y)|\le C(|y|^2+\eps),\\
            e_\eps(y)=-\frac{2\eps}{\delta^2}
            \frac{\chi'(s(y)/\delta^2)}{\chi(s(y)/\delta^2)}
            \sig(c+y)^{\msf T}\msf\Sigma^{-1}y.
        \end{gathered}
    \end{equation}
    To verify the bound, the sum of the first two terms in the definition of $r_\eps$ vanishes at $y=0$ and, by $\msf R(c)a_c=\sig_c^{\msf T}$ and \eqref{eq:unit-current-drift}, has derivative
    \[
        \msf R(c)(\msf B_c-\msf J)+\sig_c^{\msf T}\msf\Sigma^{-1}
        =2\sig_c^{\msf T}\msf P_c.
    \]
    Subtracting $2\sig_c^{\msf T}\msf P_cy$ therefore leaves an $O(|y|^2)$ remainder, while the last term is $O(\eps)$ because $\msf R\Diver a$ is bounded on the fixed support. This proves the stated uniform bound. By \eqref{eq:unit-gaussian-moments},
    \[
        \frac1\eps\int_{\bb R^d}|r_\eps(y)|^2\rho_\eps(c+y)\dd y=O(\eps).
    \]
    The other remainder, $e_\eps$, comes only from the cutoff and is supported where $\delta^2<s(y)<2\delta^2$. This fixed annulus has exponentially small Gaussian mass. Since $|\chi'|^2/\chi$ is bounded and $Z_\eps\asymp\eps^{d/2}$,
    \[
        \frac1\eps\int_{\bb R^d}|e_\eps(y)|^2\rho_\eps(c+y)\dd y
        \le C\eps^{1-d/2}e^{-\delta^2/(2\eps)}=o(1).
    \]
    These estimates also prove $j_\eps,q_\eps\in L^2(\rho_\eps\dd x;\bb R^r)$, since $j_\eps$ is bounded on the support. Together with \eqref{eq:unit-current-expansion}, they give
    \[
        j_\eps+\eps q_\eps
        =2\sig_c^{\msf T}\msf P_c(\,\cdot\,-c)
        +o_{L^2(\rho_\eps\dd x)}(\sqrt\eps).
    \]
    By Cauchy--Schwarz, the preceding $L^2$ estimate allows us to replace $j_\eps+\eps q_\eps$ by its leading term in \cref{lem:level-2.5-bound}. Using the covariance limit in \eqref{eq:unit-gaussian-moments}, we obtain
    \begin{align*}
        \mc I_\eps(\rho_\eps\dd x)
        &\le\frac1{4\eps}\int_{\bb R^d}|2\sig_c^{\msf T}\msf P_c y|^2\rho_\eps(c+y)\dd y+o(1)\\
        &=\tr(\msf P_c a_c\msf P_c\msf\Sigma)+o(1).
    \end{align*}
    To identify this trace, use the Riccati and Lyapunov equations:
    \begin{align*}
        2\tr(\msf P_c a_c\msf P_c\msf\Sigma)
        &=-\tr\bigl[(\msf A^{\msf T}\msf P_c+\msf P_c\msf A)\msf\Sigma\bigr]\\
        &=-\tr\bigl[\msf P_c(\msf A\msf\Sigma+\msf\Sigma\msf A^{\msf T})\bigr]
        =2\tr(a_c\msf P_c).
    \end{align*}
    Since $\tr(a_c\msf P_c)=\kappa_c$ by \cref{lem:Riccati}, we conclude that $\uplim_{\eps\to0}\mc I_\eps(\rho_\eps\dd x)\le\kappa_c$. Together with $\rho_\eps\dd x\rightharpoonup\delta_c$, this gives the required family $\mu_\eps^c=\rho_\eps\dd x$.

    Set $\nu_\eps=\sum_{c\in\mc C_0}\omega_c\mu_\eps^c$. Then $\nu_\eps\rightharpoonup\nu$, and convexity gives
    \[
        \uplim_{\eps\to0}\mc I_\eps(\nu_\eps)
        \le\sum_{c\in\mc C_0}\omega_c
        \uplim_{\eps\to0}\mc I_\eps(\mu_\eps^c)
        \le\sum_{c\in\mc C_0}\omega_c\kappa_c
        =\mc J^{(0)}(\nu).
    \]
\end{proof}

\begin{remark}
    The Gaussian construction uses the fixed noise range and affine-remainder condition only in a neighborhood of the chosen equilibrium. Their global versions are needed for the scale-$\eps$ regularization and the analytic framework. Alternative recovery constructions in special settings appear in \cite{LandimLeeMariani2025,LandimMaura2026}.
\end{remark}

\appendix
\crefalias{section}{appendix}
\crefname{appendix}{appendix}{appendices}
\Crefname{appendix}{Appendix}{Appendices}

\section{Finite-state Donsker--Varadhan functionals} \label{app:finite-state}

Let $\ms V$ be a finite set and let
\[
    \mf Lf(x)
    \coloneqq
    \sum_{y\in\ms V}
    \msf q(x,y)\bigl[f(y)-f(x)\bigr],
    \qquad
    \msf q(x,y)\ge0,
\]
be a Markov generator.
\begin{definition} \label{def:finite-communication}
    A state $y$ is \emph{accessible} from $x$ if there is a possibly trivial directed path from $x$ to $y$ along edges with positive jump rate. Two states \emph{communicate} if each is accessible from the other. The corresponding equivalence classes are the \emph{communication classes} of $\mf L$. A communication class $\mf C$ is \emph{closed} if $\msf q(x,y)=0$ whenever $x\in\mf C$ and $y\notin\mf C$. We reserve $\mf C$ for closed communication classes and use $\mf D$ for arbitrary communication classes. The generator is \emph{irreducible} if $\ms V$ is its only communication class.
\end{definition}

\begin{lemma} \label{lem:finite-invariant-measures}
    \cite[Appendix~A]{LandimLeeMariani2025}
    If $\mf C_1,\ldots,\mf C_k$ are the closed communication classes, the restriction of $\mf L$ to each $\mf C_i$ has a unique invariant probability measure $\omega_i$, which is strictly positive there. Writing $\widehat\omega_i$ for its zero extension to $\ms V$, the invariant probability measures of $\mf L$ are precisely
    \[
        \operatorname{co}\{\widehat\omega_1,\ldots,\widehat\omega_k\}.
    \]
\end{lemma}

Define
\begin{equation} \label{eq:finite-DV}
    \mf J_{\mf L}(\omega)
    \coloneqq
    \sup_{h>0}
    \sum_{x\in\ms V}
    -\frac{\mf Lh(x)}{h(x)}\omega(x),
    \qquad
    \omega\in\mc P(\ms V).
\end{equation}

\begin{lemma} \label{lem:finite-DV-finite}
    The functional $\mf J_{\mf L}$ is convex and lower semicontinuous, and
    \[
        0\le\mf J_{\mf L}(\omega)
        \le
        \sum_{x\in\ms V}\omega(x)\sum_{y\ne x}\msf q(x,y).
    \]
\end{lemma}
\begin{proof}
    Convexity and lower semicontinuity follow by taking a supremum of continuous linear functions. The lower bound follows by testing with $h=1$, and the upper bound follows from $1-h(y)/h(x)\le1$ for $x\ne y$.
\end{proof}

\begin{lemma} \label{lem:finite-DV-zero}
    \cite[Lemma~A.6]{LandimLeeMariani2025}
    The functional $\mf J_{\mf L}$ vanishes at $\omega$ if and only if $\omega$ is invariant for $\mf L$.
\end{lemma}

For a communication class $\mf D$, define the reflected generator and escape rate by
\[
    \mf L_{\mf D}f(x)
    \coloneqq
    \sum_{y\in\mf D}
    \msf q(x,y)\bigl[f(y)-f(x)\bigr],
    \qquad
    c_{\mf D}(x)
    \coloneqq
    \sum_{y\notin\mf D}\msf q(x,y),
    \quad x\in\mf D.
\]
Write $\mf J_{\mf D}\coloneqq\mf J_{\mf L_{\mf D}}$ and $\omega(c_{\mf D})\coloneqq\sum_{x\in\mf D}c_{\mf D}(x)\omega(x)$. For $\omega\in\mc P(\mf D)$, let $\widehat\omega$ denote its zero extension to $\ms V$.

\begin{lemma} \label{lem:finite-class-restriction}
    \cite[Lemma~A.7]{LandimLeeMariani2025}
    For every communication class $\mf D$ and every $\omega\in\mc P(\mf D)$,
    \[
        \mf J_{\mf L}(\widehat\omega)
        =
        \mf J_{\mf D}(\omega)+\omega(c_{\mf D}).
    \]
\end{lemma}

\begin{lemma} \label{lem:finite-class-splitting}
    \cite[Lemma~A.8]{LandimLeeMariani2025}
    Let $\mf D_1,\ldots,\mf D_N$ be all communication classes. For $\omega\in\mc P(\ms V)$ and every $i$ with $\omega(\mf D_i)>0$, let $\omega_i$ be the conditional probability measure of $\omega$ on $\mf D_i$. Then
    \[
        \mf J_{\mf L}(\omega)
        =
        \sum_{i:\omega(\mf D_i)>0}
        \omega(\mf D_i)\mf J_{\mf L}(\widehat\omega_i).
    \]
\end{lemma}

\begin{lemma} \label{lem:finite-class-optimizer}
    If $\omega\in\mc P(\mf D)$ is strictly positive, there exists $h:\mf D\to(0,\infty)$ such that
    \[
        \mf J_{\mf D}(\omega)
        =
        \sum_{x\in\mf D}
        -\frac{\mf L_{\mf D}h(x)}{h(x)}\omega(x).
    \]
    More generally, for any $\omega\in\mc P(\mf D)$, a positive function $h$ attains the supremum defining $\mf J_{\mf D}(\omega)$ if and only if $\omega$ is invariant for the generator
    \[
        \mf L_{\mf D}^hf
        \coloneqq
        \frac{\mf L_{\mf D}(hf)}h
        -\frac{f\mf L_{\mf D}h}h.
    \]
    This generator is irreducible; when $h$ is an optimizer, $\omega$ is its unique invariant probability measure.
\end{lemma}
\begin{proof}
    For strictly positive $\omega$, the existence of a positive optimizer follows from \cite[Lemma~A.3]{Landim2023FiniteState}.

    Now let $\omega\in\mc P(\mf D)$ be arbitrary. Writing $h=e^g$, the expression being maximized becomes
    \[
        F(g)\coloneqq
        \sum_{x,y\in\mf D}\omega(x)\msf q(x,y)
        \bigl[1-e^{g(y)-g(x)}\bigr].
    \]
    Since the exponential is convex and $\omega(x)\msf q(x,y)\ge0$, $F$ is concave on $\bb R^{\mf D}$. For every $f:\mf D\to\bb R$, differentiation gives
    \[
        \left.\frac{\dd}{\dd t}\right|_{t=0}
        F(g+tf)
        =-\sum_{x,y\in\mf D}\omega(x)\msf q(x,y)
        \frac{h(y)}{h(x)}\bigl[f(y)-f(x)\bigr]
        =-\sum_{x\in\mf D}\omega(x)\mf L_{\mf D}^hf(x).
    \]
    By concavity, $g$ maximizes $F$ if and only if this derivative vanishes for every $f$. The last expression shows that this is precisely the condition that $\omega$ be invariant for $\mf L_{\mf D}^h$.

    Finally, the jump rate from $x$ to $y$ for $\mf L_{\mf D}^h$ is $\msf q(x,y)h(y)/h(x)$. Since $h>0$, this rate is positive exactly when $\msf q(x,y)>0$. Thus $\mf L_{\mf D}^h$ has the same communication classes as $\mf L_{\mf D}$ and is irreducible. Its invariant probability measure is therefore unique.
\end{proof}

\section{Control theory} \label{app:control-theory}

We recall the finite-dimensional control-theoretic facts used in the paper. A matrix is Hurwitz if all its eigenvalues have strictly negative real parts, and hyperbolic if none of its eigenvalues has zero real part. A pair $(\msf M,\msf G)\in\bb R^{d\times d}\times\bb R^{d\times r}$ is controllable if
\[
    \rank\begin{pmatrix}
        \msf G&\msf M\msf G&\cdots&\msf M^{d-1}\msf G
    \end{pmatrix}
    =d.
\]
Equivalently, for every $T>0$, the Gramian
\begin{equation} \label{eq:control-Gramian}
    \msf Q_T
    \coloneqq
    \int_0^T
    e^{(T-t)\msf M}\msf G\msf G^{\msf T}
    e^{(T-t)\msf M^{\msf T}}\dd t
\end{equation}
is positive definite, and the system $\dot z=\msf Mz+\msf Gu$ can be steered between arbitrary points in time $T$; see \cite[Theorems~3 and~5]{Sontag1998}.

\begin{lemma} \label{lem:feedback-invariance}
    \cite[Lemma/Exercise~5.1.8]{Sontag1998}
    Let $\msf M\in\bb R^{d\times d}$ and $\msf G\in\bb R^{d\times r}$. For every $\msf F\in\bb R^{r\times d}$, the pair $(\msf M,\msf G)$ is controllable if and only if $(\msf M+\msf G\msf F,\msf G)$ is controllable.
\end{lemma}

\begin{lemma} \label{lem:stabilizing-feedback}
    \cite[Theorem~13]{Sontag1998}
    Let $\msf M\in\bb R^{d\times d}$ and $\msf G\in\bb R^{d\times r}$. If $(\msf M,\msf G)$ is controllable, then there exists $\msf K\in\bb R^{r\times d}$ such that $\msf M-\msf G\msf K$ is Hurwitz.
\end{lemma}

\begin{lemma} \label{lem:controlled-connectors}
    Under \cref{ass:algebraic-structure}, for every $T>0$ and $z,w\in\mc X$, there exist $\gamma_{z,w}\in C^\infty([0,T];\mc X)$ and $u_{z,w}\in C^\infty([0,T];\bb R^r)$ such that
    \[
        \dot\gamma_{z,w}
        =
        b(\gamma_{z,w})-\sig(\gamma_{z,w})u_{z,w},
        \qquad
        \gamma_{z,w}(0)=z,
        \qquad
        \gamma_{z,w}(T)=w.
    \]
    For fixed $T$, the path and control may be chosen smoothly in $(z,w,t)$. Consequently, for endpoints in a fixed compact set, the paths remain in a common compact set and their energies are uniformly bounded.
\end{lemma}
\begin{proof}
    Fix $x_*\in\mc X$ such that $(\msf\Omega,\sig(x_*))$ is controllable, and set
    \[
        \sig_*\coloneqq\sig(x_*),
        \qquad
        b_{\mathrm{rem}}(x)\coloneqq b(x)-b(0)-\msf\Omega x.
    \]
    Let $\msf Q_T$ be the positive definite Gramian in \eqref{eq:control-Gramian} for $(\msf M,\msf G)=(\msf\Omega,\sig_*)$, and define
    \[
        \begin{aligned}
            \xi_{z,w}
            &\coloneqq
            e^{T\msf\Omega}z
            +\int_0^T e^{(T-t)\msf\Omega}b(0)\dd t-w,\\
            u_{z,w}^0(t)
            &\coloneqq
            \sig_*^{\msf T}e^{(T-t)\msf\Omega^{\msf T}}
            \msf Q_T^{-1}\xi_{z,w},\\
            \gamma_{z,w}^0(t)
            &\coloneqq
            e^{t\msf\Omega}z
            +\int_0^t e^{(t-s)\msf\Omega}b(0)\dd s
            -\int_0^t e^{(t-s)\msf\Omega}\sig_*u_{z,w}^0(s)\dd s.
        \end{aligned}
    \]
    Then
    \[
        \dot\gamma_{z,w}^0
        =b(0)+\msf\Omega\gamma_{z,w}^0-\sig_*u_{z,w}^0,
        \qquad
        \gamma_{z,w}^0(0)=z,
        \qquad
        \gamma_{z,w}^0(T)=w,
    \]
    where the last identity follows from the definition of $\msf Q_T$.

    Set $\gamma_{z,w}\coloneqq\gamma_{z,w}^0$ and
    \[
        u_{z,w}(t)
        \coloneqq
        \msf R(\gamma_{z,w}(t))
        \left[
            b_{\mathrm{rem}}(\gamma_{z,w}(t))
            +\sig_*u_{z,w}^0(t)
        \right].
    \]
    The vector in brackets belongs to $\mf H$, and $\sig\msf R$ is the identity on $\mf H$. Hence
    \[
        b(\gamma_{z,w})-\sig(\gamma_{z,w})u_{z,w}
        =b(0)+\msf\Omega\gamma_{z,w}-\sig_*u_{z,w}^0
        =\dot\gamma_{z,w}.
    \]
    The formulas depend smoothly on $(z,w,t)$. If the endpoints range over a compact set, the path images lie in a common compact set and the controls are uniformly bounded on the corresponding compact parameter set. Their energies are therefore uniformly bounded.
\end{proof}

\begin{lemma} \label{lem:Lyapunov}
    Let $\msf A\in\bb R^{d\times d}$ be Hurwitz and let $\sig\in\bb R^{d\times r}$ be such that $(\msf A,\sig)$ is controllable. Set $a\coloneqq\sig\sig^{\msf T}$. Then
    \[
        \msf Q
        \coloneqq
        2\int_0^\infty e^{t\msf A}a e^{t\msf A^{\msf T}}\dd t
    \]
    is positive definite and satisfies the Lyapunov equation
    \[
        \msf A\msf Q+\msf Q\msf A^{\msf T}=-2a.
    \]
    Moreover, if $\msf J\coloneqq\msf A+a\msf Q^{-1}$, then
    \[
        \msf J\msf Q+\msf Q\msf J^{\msf T}=0,
        \qquad
        \tr(\msf J)=0,
        \qquad
        (\msf Jy)\cdot\msf Q^{-1}y=0
    \]
    for every $y\in\bb R^d$.
\end{lemma}

\begin{proof}
    Since $\msf A$ is Hurwitz, the integral defining $\msf Q$ converges, and differentiation under the integral gives
    \[
        \msf A\msf Q+\msf Q\msf A^{\msf T}
        =
        2\int_0^\infty
        \frac{\dd}{\dd t}
        \left(e^{t\msf A}a e^{t\msf A^{\msf T}}\right)\dd t
        =-2a.
    \]
    The matrix $\msf Q$ is positive definite because it dominates twice the finite-time Gramian for $(\msf A,\sig)$.

    Finally,
    \[
        \msf J\msf Q+\msf Q\msf J^{\msf T}
        =\msf A\msf Q+\msf Q\msf A^{\msf T}+2a=0.
    \]
    Taking the trace after multiplication by $\msf Q^{-1}$ gives $\tr(\msf J)=0$, while multiplication on the left and right by $\msf Q^{-1}$ shows that $\msf Q^{-1}\msf J$ is skew-symmetric. Hence $(\msf Jy)\cdot\msf Q^{-1}y=0$ for every $y\in\bb R^d$.
\end{proof}

\begin{lemma} \label{lem:strict-riccati-inequality}
    Let $\msf B\in\bb R^{d\times d}$ be hyperbolic and let $a\in\bb R^{d\times d}$ be symmetric and nonnegative. There exists a symmetric matrix $\msf S$ such that
    \[
        \msf B^{\msf T}\msf S+\msf S\msf B+2\msf Sa\msf S<0.
    \]
\end{lemma}
\begin{proof}
    Work in real coordinates adapted to the stable and unstable invariant subspaces of $\msf B$, so that
    \[
        \msf B=\begin{pmatrix}\msf B_s&0\\0&\msf B_u\end{pmatrix},
        \qquad \Ree\spec(\msf B_s)<0<\Ree\spec(\msf B_u).
    \]
    If either subspace is zero, omit the corresponding block. By \cref{lem:Lyapunov}, the positive definite matrices
    \[
        \msf Q_s=\int_0^\infty e^{t\msf B_s^{\msf T}}e^{t\msf B_s}\dd t,
        \qquad
        \msf Q_u=\int_0^\infty e^{-t\msf B_u^{\msf T}}e^{-t\msf B_u}\dd t
    \]
    satisfy
    \[
        \msf B_s^{\msf T}\msf Q_s+\msf Q_s\msf B_s=-\msf{Id},
        \qquad
        \msf B_u^{\msf T}\msf Q_u+\msf Q_u\msf B_u=\msf{Id}.
    \]
    Thus
    \[
        \msf S_0=\begin{pmatrix}\msf Q_s&0\\0&-\msf Q_u\end{pmatrix}
        \qquad\text{satisfies}\qquad
        \msf B^{\msf T}\msf S_0+\msf S_0\msf B=-\msf{Id}.
    \]
    Set $\msf S=\delta\msf S_0$. Then
    \[
        \msf B^{\msf T}\msf S+\msf S\msf B+2\msf Sa\msf S
        =-\delta\msf{Id}+2\delta^2\msf S_0a\msf S_0<0
    \]
    for sufficiently small $\delta>0$.
\end{proof}

\begin{lemma} \label{lem:Riccati}
    Let $\msf B\in\bb R^{d\times d}$ be hyperbolic and let $\sig\in\bb R^{d\times r}$ be such that $(\msf B,\sig)$ is controllable. Set $a=\sig\sig^{\msf T}$. There exists a symmetric nonnegative matrix $\msf P$ that satisfies the Riccati equation
    \[
        \msf B^{\msf T}\msf P+\msf P\msf B-2\msf Pa\msf P=0,
        \qquad
        \msf B-2a\msf P\text{ is Hurwitz}.
    \]
    Moreover,
    \[
        \tr(a\msf P)
        =\sum_{\substack{\lam\in\spec(\msf B)\\\Ree\lam>0}}\Ree\lam,
    \]
    where the eigenvalues are counted with algebraic multiplicity.
\end{lemma}

\begin{proof}
    Work in real coordinates adapted to the stable and unstable invariant subspaces of $\msf B$, so that
    \[
        \msf B=\begin{pmatrix}\msf B_s&0\\0&\msf B_u\end{pmatrix},
        \qquad \Ree\spec(\msf B_s)<0<\Ree\spec(\msf B_u).
    \]
    If the unstable subspace is zero, take $\msf P=0$. Otherwise, omit the stable block if it is absent. Let $\sig_u$ be the unstable block of $\sig$ and set $a_{uu}=\sig_u\sig_u^{\msf T}$. Controllability of $(\msf B,\sig)$ passes to $(\msf B_u,\sig_u)$ and is unchanged when $\msf B_u$ is replaced by $-\msf B_u$. By \cref{lem:Lyapunov}, the positive definite matrix
    \[
        \msf Q_u=2\int_0^\infty
        e^{-t\msf B_u}a_{uu}e^{-t\msf B_u^{\msf T}}\dd t
    \]
    satisfies
    \[
        \msf B_u\msf Q_u+\msf Q_u\msf B_u^{\msf T}=2a_{uu}.
    \]
    Multiplying on the left and right by $\msf Q_u^{-1}$ gives
    \[
        \msf B_u^{\msf T}\msf Q_u^{-1}
        +\msf Q_u^{-1}\msf B_u
        -2\msf Q_u^{-1}a_{uu}\msf Q_u^{-1}=0.
    \]
    Thus the symmetric nonnegative matrix
    \[
        \msf P=\begin{pmatrix}0&0\\0&\msf Q_u^{-1}\end{pmatrix}
        \qquad\text{satisfies}\qquad
        \msf B^{\msf T}\msf P+\msf P\msf B-2\msf Pa\msf P=0.
    \]
    The equation for $\msf Q_u$ also gives
    \[
        \msf B-2a\msf P
        =\begin{pmatrix}
            \msf B_s&-2a_{su}\msf Q_u^{-1}\\
            0&-\msf Q_u\msf B_u^{\msf T}\msf Q_u^{-1}
        \end{pmatrix},
    \]
    where $a_{su}$ is the block of $a$ with stable rows and unstable columns. Both diagonal blocks are Hurwitz: the first is $\msf B_s$, and the second is similar to $-\msf B_u^{\msf T}$. Hence $\msf B-2a\msf P$ is Hurwitz.

    Finally, multiplying the equation for $\msf Q_u$ by $\msf Q_u^{-1}$ and taking the trace gives
    \[
        \tr(a\msf P)=\tr(a_{uu}\msf Q_u^{-1})
        =\tr(\msf B_u).
    \]
    The last trace is the sum of the positive real parts of the eigenvalues of $\msf B$.
\end{proof}

\section{Proofs for the Langevin models}
\label{app:langevin-verification}

For each model in \cref{sec:applications}, we first verify the algebraic and concentration assumptions and identify the stable equilibria. We then apply \cref{thm:main} to prove the corresponding $\Gamma$-expansion. The structural verifications also yield the pre-metastable limits through \cref{prop:pre-metastable-gamma,prop:unit-gamma}, without assuming a metastable hierarchy.

\subsection{Generalized overdamped Langevin dynamics}\label{ssec:overdamped-verification}

\begin{lemma} \label{lem:overdamped-structural}
    Under \textup{\textbf{(OD1)}}--\textup{\textbf{(OD3)}}, the overdamped diffusion satisfies \cref{ass:algebraic-structure,ass:cocentration}. Its equilibria are the critical points of $U$, and its stable equilibria are the local minima of $U$. One may take
    \[
        V=\frac12U,
        \qquad
        W=\frac14|\nabla U|^2.
    \]
\end{lemma}
\begin{proof}
    \emph{Algebraic structure.}
    Since $\sig=\msf{Id}_n$, take $\mf H=\bb R^n$ and $\msf\Omega=0$. The noise has full range, the affine-remainder condition is automatic, and the controllability matrix has rank $n$. Thus \cref{ass:algebraic-structure} holds.

    \emph{Equilibria and stability.}
    Orthogonality gives $b\cdot\nabla U=-|\nabla U|^2$, so every equilibrium is a critical point of $U$. Conversely, differentiating $\ell\cdot\nabla U=0$ at a critical point $c$ gives $\nabla^2U(c)\ell(c)=0$. The Morse property implies $\ell(c)=0$, so $c$ is an equilibrium.

    Write $\msf H_c=\nabla^2U(c)$ and $\msf B_c=Db(c)$. Differentiating the orthogonality identity twice at $c$ gives
    \[
        \msf B_c^{\msf T}\msf H_c+\msf H_c\msf B_c=-2\msf H_c^2.
    \]
    If $\msf B_cz=i\beta z$ for some $z\in\bb C^n\setminus\{0\}$ and $\beta\in\bb R$, multiplication on the left and right by $z^*$ and $z$ gives $0=-2|\msf H_cz|^2<0$, a contradiction. Thus every equilibrium is hyperbolic. If $\msf H_c>0$, the same identity implies that every eigenvalue of $\msf B_c$ has negative real part. Conversely, if $\msf B_c$ is Hurwitz, integration gives
    \[
        \msf H_c
        =
        2\int_0^\infty
        e^{t\msf B_c^{\msf T}}\msf H_c^2e^{t\msf B_c}\,\dd t>0.
    \]
    Hence the stable equilibria are exactly the local minima of $U$.

    \emph{Lyapunov function.}
    For $V=U/2$, orthogonality gives
    \[
        W=-\ms H(\,\cdot\,,\nabla V)=\frac14|\nabla U|^2.
    \]
    The first limit in \textup{\textbf{(OD2)}} implies that $U$ and $|\nabla U|$ are inf-compact, so $V$ and $W$ are inf-compact. Since $\min U=0$, we have $V\ge0$. Moreover, $W\ge0$ and $W^{-1}(0)=\mc C_0$. Outside a sufficiently large compact set, the second limit in \textup{\textbf{(OD2)}} and $|\nabla U|\ge1$ give
    \[
        a:\nabla^2V
        =
        \frac12\Delta U
        \le
        \frac14|\nabla U|
        \le
        \frac14|\nabla U|^2
        =
        W.
    \]
    The left-hand side is bounded above on the remaining compact set. Thus $a:\nabla^2V\le C(1+W)$ on $\bb R^n$, completing the verification of \cref{ass:cocentration}.
\end{proof}

\begin{proof}[Proof of \cref{thm:overdamped-gamma-expansion}]
    By \cref{lem:overdamped-structural}, the diffusion satisfies \cref{ass:algebraic-structure,ass:cocentration}, and its stable equilibria are the local minima of $U$. It remains to verify \cref{ass:complete-metastable-hierarchy}.

    If $U$ has a unique local minimum, the hierarchy is complete with $\mf q=0$ by the convention in \cref{ssec:metastable-hierarchy}. Suppose now that $U$ has several local minima. Radial integration of the first condition in \textup{\textbf{(OD2)}} gives $U(x)/|x|\to\infty$, which verifies the additional growth condition in the following references. Under \textup{\textbf{(OD1)}}--\textup{\textbf{(OD3)}} and \textup{\textbf{(OD-M)}}, the first metastable scale is identified in \cite[Theorem~2.2]{LandimLeeSeo2024First}. The recursive tree is constructed through its terminal level in \cite[Corollary~4.8]{LandimLeeSeo2025General}, while \cite[Theorem~2.14]{LandimLeeSeo2025General} gives the successive-scale resolvent convergence. Taking $\mc E(m)$ to be the valleys defined in \cite[Section~2.5]{LandimLeeSeo2025General}, these results verify \cref{ass:complete-metastable-hierarchy}.

    The full $\Gamma$-expansion now follows from \cref{thm:main}.
\end{proof}

\subsection{Underdamped Langevin dynamics}\label{ssec:underdamped-verification}

Define the physical Hamiltonian
\[
    \mc H(q,p)\coloneqq U(q)+\frac12|p|^2.
\]

\begin{lemma} \label{lem:underdamped-structural}
    Under \textup{\textbf{(UD)}}, the underdamped diffusion satisfies \cref{ass:algebraic-structure,ass:cocentration}. Its equilibria are the points $(c,0)$ with $\nabla U(c)=0$, and its stable equilibria are the points $(m,0)$ with $m$ a local minimum of $U$. One may choose a smooth vector field $v:\bb R^n\to\bb R^n$ such that
    \[
        \begin{aligned}
            V(q,p)&=\frac12\mc H(q,p)+p\cdot v(q)+1,\\
            W(q,p)&=\frac14|p|^2-p^{\msf T}Dv(q)p
            +v(q)\cdot\nabla U(q)-|v(q)|^2.
        \end{aligned}
    \]
\end{lemma}
\begin{proof}
    \emph{Algebraic structure.}
    Take
    \[
        \mf H=\{0\}^n\times\bb R^n,
        \qquad
        \sig=
        \begin{pmatrix}
            0\\
            \msf{Id}_n
        \end{pmatrix},
        \qquad
        \msf\Omega=
        \begin{pmatrix}
            0&\msf{Id}_n\\
            0&0
        \end{pmatrix}.
    \]
    Then $\ran\sig=\mf H$, $\dim\mf H=n=r$, and
    \[
        b(q,p)-b(0,0)-\msf\Omega(q,p)
        =
        \begin{pmatrix}
            0\\
            -\nabla U(q)+\nabla U(0)-p
        \end{pmatrix}
        \in\mf H.
    \]
    Moreover,
    \[
        \msf\Omega\sig=
        \begin{pmatrix}
            \msf{Id}_n\\
            0
        \end{pmatrix},
    \]
    so $[\sig,\msf\Omega\sig]$ has rank $2n$. This proves \cref{ass:algebraic-structure}.

    \emph{Equilibria and stability.}
    The equilibrium set is
    \[
        \mc C_0=\{(c,0):\nabla U(c)=0\}.
    \]
    Write $\msf H_c=\nabla^2U(c)$ and
    \[
        \msf B_c=Db(c,0)
        =\begin{pmatrix}
            0&\msf{Id}_n\\
            -\msf H_c&-\msf{Id}_n
        \end{pmatrix},
        \qquad
        \msf K_c=\operatorname{diag}(\msf H_c,\msf{Id}_n).
    \]
    The linearized energy identity is
    \[
        \msf B_c^{\msf T}\msf K_c+\msf K_c\msf B_c=-2a.
    \]
    If $\msf B_cz=i\beta z$ for some $z\in\bb C^{2n}$ and $\beta\in\bb R$, multiplication by $z^*$ and $z$ shows that the momentum block of $z$ vanishes. The second block equation then gives $\msf H_cz_q=0$, so $z=0$ because $U$ is Morse. Thus every equilibrium is hyperbolic. If $\msf H_c>0$, the same identity places every eigenvalue in the closed left half-plane, and hyperbolicity makes $\msf B_c$ Hurwitz. Conversely, if $\msf B_c$ is Hurwitz, integration gives
    \[
        \msf K_c
        =2\int_0^\infty e^{t\msf B_c^{\msf T}}a e^{t\msf B_c}\,\dd t
        \ge0.
    \]
    Since $\msf H_c$ is nonsingular, it is then positive definite. Hence the stable equilibria correspond exactly to the local minima of $U$.

    \emph{Lyapunov function.}
    The condition \textup{\textbf{(UD)}} implies that $U$ is inf-compact and that its critical set is compact. This set is finite because $U$ is Morse. Set $v_\infty(q)\coloneqq q/\sqrt{1+|q|^2}$. Choose $R>0$ so that all critical points lie in $B(0,R)$ and $v_\infty\cdot\nabla U>0$ on $B(0,R)^c$. Let $\chi\in C_c^\infty(\bb R^n;[0,1])$ equal one on $B(0,R)$ and vanish outside $B(0,2R)$, and set
    \[
        \begin{aligned}
            v_0&\coloneqq\chi\nabla U+(1-\chi)v_\infty,\\
            g&\coloneqq v_0\cdot\nabla U
            =\chi|\nabla U|^2+(1-\chi)v_\infty\cdot\nabla U.
        \end{aligned}
    \]
    The fields $v_0$ and $Dv_0$ are bounded. The function $g$ is nonnegative, vanishes precisely at the critical points, and is inf-compact because $g=v_\infty\cdot\nabla U\to\infty$ at infinity. Near each critical point, $g=|v_0|^2$. On compact sets disjoint from the critical points, $g$ has a positive minimum; at infinity, $g\to\infty$ while $v_0$ remains bounded. Hence
    \[
        g\ge c_v|v_0|^2
    \]
    for some $c_v>0$.

    Choose $\delta>0$ so small that
    \begin{equation} \label{eq:underdamped-delta-choice}
        2\delta^2\|v_0\|_\infty^2\le1,
        \qquad
        \delta\|Dv_0\|_\infty\le\frac18,
        \qquad
        \delta\le\frac{c_v}{2}.
    \end{equation}
    Set $v\coloneqq\delta v_0$ and $V(q,p)\coloneqq\mc H(q,p)/2+p\cdot v(q)+1$. Young's inequality gives
    \[
        V(q,p)
        \ge
        \frac12U(q)+\frac18|p|^2
        +1-2\delta^2\|v_0\|_\infty^2
        \ge
        \frac12U(q)+\frac18|p|^2.
    \]
    Thus $V$ is nonnegative and inf-compact. Direct calculation and \eqref{eq:underdamped-delta-choice} give
    \[
        \begin{aligned}
            W(q,p)
            &=-\ms H\bigl((q,p),\nabla V(q,p)\bigr)\\
            &=\frac14|p|^2-\delta p^{\msf T}Dv_0(q)p
            +\delta g(q)-\delta^2|v_0(q)|^2\\
            &\ge\frac18|p|^2+\frac{\delta}{2}g(q).
        \end{aligned}
    \]
    Therefore $W$ is nonnegative and inf-compact. The lower bound and the formula for $W$ show that $W^{-1}(0)=\{(c,0):\nabla U(c)=0\}=\mc C_0$. Finally,
    \[
        a:\nabla^2V=\Delta_pV=\frac n2,
    \]
    which proves the Lyapunov bound and completes the verification of \cref{ass:cocentration}.
\end{proof}

\begin{proof}[Proof of \cref{thm:underdamped-gamma-expansion}]
    By \cref{lem:underdamped-structural}, the diffusion satisfies \cref{ass:algebraic-structure,ass:cocentration}. Together with the assumed complete metastable hierarchy, these conditions allow us to apply \cref{thm:main}, which gives the full $\Gamma$-expansion.
\end{proof}

\subsection{Third-order Langevin dynamics}\label{ssec:third-order-verification}

\begin{lemma}\label{lem:third-order-structural}
    Under \textup{\textbf{(TD)}}, the third-order diffusion transformed by $T$ satisfies \cref{ass:algebraic-structure,ass:cocentration}. In physical coordinates, its equilibria are the points $(c,0,0)$ with $\nabla U(c)=0$, and its stable equilibria are the points $(m,0,0)$ with $m$ a local minimum of $U$.
\end{lemma}
\begin{proof}
    \emph{Algebraic structure.}
    The change of variables $T$ produces no It\^o correction because $Q_t^\eps$ has finite variation. It gives the drift displayed in \cref{ssec:third-order} and noise coefficient $\widetilde\sig u=(0,0,u)$. Take $\mf H=\{0\}^n\times\{0\}^n\times\bb R^n$ and $\msf\Omega(q,p,\widetilde r)=(p,\widetilde r,0)$. Then $\ran\widetilde\sig=\mf H$ and $\widetilde b(z)-\widetilde b(0)-\msf\Omega z\in\mf H$. The ranges of $\widetilde\sig$, $\msf\Omega\widetilde\sig$, and $\msf\Omega^2\widetilde\sig$ are the third, second, and first coordinate blocks. Thus \cref{ass:algebraic-structure} holds.

    \emph{Equilibria and stability.}
    Solving $b=0$ gives the stated equilibrium set. Write $\msf H_c=\nabla^2U(c)$, $\msf B_c=Db(c,0,0)$, and $\msf K_c=\operatorname{diag}(\msf H_c,\msf{Id}_n,\msf{Id}_n)$. The identities
    \[
        \msf B_c(u,v,w)=(v,-\msf H_cu+w,-v-w),
        \qquad
        \msf B_c^{\msf T}\msf K_c+\msf K_c\msf B_c=-2a
    \]
    imply hyperbolicity: an eigenvector for an imaginary eigenvalue has zero last block by the energy identity, and the third and second block equations then force the other blocks to vanish because $\msf H_c$ is nonsingular. If $\msf H_c>0$, the energy identity places the spectrum in the closed left half-plane, so $\msf B_c$ is Hurwitz. Conversely, if $\msf B_c$ is Hurwitz, then
    \[
        \msf K_c=2\int_0^\infty e^{t\msf B_c^{\msf T}}a e^{t\msf B_c}\,\dd t\ge0,
    \]
    and nonsingularity gives $\msf H_c>0$. Thus the stable equilibria correspond exactly to the local minima of $U$. The map $T$ fixes the equilibria and conjugates their linearized drifts, so these conclusions also hold in transformed coordinates.

    \emph{Energy and auxiliary weight.}
    Work in physical coordinates and set
    \[
        \mc H(q,p,r)\coloneqq U(q)+\frac12\bigl(|p|^2+|r|^2\bigr).
    \]
    Radial integration of \eqref{eq:third-order-radial-coercivity} gives $U(q)\ge c|q|^2-C$, and the same condition confines the critical points to a compact set. Thus $U$ and $\mc H$ are inf-compact, and the critical set is finite because $U$ is Morse. The identities
    \begin{equation} \label{eq:third-order-HJ}
        b\cdot\nabla\mc H=-|r|^2,
        \qquad
        \sig^{\msf T}\nabla\mc H=r
    \end{equation}
    show why the Lyapunov function needs a correction to $\mc H/2$: the energy dissipates only in $r$.

    Choose $R>R_0$ and a cutoff $\chi\in C_c^\infty(\bb R^n;[0,1])$ equal to one on $B(0,R)$ and zero outside $B(0,2R)$. Put
    \[
        v(q)=\chi(q)\nabla U(q)+(1-\chi(q))q,
        \qquad
        g(q)=v(q)\cdot\nabla U(q).
    \]
    Then $Dv$ is bounded, $g\ge0$, and $g$ vanishes exactly at the critical points. Near these points $g=|v|^2=|\nabla U|^2$, whereas outside a compact set $g=q\cdot\nabla U\ge c_0(|q|^2+U)$. Near $\{U=0\}$, Taylor's theorem gives $|v|^2\le CU$. These observations, together with bounds on the remaining compact sets, give
    \begin{equation} \label{eq:third-order-v-bounds}
        |v|^2\le CU,
        \qquad |v|^2\le Cg,
        \qquad 1+U\le C(1+g).
    \end{equation}

    We will add a weighted correction $\rho(\mc H)(v\cdot p+\kappa r\cdot p)$. Its drift derivative contains the mixed term $r\cdot\nabla U$, so the weight must satisfy $\rho(U)|\nabla U|^2\le Cg$. To preserve confinement, we also require $(1+t)\rho(t)\to\infty$. Construct such a weight by defining
    \[
        \eta(t)=\sup_{1+U(q)\ge t}
        \frac{|\nabla U(q)|^2}{(1+U(q))(1+g(q))},
        \qquad t\ge1.
    \]
    This function is finite, nonincreasing, and tends to zero by \eqref{eq:third-order-force-growth} and inf-compactness of $U$. Choose a smooth nondecreasing $\vartheta:\bb R\to[0,1]$ equal to zero on $(-\infty,0]$ and one on $[2,\infty)$, with $\|\vartheta'\|_\infty\le1$. Choose $s_j\uparrow\infty$ with $s_1>0$, $s_{j+1}-s_j\ge4$, and $\eta(e^s)^{-1/2}\ge j+2$ for $s\ge s_j$, with value infinity when $\eta(e^s)=0$. Set
    \[
        k(s)=1+\sum_{j\ge1}\vartheta(s-s_j),
        \qquad
        \rho(t)=\frac{k(\log(1+t))}{1+t}.
    \]
    The sum is locally finite, and the derivative supports are disjoint. Hence $0\le k'\le k$, $|k''|\le Ck$, and $k(s)\to\infty$. Moreover, $k(s)\le C(1+s)$ for $s\ge0$, and $k(s)\le j+1\le\eta(e^s)^{-1/2}$ on $[s_j,s_{j+1})$. It follows that $k(s)\eta(e^s)\to0$ and that $\rho$ is smooth, positive, and nonincreasing on $[0,\infty)$, with
    \begin{equation} \label{eq:third-order-rho-bounds}
    \begin{gathered}
        \rho(t)\longrightarrow0,
        \qquad (1+t)\rho(t)\longrightarrow\infty,
        \qquad (1+t)\rho(t)\eta(1+t)\longrightarrow0,\\
        (1+t)|\rho'(t)|+(1+t)^2|\rho''(t)|\le C\rho(t).
    \end{gathered}
    \end{equation}
    Outside a compact set, $g\ge1$ and
    $\rho(U)|\nabla U|^2\le(1+U)\rho(U)\eta(1+U)(1+g)\le2g$.
    Near the critical points $g=|\nabla U|^2$, and on the remaining compact set $g$ is bounded away from zero. Hence
    \begin{equation} \label{eq:third-order-force-weight}
        \rho(U(q))|\nabla U(q)|^2\le Cg(q),
        \qquad q\in\bb R^n.
    \end{equation}

    \emph{Lyapunov function and dissipation.}
    Since $Dv$ is bounded, choose $\kappa>0$ such that
    \begin{equation} \label{eq:third-order-kappa}
        p^{\msf T}(\kappa\msf{Id}_n-Dv(q))p\ge|p|^2,
        \qquad q,p\in\bb R^n.
    \end{equation}
    Put $\Theta(q,p,r)=v(q)\cdot p+\kappa r\cdot p$. By Young's inequality and \eqref{eq:third-order-v-bounds},
    \begin{equation} \label{eq:third-order-Theta-bound}
        |\Theta|\le\tfrac12|v|^2+C_\kappa(|p|^2+|r|^2)\le C\mc H.
    \end{equation}
    For a sufficiently small $\delta>0$, define
    \[
        V=\frac12\mc H+\delta\rho(\mc H)\Theta+1.
    \]
    Boundedness of $\rho$ gives
    \begin{equation} \label{eq:third-order-V-comparison}
        1+\frac14\mc H\le V\le1+\frac34\mc H,
    \end{equation}
    so $V$ is nonnegative and inf-compact.

    To compute $W=-b\cdot\nabla V-|\sig^{\msf T}\nabla V|^2$, first note that
    \[
        \begin{aligned}
            b\cdot\nabla\Theta
            &=-p^{\msf T}(\kappa\msf{Id}_n-Dv)p-g
              +r\cdot v-\kappa r\cdot\nabla U+\kappa|r|^2-\kappa p\cdot r,\\
            \sig^{\msf T}\nabla V
            &=\frac12r+\delta\bigl[\kappa\rho(\mc H)p+\rho'(\mc H)\Theta r\bigr].
        \end{aligned}
    \]
    In the expression for $W$, the terms linear in $\rho'(\mc H)\Theta$ cancel by \eqref{eq:third-order-HJ}, as do those containing $\kappa\rho(\mc H)p\cdot r$. Thus
    \begin{equation} \label{eq:third-order-W-identity}
    \begin{aligned}
        W={}&\frac14|r|^2
        +\delta\rho(\mc H)\Bigl[
            p^{\msf T}(\kappa\msf{Id}_n-Dv)p+g
            +\kappa r\cdot\nabla U-r\cdot v-\kappa|r|^2
        \Bigr]\\
        &-\delta^2\bigl|\kappa\rho(\mc H)p+\rho'(\mc H)\Theta r\bigr|^2.
    \end{aligned}
    \end{equation}
    Monotonicity of $\rho$, $\mc H\ge U$, and \eqref{eq:third-order-force-weight} give $\rho(\mc H)|\nabla U|^2\le Cg$; also $\rho(\mc H)|v|^2\le Cg$ by \eqref{eq:third-order-v-bounds}. The derivative bounds in \eqref{eq:third-order-rho-bounds} and $|\Theta|\le C\mc H$ imply
    \begin{equation} \label{eq:third-order-rho-derivative-controls}
        |\rho'(\mc H)\Theta|
        +|\rho''(\mc H)|\,|r|^2|\Theta|
        +|\rho'(\mc H)|\,|p|\,|r|
        \le C\rho(\mc H).
    \end{equation}
    Consequently, for every $\zeta>0$,
    \[
        \begin{aligned}
            \delta\rho(\mc H)|r|(\kappa|\nabla U|+|v|)
            &\le\zeta|r|^2+C_\zeta\delta^2\rho(\mc H)g,\\
            \bigl|\kappa\rho(\mc H)p+\rho'(\mc H)\Theta r\bigr|^2
            &\le C\rho(\mc H)|p|^2+C|r|^2.
        \end{aligned}
    \]
    Substitution into \eqref{eq:third-order-W-identity}, with \eqref{eq:third-order-kappa}, gives
    \[
        W\ge(\tfrac14-\zeta-C\delta-C\delta^2)|r|^2
        +\delta\rho(\mc H)\bigl[(1-C\delta)|p|^2+(1-C_\zeta\delta)g\bigr].
    \]
    Choose $\zeta>0$ small and then decrease $\delta$. For some $c_1,c_2>0$,
    \begin{equation} \label{eq:third-order-W-lower}
        W\ge c_1|r|^2+c_2\delta\rho(\mc H)(|p|^2+g).
    \end{equation}

    \emph{Concentration in transformed coordinates.}
    To prove inf-compactness of $W$, note that \eqref{eq:third-order-v-bounds} gives
    \begin{equation} \label{eq:third-order-energy-control}
        1+\mc H\le C(1+|p|^2+|r|^2+g).
    \end{equation}
    By \eqref{eq:third-order-W-lower}, a sublevel set of $W$ bounds both $|r|^2$ and $\rho(\mc H)(|p|^2+g)$. Multiplying \eqref{eq:third-order-energy-control} by the bounded function $\rho(\mc H)$ bounds $(1+\mc H)\rho(\mc H)$, hence $\mc H$ by \eqref{eq:third-order-rho-bounds}. Thus $W$ is inf-compact. Equations \eqref{eq:third-order-W-lower} and \eqref{eq:third-order-W-identity} also give $W\ge0$ and $W^{-1}(0)=\mc C_0$. Finally,
    \[
        a:\nabla^2V=\Delta_rV
        =\frac n2+\delta\bigl[n\rho'(\mc H)\Theta
          +\rho''(\mc H)|r|^2\Theta+2\kappa\rho'(\mc H)p\cdot r\bigr]
    \]
    satisfies $|\Delta_rV|\le C$ by \eqref{eq:third-order-rho-derivative-controls}. Since $W\ge0$, this proves $a:\nabla^2V\le C(1+W)$.

    Set $\widetilde V=V\circ T^{-1}$. The map $T^{-1}$ is a translation in $r$ for fixed $(q,p)$, and the diffusion acts only in $r$. The chain rule therefore gives
    \[
        \widetilde W\circ T=W,
        \qquad
        (\widetilde a:\nabla^2\widetilde V)\circ T=a:\nabla^2V.
    \]
    Since $T$ is a diffeomorphism, it preserves inf-compactness. Together with hyperbolicity proved above, these identities verify \cref{ass:cocentration} for the transformed process.
\end{proof}

\begin{proof}[Proof of \cref{thm:third-order-gamma-expansion}]
    By \cref{lem:third-order-structural}, the transformed process satisfies the hypotheses of \cref{prop:pre-metastable-gamma,prop:unit-gamma}. We first transfer these unconditional conclusions to physical coordinates, then transfer the complete metastable hierarchy to apply \cref{thm:main}.

    Use tildes for objects associated with the transformed process. The map $T_\#$ is a homeomorphism for weak convergence and sends each occupation measure to its transformed counterpart. Thus \cref{thm:LDP-eps} and the contraction principle give the level-2 LDP in physical coordinates, with
    \[
        \widetilde{\mc I}_\eps(T_\#\nu)=\mc I_\eps(\nu).
    \]
    The chain rule also gives
    \[
        (\widetilde{\mc L}_\eps f)\circ T=\mc L_\eps(f\circ T),
        \qquad
        (\widetilde R_\lam^\eps f)\circ T=R_\lam^\eps(f\circ T).
    \]
    Since $f\mapsto f\circ T$ is a bijection on $C_c^\infty$, the first identity and \eqref{eq:def-J-minus} transport $\mc J^{(-1)}$. The map $T$ also fixes the equilibria and conjugates their linearized drifts, so it transports $\mc J^{(0)}$ as well. Thus
    \[
        \widetilde{\mc J}^{(j)}(T_\#\nu)=\mc J^{(j)}(\nu),
        \qquad j\in\{-1,0\}.
    \]
    Invariance of $\Gamma$-convergence under the weak homeomorphism $T_\#$ therefore proves all the pre-metastable conclusions stated in \cref{ssec:third-order}, without a hierarchy assumption.

    Now suppose that the physical process admits a complete metastable hierarchy. Replace each well $\mc E(\mc M)$ by $T(\mc E(\mc M))$. A smooth diffeomorphism preserves compact closures and Lebesgue-null boundaries, and conjugacy of the deterministic flows preserves the basin conditions. The resolvent identity then gives \eqref{eq:metastable-resolvent-convergence} with the same scales and effective generators. Since $T$ fixes the equilibria, it also fixes every lifting measure $\pi_{\mc M}$. Thus the transformed process has the same complete hierarchy, and all its limiting functionals agree with those in physical coordinates under $T_\#$. Applying \cref{thm:main} and transporting its conclusions through this homeomorphism proves the full $\Gamma$-expansion.
\end{proof}

\section*{Acknowledgements}
The authors thank Professor Claudio Landim for his careful reading of the manuscript and his valuable suggestions.

\bibliographystyle{amsplain}
\bibliography{references}

\end{document}